\documentclass{article}
\usepackage{amssymb,amsmath,graphicx,ucs}
\usepackage{amstext}
\usepackage{amscd, amsthm}
\usepackage{amsfonts}
\usepackage{caption}
\usepackage{subcaption}
\usepackage{tikz}
\usepackage{enumitem}
\usepackage[utf8x]{inputenc}
\usepackage{version}
\usepackage{caption}
\usepackage[makeroom]{cancel}

\usepackage[colorlinks=true, linkcolor=blue, citecolor=magenta]{hyperref}
\usepackage[all]{xy}
\newtheorem{theorem}{Theorem}

\newtheorem{proposition}[theorem]{Proposition}
\newtheorem{lemma}[theorem]{Lemma}
\newtheorem{example}[theorem]{Example}

\newtheorem{definition}[theorem]{Definition} 
\newtheorem{remark}[theorem]{Remark} 
 
\newtheorem{corollary}[theorem]{Corollary} 
\newtheorem{convention}[theorem]{Convention}
\newtheorem{conjecture}[theorem]{Conjecture}

\DeclareMathOperator{\rk}{rk}
\DeclareMathOperator{\im}{im}
\DeclareMathOperator{\coker}{coker}

\begin{document}

\title{Cohomology groups for spaces of $12$-fold tilings}

\author{Nicolas Bédaride\footnote{
    Aix Marseille Université, CNRS,
    Centrale Marseille, I2M, UMR 7373, 13453 Marseille, France.
    Email: \texttt{nicolas.bedaride@univ-amu.fr}}  \and
    Franz G{\"a}hler\footnote{Faculty of Mathematics, Bielefeld University, 33615 Bielefeld, Germany.
    Email: \texttt{gaehler@math.uni-bielefeld.de}} \and
    Ana G. Lecuona\footnote{School of Mathematics and Statistics, University of Glasgow, Glasgow, UK.
    Email: \texttt{ana.lecuona@glasgow.ac.uk}}}
    
\date{}
\maketitle
\begin{abstract}
We consider tilings of the plane with $12$-fold symmetry obtained by the cut and projection method. We compute their cohomology groups using the techniques introduced in \cite{Gah.Hun.Kell.13}. To do this we completely describe the window, the orbits of lines under the group action and the orbits of 0-singularities. The complete family of generalized 12-fold tilings can be described
using 2-parameters and it presents a surprisingly rich cohomological structure. To put this finding into perspective, one should compare our results with the cohomology of the generalized $5$-fold tilings (more commonly known as generalized Penrose tilings). In this case the tilings form a 1-parameter family, which fits in simply one of two types of cohomology.
\end{abstract}

%\tableofcontents
%%%%%%%%%%%%%%
%%%%%%%%%%%%%%
\section{Introduction}
%%%%%%%%%%%%%
%%%%%%%%%%%%%%

This paper deals with tilings of the plane $\mathbb R^{2}$. Given a tiling, the group of translations of the plane acts on it, allowing us to associate to each tiling a space, the hull of the tiling, defined as the closure of the orbit of the tiling under the group action. This space, which has been thoroughly studied for the last 20 years, has many interesting properties. We suggest~\cite{Sad.08} for a good introduction.

Cohomology has been an essential tool of algebraic topology. It is a topological invariant that associates groups (or a more complex ring structure) to spaces and can be used, among other things, to tell spaces apart. The computation of the cohomology of the hull of a tiling emerged at the beginning of the $2000$'s. It has been successfully used to obtain dynamical results on tilings, see for example  \cite{Jul.10,Jul.Sad.18} or \cite{Bar.Gamb.14,Schmi.Trev.18}. In this article we will compute the cohomology of some tiling spaces, enhancing our understanding of them and providing the community with some complete calculations. 

There are several classes of tilings with nice properties. Worth mentioning are the class of substitution tilings and the class of cut and project tilings. For the former one, there is a well developed technique that allows to compute the cohomology groups of these tilings \cite{And.Put.98,Gah.Mal.13}. For the latter ones, in spite of the general methods described by G\"ahler, Hunton and Kellendonk to compute these groups \cite{Gah.Hun.Kell.13}, very few examples have been treated in detail. Some examples can be found in \cite{Gah.Kell.00,And.Put.98}. However, most of these examples are also substitution tilings; thus, the theoretical method concerning the cut and project tilings has never been optimized to be applied in actual computations.

In this article we are going to focus on a two parameter family of cut and projection tiling spaces: the generalized 12-fold tilings. The hull of these tilings will be denoted by $\Omega_{E_{12}^{\gamma}}$ with $\gamma\in\mathbb R^{2}$ (the notation and construction method are explained in Sections~\ref{subcohom} and~\ref{subnfold}). Our goal is double: first, we want to obtain a deeper understanding of this space of tilings using cohomology, and second we want to exploit the techniques in \cite{Gah.Hun.Kell.13} to carry out the computations. Our computational efforts crystallize in the following theorem.
\begin{theorem}\label{main}
The cohomology groups of $\Omega_{E_{12}^{\gamma}}$ satisfy:
\begin{enumerate}
\item For all parameters $\gamma$, all the cohomology groups are torsion free. 
\item For all parameters $\gamma$, $H^{0}(\Omega_{E_{12}^{\gamma}})=\mathbb Z$.
\item The rank of the group $H^{1}(\Omega_{E_{12}^{\gamma}})$ depends on $\gamma$ and takes values in the set $\{7, 10, 13, 16, 19, 22, 25\}$. The precise values of $\gamma$ associated to the different ranks are explicit in Figure~\ref{fig-orbit-tout}.
\item The rank of the group $H^{2}(\Omega_{E_{12}^{\gamma}})$ depends on $\gamma$. Moreover, if the rank of $H^{1}(\Omega_{E_{12}^{\gamma}})\leq 13$, then the rank of $H^{2}(\Omega_{E_{12}^{\gamma}})$ is presented in Propositions~\ref{allgamma0} to~\ref{last2orbits}.
\end{enumerate}
\end{theorem}

As mentioned before, the proof of this theorem is based on the techniques described in \cite{Gah.Hun.Kell.13}, which we briefly summarize in the next subsections. A detailed plan of the proof is given in Subsection \ref{practical}. 

In this article, the computation of the groups $H^{1}(\Omega_{E_{12}^{\gamma}})$ is completely explicit; on the other hand, the results concerning the groups $H^{2}(\Omega_{E_{12}^{\gamma}})$ are computer assisted if $\gamma\not\in\mathbb Z[\sqrt 3]\times\mathbb Z[\sqrt 3]$ \cite{FGweb}. We have not explicitly described all the possible ranks of the second cohomology groups as a function of the parameter $\gamma$. However, the information gathered in Proposition~\ref{prop-points-droite}, particularly in Tables~\ref{tab-even1} to~\ref{tab-odd2}, makes it very easy to compute the rank of $H^{2}(\Omega_{E_{12}^{\gamma}})$ for any fixed value of $\gamma\in\mathbb R^{2}$. For a `generic' value of $\gamma$ we propose the following conjecture, supported by computer calculations.
Generic $\gamma$ parameters are those with the maximal number of singular
lines and the least order of multiple intersections of singular lines.
Starting with generic parameters, a multiple intersection of lines
will only move when the parameters are varied, but never split into
several intersections. For special $\gamma$ values, however, several
intersections may merge to a single one of higher-order.

\begin{conjecture}
The maximal cohomology attained among all the generalized 12-fold tilings is
$$
H^0(\Omega_{E_{12}^\gamma})=\mathbb Z,\quad H^1(\Omega_{E_{12}^\gamma})=\mathbb Z^{25}\quad\mathrm{and\ }\quad H^2(\Omega_{E_{12}^\gamma})=\mathbb Z^{564}.
$$
\end{conjecture}

\subsection{Topology of tiling spaces}
Consider a finite set $\mathcal{P}$ of polygons. A \emph{tiling} of the plane based on $\mathcal{P}$ is given by a family of polygons $(T_i)_{i\in I}$ such that 
\begin{itemize}
\item $\bigcup T_i=\mathbb R^2$,
\item $T_i, T_j$ have disjoint interiors for $i\neq j$,
\item each $T_i$ is the image by a translation of an element of $\mathcal{P}$,
\item any two tiles meet full-edge to full-edge.
\end{itemize}
A \emph{patch} of the tiling is a union of some $T_i$. 

Let us consider the set (possibly empty) of all tilings of the plane based on $\mathcal{P}$. Given any two elements $T$ and $T'$ in this set we are going to define a distance, $d(T,T')$, between them (see~\cite{Sad.08} for further details). Let
$R(T,T')$ be the supremum of all radii $r$ such that there exists two translations $t_u, t_v$ of the plane of norm less than $\frac{1}{2r}$ such that the patch of $T-t_u$ intersecting the ball $B(0,r)$ coincides with the patch of $T'-t_v$ in the ball $B(0,r)$. With these notations in place we define
$$
d(T,T'):=\min\{1,R(T,T')\}.
$$ 
We are now ready to introduce the concept of a \emph{tiling space} which is a set of tilings based on $\mathcal{P}$, closed under translation, and complete for the distance metric.
Now, if $T$ is a tiling based on $\mathcal{P}$ we can can construct a tiling space from $T$ by simply considering the closure (with respect to the topology defined by the above distance) of the set of tilings obtained from the action of the translation group of the plane on $T$. This tiling space is called the \emph{hull} of $T$ and denoted $\Omega(T)$. It is well known that $\Omega(T)$ is connected, not path connected and with contractible connected components, see  \cite{Sad.08} for details. 

The key objective of this paper is to compute the cohomology groups of the hulls of a set of tilings obtained by the cut-and-projection method, which we recall in the next section.

\subsection{Background on cut-and-projection tilings cohomology}\label{subcohom}

Given a $2$-plane $E$ in the vector space $\mathbb R^n$ containing no integer line (a line directed by an element of $\mathbb Z^n$) we want to study the tiling $T$ of the plane $E$ obtained by the cut and projection method. We recall this construction: to start with we need to consider the following orthogonal decomposition 
$$\mathbb R^n=E\oplus E^\perp.$$  Let $\pi$ be the projection on $E$ and $\pi^\perp$ the projection on $E^\perp$. We are now ready to describe the tiling $T$ we want to study: Its vertices are the images by $\pi$ of the points in $\mathbb Z^n\cap (E+[0,1]^n)$. The tiles of the tiling are rhombi obtained by the projection $\pi$ of the $2$-faces of the complex $\mathbb Z^n$ which are inside the ``band'' $E+[0,1]^n$.
Let us denote $\Omega_E$ the hull of this tiling. In our case it is the closure of $T$ under the action of the group of translations of the plane $E$. 

Remark that we can also consider an affine plane $\mathcal E=E+\gamma$ for some vector $\gamma$ and associate a tiling to $\mathcal E$ by considering the band $\mathcal E+[0,1]^n$ and doing the same construction. The hull $\Omega_{\mathcal E}$ of this new tiling will be denoted $\Omega_E^\gamma$. Obviously if $\gamma$ is in $\mathbb Z^n$ or in $E$, the hulls of $E$ and $\mathcal E$ coincide. Hence, we will only consider translation vectors $\gamma$ in $E^{\perp}$. In this case, the hulls of $E$ and $\mathcal E$ might differ. This is for example the case of Penrose tilings, where one obtains for a generic value $\gamma$ what is known under the name of a ``generalized Penrose tiling''. 

During the last decade, starting with the work of \cite{And.Put.98} and \cite{Ben.Gamb.03}, a lot of effort has been put into understanding the topology of $\Omega_E$. We are particularly interested in computing the cohomology groups with integer coefficients, $H^k(\Omega_E,\mathbb Z)$ for $k=0,1,2$, of the hull of the tiling. A general description of these cohomology groups is given by G\"ahler-Hunton-Kellendonk in \cite{Gah.Hun.Kell.13}. In order to state their results, we need to introduce quite a bit of notation and establish some conventions, which we will follow throughout the paper. Figure~\ref{Penrose-window} might help the reader visualize some of the terminology.

\begin{figure}
\begin{center}
\includegraphics[width=5cm]{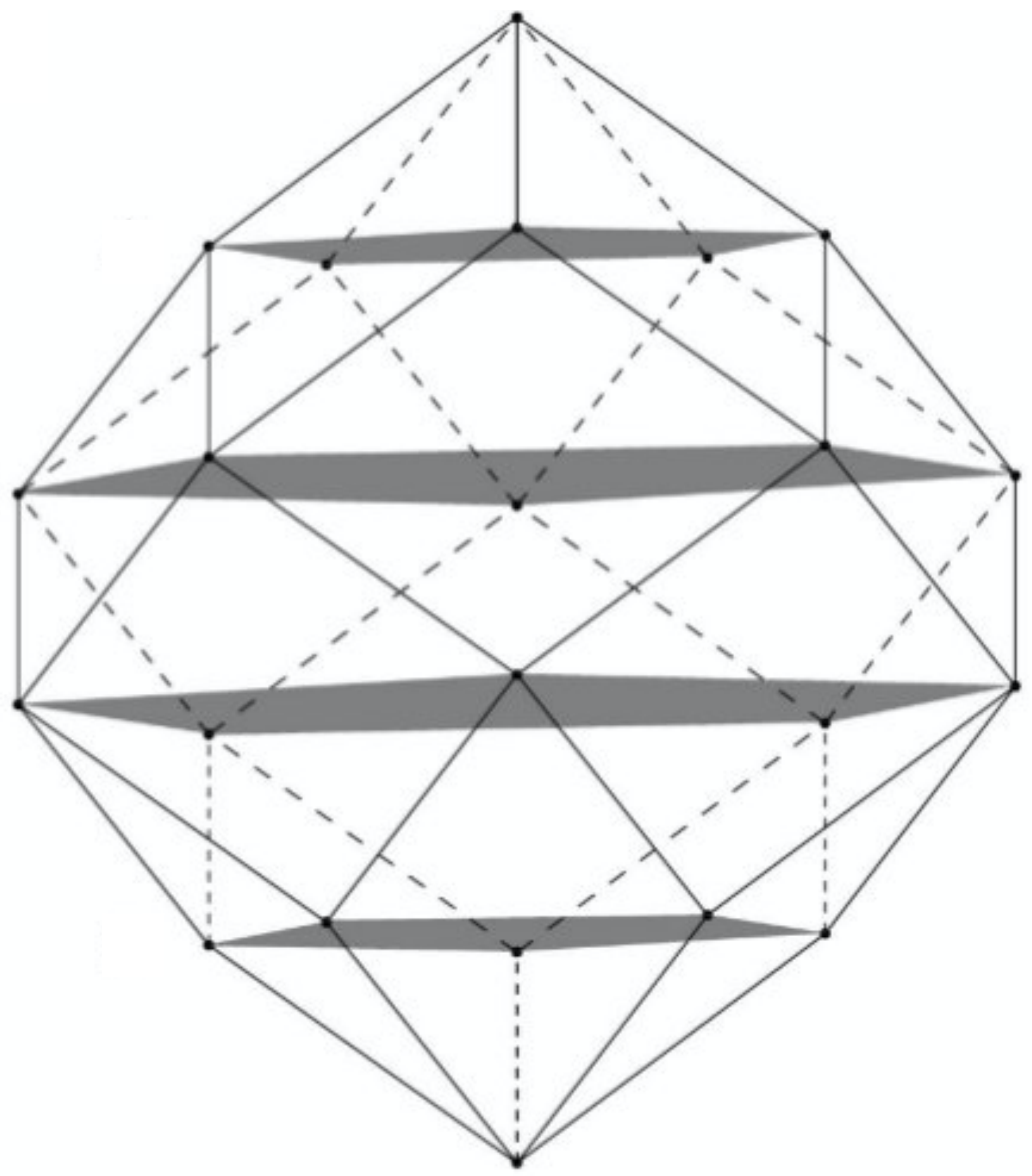}
\end{center}
\caption{The cut and projection description of the Penrose tiling starts with $\mathbb R^{5}$ decomposed into a plane $E$, which will be tiled, and $E^{\perp}\cong\mathbb R^{3}$. In this case, the closure of the group $\Gamma=\pi^{\perp}(\mathbb Z^{5})$ is $\mathbb Z\times\mathbb R^{2}$ and thus $\Delta=\mathbb Z$. The rhombic-icosahedron in the figure is the window, that is, $\pi^{\perp}([0,1]^{5})\subset E^{\perp}$. The $z$-axis, through the bottom and top vertices of the window, is the vector space $F$. The four horizontal shaded pentagons, plus the bottom and top vertices, are the intersection of the family $P$ of planes (parallel to $F^{\perp}$) and the window.}
\label{Penrose-window}
\end{figure}

\begin{itemize}
 
\item The group $\pi^\perp(\mathbb Z^n)$ will be called $\Gamma$. 
  The closure of this group is the product of a free abelian group and a continuous group since it has no torsion part:
  $\overline\Gamma\sim \mathbb Z^k\times \mathbb R^\ell$. We call $\Delta$ the discrete part of $\overline\Gamma$.

\item Denote by $F$ the vector space generated by $\Delta\subset E^\perp$. We will decompose the space $E^\perp$ as $E^\perp=F\oplus F^{\perp}$ and denote by $\Delta_0$ the stabilizer of $F^{\perp}$ under the group $\Gamma$.

\item Assuming that $\ell = \dim F^\perp = 2$, we let $P$ be the collection of planes parallel to $F^{\perp}$, defined as 
  $$
  P=\bigcup_{\delta\in\Delta}F^{\perp}+\delta.
  $$ 

\item The polytope obtained by projecting the cube $[0,1]^n$ onto $E^\perp$ will be called the window and is denoted $W$.
  We need to assume that the intersection between $W$ and $P$ is a union of polygons, segments and points.
  The boundaries of the polygons and the segments are contained in lines directed by the vectors $f_1,\dots, f_n$.

\item We consider the action of $\Gamma$ on the lines described in the last point. The set of orbits is denoted $I_1$
  and each orbit is called a $1$-singularity. The cardinality of the set $I_1$ will be denoted $L_{1}=|I_1|$.

\item The intersection of any two $1$-singularities will be called a $0$-singularity. The set of all orbits of
  $0$-singularities is denoted $I_0$. The cardinality of the set $I_0$ will be denoted $L_{0}=|I_0|$.

\item Let $\Gamma^i\leq\Gamma$ be the stabilizer under the action of $\Gamma$ of the vector space spanned by $f_i$.
  Every $1$-singularity $\alpha\in I_1$ is generated by some direction $f_i$. We denote by $L_0^\alpha$ the set of orbits
  of $0$-singular points on $\alpha$ under the action of $\Gamma^i$, and we set $\Gamma^\alpha:=\Gamma^i$, if the
  $1$-singularity $\alpha$ is in direction $f_i$.

\item Since $\Gamma^\alpha\subset\Delta_0$ for all $\alpha\in I_1$, we can define the natural inclusion map
  \begin{equation}
    \beta: \bigoplus_{\alpha\in I_1}\Lambda^2\Gamma^\alpha \rightarrow \Lambda^2\Delta_0,\label{beta-map}
  \end{equation}
  where $\Lambda^2 A$ denotes the exterior product of two copies of the free abelian group $A$.
  In our case, each group $\Gamma^\alpha$ is of rank two, thus generated by two vectors, and $\Lambda^2\Gamma^\alpha$ represents
  the exterior product of the two generators.

\item We define the numbers $R$ and $e$ as: 
  $$ R = \rk \beta \quad \textrm{and} \quad e=-L_0+\displaystyle\sum_{\alpha\in I_1}L_0^\alpha. $$
  $e$ will turn out to be the Euler characteristic of the tiling space $\Omega_{E}$.  
\end{itemize}

With all this notation and conventions in place, we are now ready to state two general results about the cohomology groups
$H^{*}(\Omega_{E},\mathbb Z)$.

Firstly, we note (compare \cite[Thm.~2.10]{Gah.Hun.Kell.13}) that the cohomology is finitely generated, iff there exists
a natural number $\nu$ such that $\rk \Delta_0 = \nu \dim F^\perp$ (or, equivalently, $\dim E=(\nu-1)\dim F^\perp$), and
$\nu = \rk \Gamma^\alpha$ for all $\alpha\in I_1$. In our case, these conditions are met for $\nu=2$, and the cohomologies
of all dodecagonal tilings considered here are finitely generated.

Secondly, with the notation introduced above, the cohomology groups can now be expressed explicitly.

\begin{theorem}[Thm.~5.3 in~\cite{Gah.Hun.Kell.13}]\label{ref-rank-cohomologie}
The free abelian part of $H^{k}(\Omega_E,\mathbb Z)$ is given by the following isomorphisms:
{\renewcommand\arraystretch{1.3}
\[ \begin{array}{|c|c|c|}
\hline
H^0(\Omega_E,\mathbb Z) & H^1(\Omega_E,\mathbb Z) & H^2(\Omega_E,\mathbb Z)\\ 
\hline
\mathbb Z & \mathbb Z^{4+L_1-R} & \mathbb Z^{3+L_1+e-R}\\
\hline
\end{array}\]}
The torsion part of $H^2(\Omega_E,\mathbb Z)$ is isomorphic to the torsion in $\coker \beta:= \Lambda^2\Delta_0 / \langle \Lambda^2\Gamma^\alpha \rangle_{\alpha \in I_1}$,
whereas the other cohomology groups are torsion free.
\end{theorem}

We note that $R = \rk \beta = \rk \langle \Lambda^2\Gamma^\alpha \rangle_{\alpha \in I_1}$ depends
only on the directions of the $1$-singularities $\alpha$, not on their positions, and the same
holds true for $\coker \beta$.
Hence, these quantities are the same for all dodecagonal tilings considered here.
As we shall see, we have $R=3$, and the torsion of $\coker \beta$ vanishes, so that the cohomology
of all dodecagonal tilings is free.

%%%%%%%%%%%%%%%%%%%%
\subsection{Background on $n$-fold tilings}\label{subnfold}
%%%%%%%%%%%%%%%%%%%%%

In this section we recall the definition of the $n$-fold tiling and summarize some known results on its cohomology groups. We refer the reader to \cite{Bed.fern.15} for further details on these tilings. Following the description in the previous section, to produce a tiling by the cut and projection method we need to start with a 2-plane in a vector space. In this case the plane we intend to tile, $E_{n}$, is defined as follows:
\begin{itemize}
\item If $n=2p+1$, then $E_{n}$ is the 2-plane in $\mathbb R^{2p+1}$ generated by
$$
\left(
\begin{array}{c}
1\\
\cos \frac{2\pi}{2p+1}\\
\vdots\\
\cos\frac{4p\pi}{2p+1}
\end{array}
\right)\ \mathrm{and}\ 
\left(
\begin{array}{c}
0\\
\sin \frac{2\pi}{2p+1}\\
\vdots\\
\sin\frac{4p\pi}{2p+1}
\end{array}
\right).
$$
\item If $n=2p$, then $E_{n}$ is the 2-plane in $\mathbb R^{p}$ generated by
$$
\left(
\begin{array}{c}
1\\
\cos \frac{\pi}{p}\\
\vdots\\
\cos\frac{(p-1)\pi}{p}
\end{array}
\right)\ \mathrm{and}\ 
\left(
\begin{array}{c}
0\\
\sin \frac{\pi}{p}\\
\vdots\\
\sin\frac{(p-1)\pi}{p}
\end{array}
\right).
$$
\end{itemize} 

The plane $E_{n}$ yields a decomposition of the ambient space as
$\mathbb R^{k}=E_{n}\oplus E_{n}^{\perp}$, where the value $k$ depends on the parity of $n$. Furthermore, the space $E_{n}^{\perp}$ decomposes as $E_{n}^{\perp}=F_{n}\oplus F_{n}^{\perp}$, where $F_{n}$ is the vector space spanned by $\Delta$ (recall that $\Delta$ is the discrete part of $\overline{\pi^{\perp}(\mathbb Z^{k})}\subset E_{n}^{\perp}$).
As explained in Section~\ref{subcohom}, starting with $E_{n}^{\perp}$ we can define a family of ``related'' tilings by considering the translate planes $E_{n}+\gamma$, $\gamma\in E_{n}^{\perp}$. 

If $\gamma\in F_n^\perp$, then the hulls of the tilings of $E_{n}$ and $E_{n}+\gamma$ coincide. 
Indeed, translating $E_n$ by $\gamma$ and looking at integer points in the band is the same as keeping $E_n$ and looking at points of the form $-\gamma+\mathbb Z^n$ in the band. Thus if $\gamma,\gamma'\in F_{n}$ and $\gamma-\gamma'\in\Delta$, then the tilings of $E_{n}+\gamma$ and $E_{n}+\gamma'$ coincide, since $\gamma-\gamma'$ is the projection of an element of $\mathbb Z^n$. 
Hence, we can restrict our study to the case $\gamma\in F_n$.

Our aim is to compute the cohomology groups of the hulls $\Omega_{E_n^{\gamma}}$ with the techniques introduced in \cite{Gah.Hun.Kell.13}. As explained above, only some values of $\gamma$ yield interesting sets $\Omega_{E_n^{\gamma}}$ (i.e.\ different from $\Omega_{E_n}$). From~\cite{Gah.Hun.Kell.13} we know that to compute the cohomology groups of $\Omega_{E_n^{\gamma}}$ we first need to understand the intersection of window $W$ (defined in Section~\ref{subcohom}) with the the family of planes $P^\gamma$ of the form
 
$$
P^{\gamma}=\bigcup_{\delta\in\Delta} F_{n}^{\perp}+\delta+\gamma,\ \mathrm{with}\ \gamma\in F_{n}.
$$

The following statement summarizes all the known values of the cohomology
groups of the (undecorated) $n$-fold tilings (cf.\ \cite[Table~5.1]{Gah.Hun.Kell.13}).
These results had mostly been announced without proof already in \cite{Gah.Kell.00},
but that paper contains some mistakes in the cohomology of the generalized Penrose
tilings, which were corrected later by Kalugin \cite{Kal.05}.

\begin{theorem}\label{known}
For small values of $n$, the $n$-fold tilings satisfy:
\begin{enumerate}
\item If $n=8$, these tilings are also known as Ammann--Beenker tilings. In this case $F=0$. The cohomology groups with integer coefficients of the $8$-fold tilings are:
\[{\renewcommand\arraystretch{1.3}
\begin{array}{|c|c|c|}
\hline
H^0(\Omega_{E_{8}})&H^1(\Omega_{E_{8}})&H^2(\Omega_{E_{8}})\\
\hline
\mathbb Z&\mathbb Z^5&\mathbb Z^9\\
\hline
\end{array}
}\]
\item If $n=5$, these tilings are also known as (Generalized) Penrose tilings. In this case $F$ is of dimension one and we can understand the translation parameter $\gamma$ as a real number. The cohomology groups with integer coefficients of the $5$-fold tilings depend on whether $\gamma$ is in $\mathbb Z[\varphi]$ or not, where $\varphi$ is the golden mean. The groups are as follows:
\[{\renewcommand\arraystretch{1.3}
\begin{array}{|c|c|c|c|}
\hline
\gamma&H^0(\Omega_{E_5^{\gamma}})&H^1(\Omega_{E_5^{\gamma}})&H^2(\Omega_{E_5^{\gamma}})\\
\hline
\mathbb Z[\varphi]&\mathbb Z&\mathbb Z^5&\mathbb Z^8\\
\hline
\mathbb R\setminus \mathbb Z[\varphi]&\mathbb Z&\mathbb Z^{10}&\mathbb Z^{34}\\
\hline
\end{array}
}\]

\item If $n=12$, then $F$ is of dimension two. For all $\gamma\in\Delta$, it holds $\Omega_{E_{12}^{\gamma}}=\Omega_{E_{12}}$. In this case the cohomology groups with integer coefficients of the $12$-fold tilings are:
\[{\renewcommand\arraystretch{1.3}
\begin{array}{|c|c|c|}
\hline
H^0(\Omega_{E_{12}})&H^1(\Omega_{E_{12}})&H^2(\Omega_{E_{12}})\\
\hline
\mathbb Z&\mathbb Z^7&\mathbb Z^{28}\\
\hline
\end{array}
}\]
\item If $n=7$, the cohomology groups with rational coefficients of the $7$-fold tilings are not finite-dimensional (compare \cite[Section~5]{FHK.02}).
\end{enumerate}
\end{theorem}

%%%%%%%
\subsection{Generalized Penrose tilings vs.\ 12-fold tilings}
%%%%%%%

To give some perspective to the results in this paper, we will make a fast comparison between the computation of the cohomology groups of the generalized Penrose tilings and the 12-fold tilings. 

In the Penrose case, we start from the space $\mathbb R^5$ and we consider the cut and project method onto a two dimensional plane. The main character in the construction, the polytope $W$, is in this case a 3-dimensional rhombic-icosahedron with vertices in five parallel planes depicted in Figure~\ref{Penrose-window}. The intersection of these planes and $W$ is a series of points and polygons. The boundaries of the polygons are contained in lines directed by $5$ different vectors. To compute the cohomology of the standard Penrose tiling, we need to study the orbits of these lines under the action of a certain group $\Gamma$.

In the classic $12$-fold tiling, we obtain a $4$-dimensional polytope $W$, described in Proposition~\ref{Prop-W}. Its vertices are contained in 16 parallel planes. This time the boundaries of the polygons determined by these vertices (the analogues of the shaded pentagons in Figure~\ref{Penrose-window}) are contained in lines directed by $6$ different vectors. 

The main difference between these two settings though is not simply that there is one more direction in the 12-fold tiling. The main difference is apparent when we consider the \emph{generalized} tilings. Indeed, to compute the cohomology of the generalized tilings we need to shift the planes cutting the polytope. In the Penrose case, the shifting occurs along the vertical axis of the rhombic-icosahedron (see Figure~\ref{Penrose-window}), and it can be encoded by the real number $\gamma$ in the statement of Theorem~\ref{known}.2. When we shift slightly the planes cutting the Penrose rhombic-icosahedron we obtain some polygons with 10 sides (contained still in 5 directions). Now, depending on whether or not
the translation parameter $\gamma$ is in $\mathbb Z[\varphi]$ (where $\varphi$ is the golden mean) the parallel opposite lines on the decagon will belong or not to the same orbit under the action of $\Gamma$.  
 
In the generalized 12-fold tiling, the situation is much richer. The shifting of the planes that cut the polytope is encoded by a 2-dimensional parameter $\gamma=(\gamma_{1},\gamma_{2})$. In Figure~\ref{fig:decalage-axe} one can see the number of sides that a polygon cut by the shifted planes will have. The 10 lines obtained in the generalized Penrose become 24 in the generalized 12-fold tiling (see Proposition~\ref{prop-intersections-cubes}). Each one of the original 6 directions can have at most 4 representatives in different orbits. The game is now to understand when, depending on the value of $\gamma$, these representatives are in the same orbit under the action of $\Gamma$. The result, surprisingly complicated, is the content of Proposition~\ref{orbits-line} and Figure~\ref{fig-orbit-tout}. This information is enough to compute completely the first cohomology group of the 12-fold generalized tilings. When trying to compute the second cohomology groups, the intricacy grows further as explored in Section~\ref{sec-lines}.

%%%%%%%%%%%%%%%
\subsection{Organization of the paper}\label{practical}
%%%%%%%%%%%%%%%%
The proof of the main theorem of this paper is a long and involved computation of the quantities $e,L_{1}$ and $R$ which will allow us, via Theorem~\ref{ref-rank-cohomologie}, to compute the cohomology groups of the 12--fold tilings (as discussed after that theorem, these groups are all free). Roughly, the steps we will follow are:

\begin{enumerate}
\item We intersect $W$ with $P^\gamma$: we obtain several polygons. We compute the equations of the lines which support the edges of the polygons. The case $\gamma\in\Delta$ is presented in Section~\ref{delta=0} while the general case can be found in Section \ref{sec-cubes-droites}.
\item For each line $\alpha$ obtained in the preceding point, we consider its orbit under the action of $\Gamma$. The number of orbits is the quantity $L_1$. The relevant computations can be found in Section~\ref{sect:lines-gamma} and culminate with Proposition~\ref{orbits-line} and Figure~\ref{fig-orbit-tout}.

The lines $\alpha$ are in one of 6 possible directions. The smallest value of $L_{1}$ is then six, with just one $\Gamma$-orbit per direction; and the largest is $L_{1}=24$, corresponding to the case of four $\Gamma$-orbits per direction.
\item We compute the stabilizer $\Gamma^\alpha$ of each line $\alpha$ and obtain the quantity $R$, which is the rank of $\langle \Lambda^2\Gamma^\alpha \rangle_{\alpha \in I_1}$. This computation takes place in Section~\ref{sec-first-group} and allows us to compute the first cohomology groups for the 12-fold tiling in Proposition~\ref{rank}.
\item We count the number of intersections, up to the action, between a 1-singularity $\alpha$ and the other 1-singularities. This quantity is $L_0^\alpha$. We have not completed the computation of $L_{0}^{\alpha}$ for all the possible values of the parameter $\gamma$. The study breaks down into too many subcases which did not seem worth looking into in full detail. However: 
	\begin{itemize}
	\item In Section~\ref{gammanul} we present a complete calculation of the second cohomology group when $\gamma=0$. In this case we have actually computed the values $L_{0}$ and $L_{0}^{\alpha}$ (Lemma~\ref{lem-calc-droites-gamma-nul}), which is all that is needed to determine the cohomology groups of $\Omega_{E_{12}}$. These groups were known (Theorem~\ref{known}.3) and are restated in Proposition~\ref{allgamma0}.
	
	The case $\gamma=0$ is the simplest analyzed and corresponds to the case of $L_{1}=6$ described above.
	\item In Section~\ref{2lines} we have further computed the second cohomology groups, and thus the full cohomology, for all the cases in which under the action of $\Gamma$ there are at most two representatives per line, that is, the cases $L_{1}\leq 12$. This computation is computer assisted and we collect the results in Propositions~\ref{allgamma0} to~\ref{last2orbits}.	
	\item Finally, in Section~\ref{general1} we analyze the general case and summarize the results in Section~\ref{generalfinal}. We consider the largest set of 1-singularities, which has 4 orbits per direction and thus 24 1-singularities. We intersect each of these with the translates by the action of $\Gamma$ of the 20 non parallel 1-singularities. This yields a complete list of 0-singularities in each 1-singularity. Our results are collected in Proposition~\ref{prop-points-droite}, which provides an upper bound on the quantity $L_{0}^{\alpha}$. The complete lists of 0-singularities are displayed in Tables~\ref{tab-even1} to~\ref{tab-odd2}. 
	\end{itemize}
\item In the last section of the paper, we put our results into a
broader perspective. We briefly discuss sets of $\gamma$-parameters
which yield homeomorphic tiling spaces, and we show that projection
tilings with rational $\gamma$-parameters are also substitutive, from
which further conclusions with interesting applications to dynamical
systems are drawn. Although rational $\gamma$-parameters form only a
null set within all $\gamma$-parameters, they yield the most
interesting examples. In particular, all tilings considered in
Section~\ref{2lines} fall into this class.

\end{enumerate}

%%%%%%%%%%%%%%%
%%%%%%%%%%%%%%%
\section{The $12$-fold tilings}\label{conventions}
%%%%%%%%%%%%%%%
%%%%%%%%%%%%%%%%

The $12$-fold tiling space corresponding to the plane $E_{12}^\gamma$ with $\gamma\in\Delta$ was introduced by Socolar in \cite{Soc.89}, where it was shown to be also a substitution tiling. Thus, the classical method exposed in \cite{Sad.08} allows us to compute its cohomology groups. In fact, many more projection tilings in the family $E_{12}^\gamma$ are substitutive (compare Section~\ref{sec-disc}), but they form only a null set among all projection tilings, and except for a few simple cases, their substitution description gets very complicated.
Our goal here is to extend this computation to all parameters $\gamma\in F_{12}$. (While it is not directly relevant to the discussion in this paper, the interested reader might want to take a look at Figure~\ref{fig:patches} for some patches of $E_{12}^{\gamma}$ tilings for different values of $\gamma$.)
\begin{figure*}
        \centering
        \begin{subfigure}{0.485\textwidth}
            \centering
            \includegraphics[width=\textwidth]{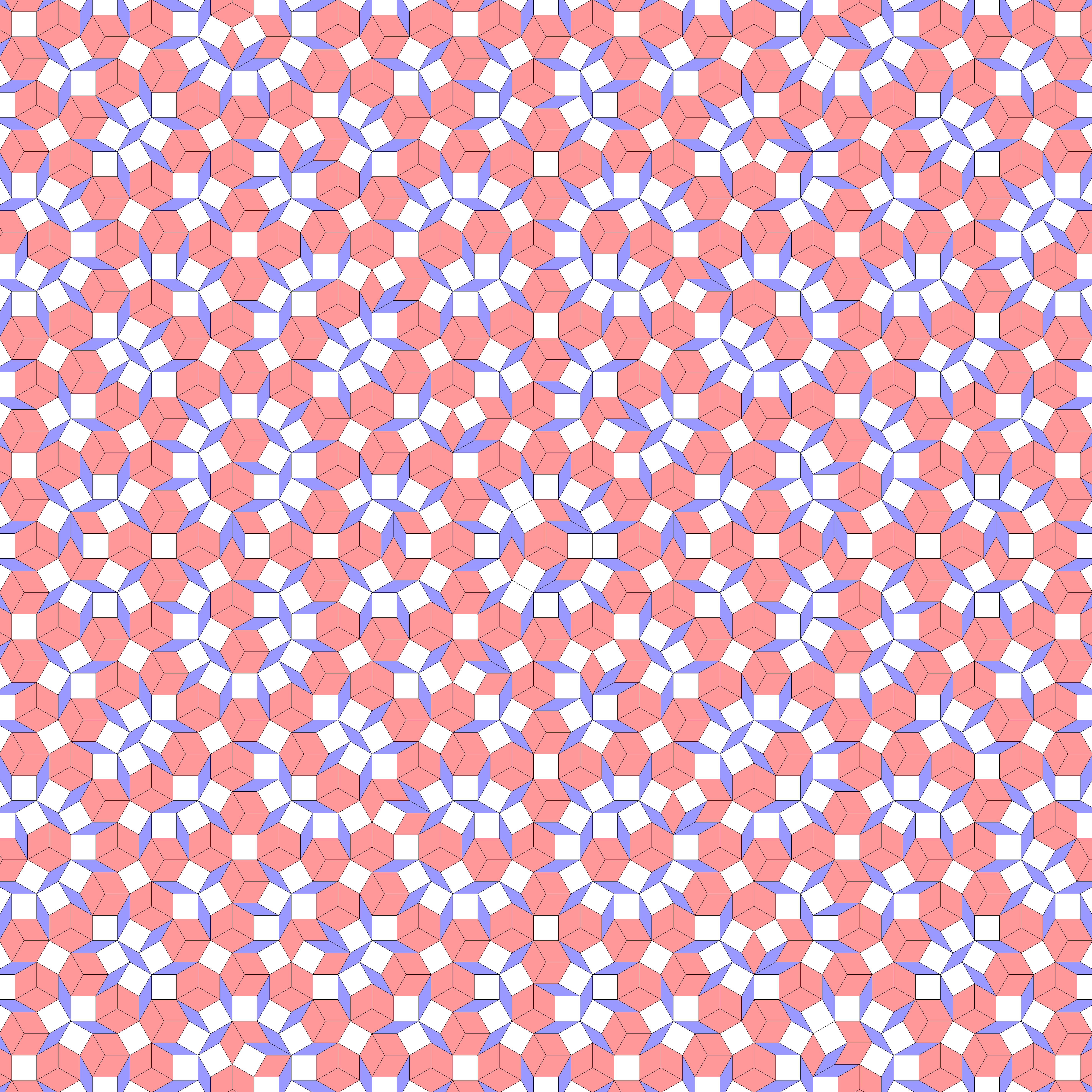}
            \caption[]%
            {{\small $\gamma=(0,0)$}}      
        \end{subfigure}
        \hfill
        \begin{subfigure}{0.485\textwidth}  
            \centering 
            \includegraphics[width=\textwidth]{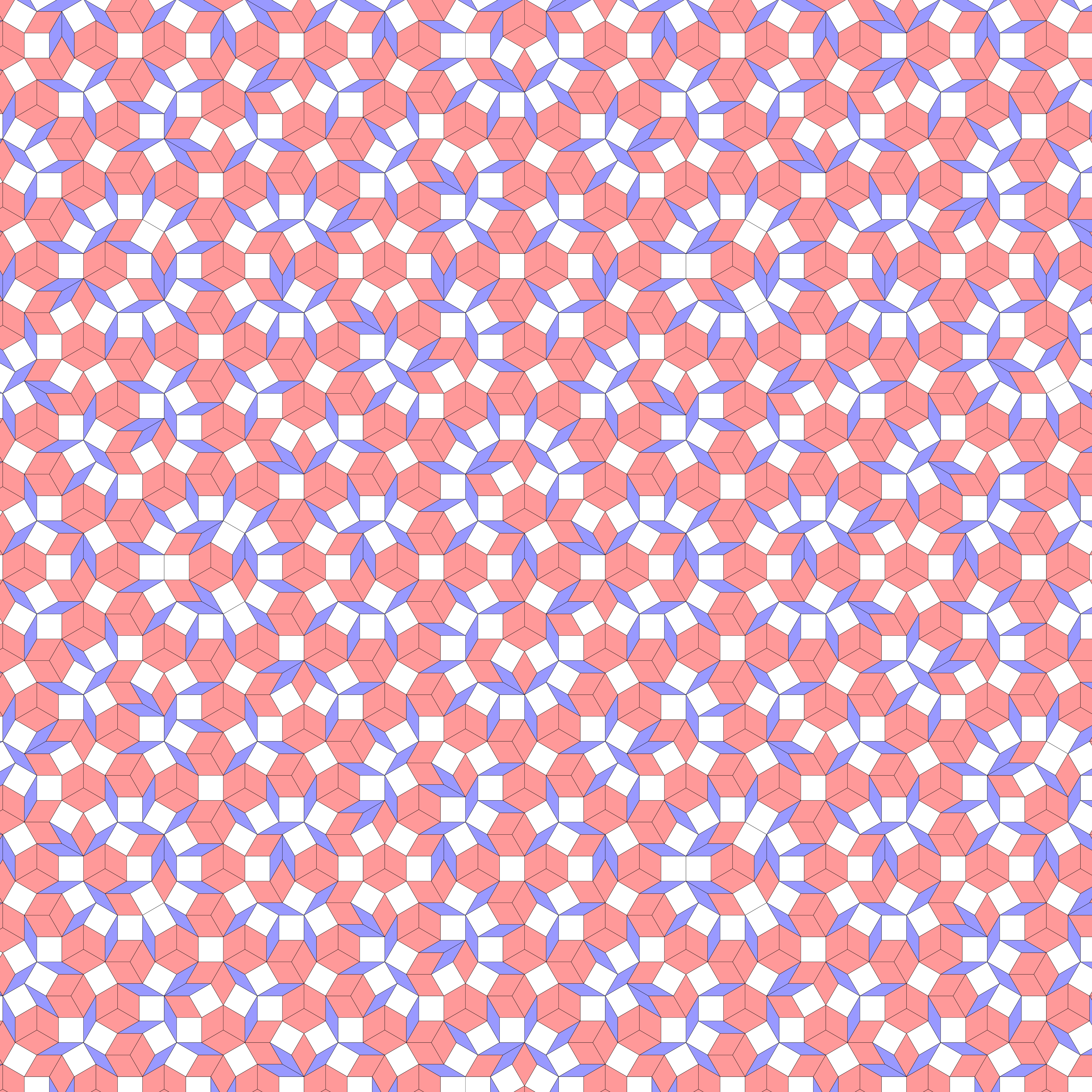}
            \caption[]%
            {{\small $\gamma=(\sqrt 3/2,0)$}}    
        \end{subfigure}
        \vskip\baselineskip
        \begin{subfigure}{0.485\textwidth}   
            \centering 
            \includegraphics[width=\textwidth]{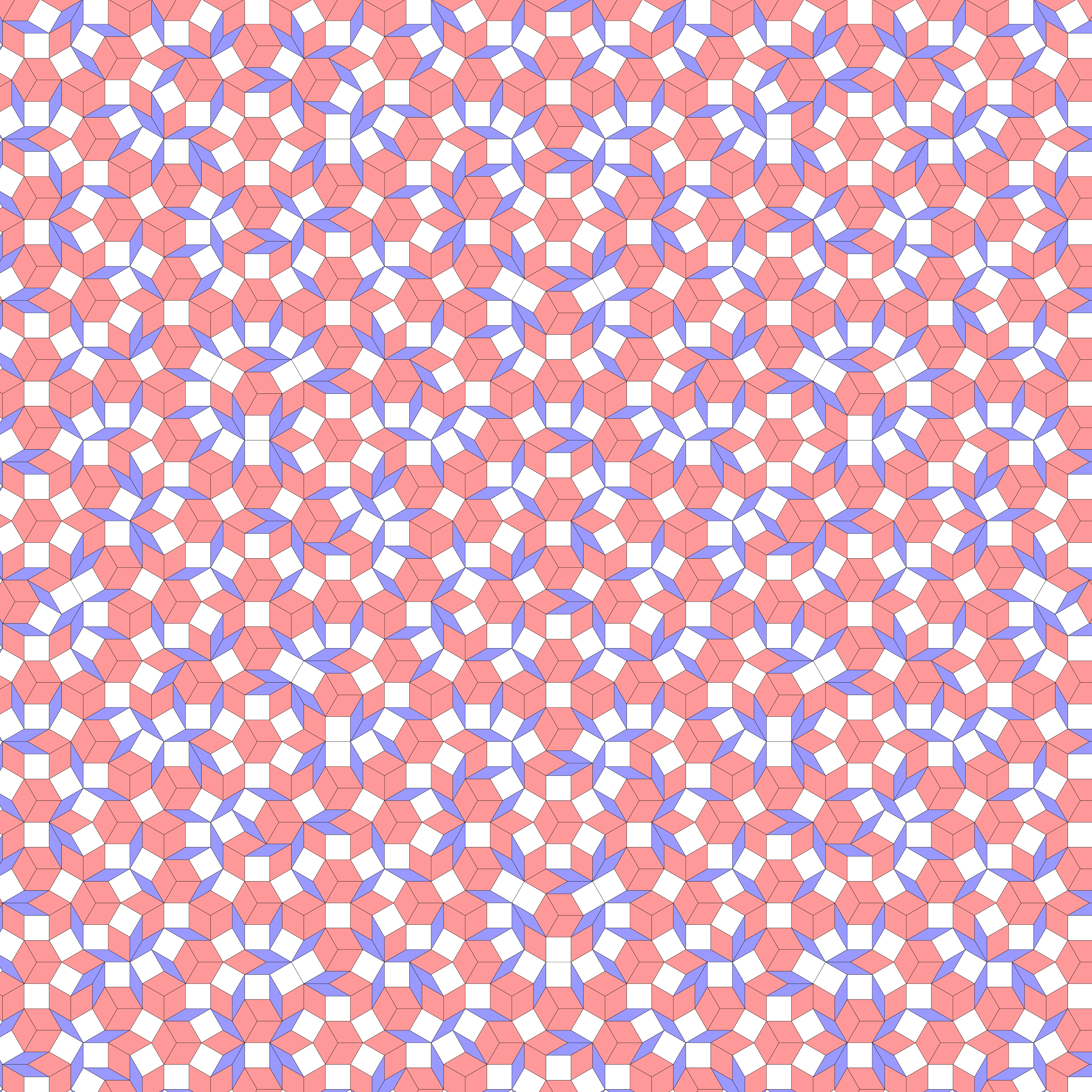}
            \caption[]%
            {{\small $\gamma=(0,\frac{1+\sqrt 3}{2\sqrt 3})$}}    
        \end{subfigure}
        \hfill
        \begin{subfigure}{0.485\textwidth}   
            \centering 
            \includegraphics[width=\textwidth]{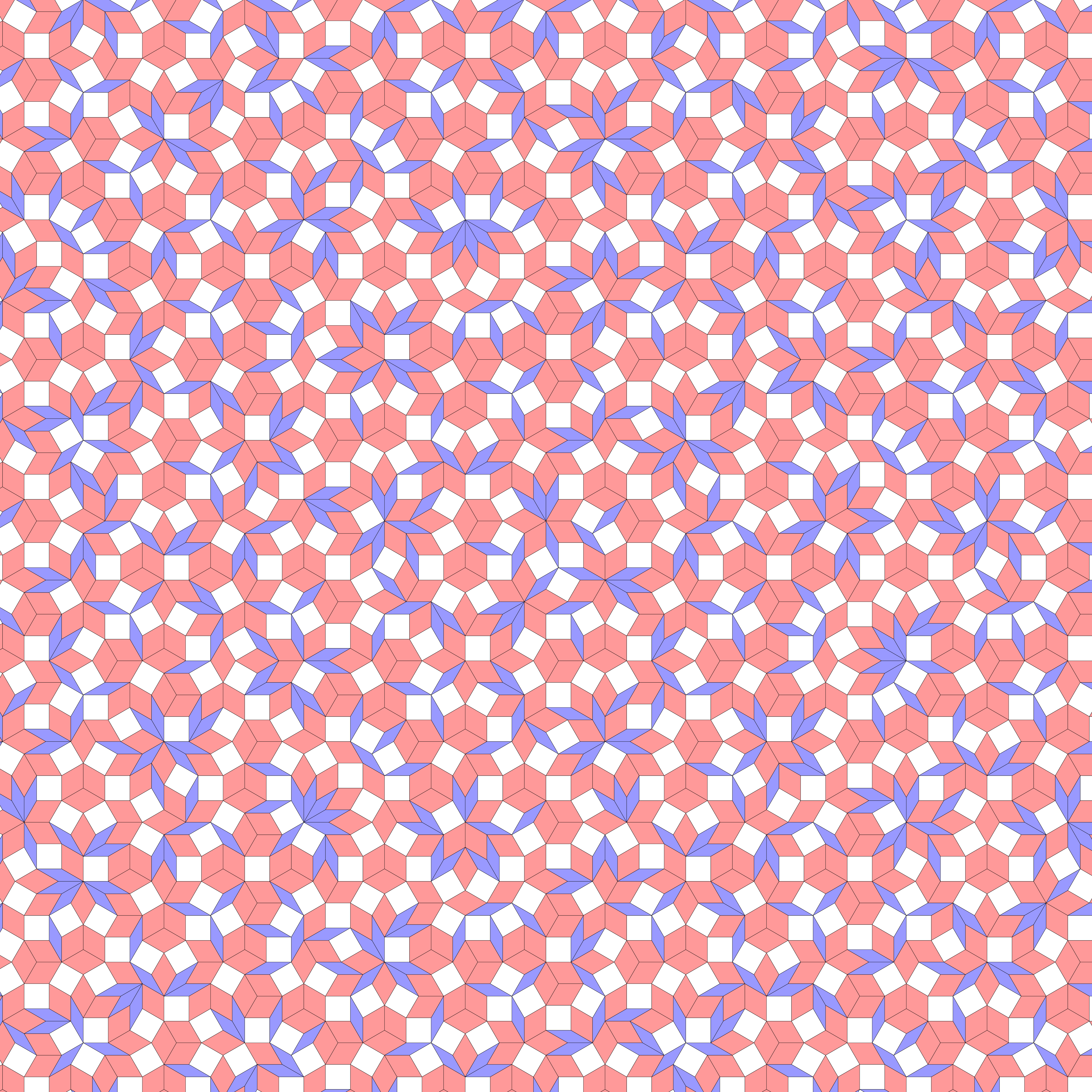}
            \caption[]%
            {{\small $\gamma$ generic}}    
        \end{subfigure}
        \caption[]%
        {\small Four patches of 12-fold tilings $E_{12}^{\gamma}$ for different values of the parameter $\gamma$.} 
        \label{fig:patches}
    \end{figure*}
Since we will be only concerned with the case $n=12$, from now on we will denote simply by $E,F$ etc.\ the spaces $E_{12}, F_{12}$ etc. We start by establishing some preferred bases for the spaces involved in the decompositions
$$
\mathbb R^6=E\oplus E^{\perp}\ \mathrm{and}\ E^{\perp}=F^{\perp}\oplus F.
$$
The plane $E$ will have the fixed orthonormal basis $\{u,v\}$ given by
$$
u=\frac{1}{\sqrt 3}\begin{pmatrix}
1\\ \sqrt{3}/2\\1/2\\0\\-1/2\\-\sqrt{3}/2
\end{pmatrix},
v=\frac{1}{\sqrt 3}\begin{pmatrix}
0\\1/2\\ \sqrt{3}/2\\1\\ \sqrt{3}/2\\ 1/2
\end{pmatrix}.
$$
The plane $F$ will have the following fixed orthogonal basis:
$$
A=\frac{1}{3}\begin{pmatrix}
1\\0\\-1\\0\\1\\0
\end{pmatrix},\ 
B=\frac{1}{3}\begin{pmatrix}
0\\1\\0\\-1\\0\\1
\end{pmatrix}.
$$
The basis $\{A,B\}$ is not orthonormal, but it has been chosen to facilitate the computations. Finally, the plane $F^{\perp}$ will be chosen to be the algebraic conjugate of $E$ and we will denote $\{u',v'\}$ its orthonormal basis. (Recall that to obtain $u',v'$ we simply need to replace $\sqrt3$ by $-\sqrt 3$ in the definition of $u$ and $v$.)

The projection $\pi^{\perp}$ onto $E^{\perp}$ is described in the following lemma. The simple proof is left to the reader.

\begin{lemma}\label{projections}
The image of the canonical basis of $\mathbb R^6$ by the projection $\pi^{\perp}$ onto $F^{\perp}\oplus F$ defines six vectors
$g_i, i=1,\dots, 6$. We have $g_i=f_i+\delta_i$ with $f_i\in F^{\perp}, \delta_i\in F$ and:
$$\begin{cases}
f_1=(\sqrt 3/3,0)\\
f_2=(-1/2,\sqrt 3/6)\\
f_3=(\sqrt 3/6,-1/2)\\
f_4=(0,\sqrt 3/3)\\
f_5=(-\sqrt 3/6,-1/2)\\
f_6=(1/2,\sqrt 3/6)
\end{cases}\quad\quad\mathrm{and}\quad\quad 
\begin{cases}
\delta_1=(1,0)\\
\delta_2=(0,1)\\
\delta_3=(-1,0)\\
\delta_4=(0,-1)\\
\delta_5=(1,0)\\
\delta_6=(0,1).\\
\end{cases}
$$
The vectors $f_{i}$ are expressed in the basis $\{u',v'\}$ and the $\delta_{i}$ in the basis $\{A,B\}$. 
\end{lemma}

%%%%%%%%%%%
%%%%%%%%%%%
\subsection{Algebra}
In this section we collect some technical lemmas that will be useful in the forthcoming computations. Many of them are straightforward and the proofs are left to the reader.

\begin{lemma}\label{rem-simplif-gamma}
The vectors $f_i\in F^{\perp},\ i=1,\dots,6$, satisfy the following two linear relations over $\mathbb Z$: 
$$
f_5=f_3-f_1\quad\mathrm{and}\quad f_6=f_4-f_2.
$$
Moreover, we have the following relations (where the indices are taken modulo $6$)
$$
\begin{cases}
f_1-f_{5}=\sqrt{3} f_{6}\\
f_2-f_{6}=-\sqrt{3} f_{1}\\ 
f_i+f_{i+2}=-\sqrt 3 f_{i+1}\quad \forall i=1,\dots, 6
\end{cases}
$$
\end{lemma}

\begin{corollary}\label{cor-groupe-delta}
The following group isomorphisms hold 
$$
\pi^{\perp}(\mathbb Z^6)\sim \mathbb ZA\oplus \mathbb ZB\oplus \langle f_1,\dots,f_6\rangle_{\mathbb Z}\quad\mathrm{and}\quad \overline{\pi^{\perp}(\mathbb Z^6)}\sim \mathbb Z^2\oplus \mathbb R^2.
$$
\end{corollary}
\begin{proof}
Every element in $\pi^{\perp}(\mathbb Z^6)$ is of the form $\sum n_i g_i$. By Lemma~\ref{projections}, we can rewrite this expression as 
$$
\sum_{i=1}^6 n_i f_i+\sum_{i=1}^6 n_i \delta_i=\sum_{i=1}^6 n_i f_i+(n_1-n_3+n_5)A+(n_2-n_4+n_6)B.
$$ 
It is straightforward that $\mathbb Z^2$ is isomorphic to the abelian group generated by $A, B$, while Lemma~\ref{rem-simplif-gamma} implies that the closure of the set $\{\sum_{i=1}^6 n_if_i: n_i\in\mathbb Z\}$ is isomorphic to $\mathbb R^2$. 
\end{proof}

\begin{corollary}\label{c:fvsx}
If we identify the plane $F^{\perp}$ with $\mathbb C$, then the vectors $f_i$ correspond to some roots of the equation $z^{12}=3^{-6}.$ 
The correspondence, illustrated in Figure~\ref{fig-complexe}, can be written as follows, where $x$ denotes the complex number $e^{i\frac{\pi}{6}}$: 
\begin{equation}\label{eq-fx}
f_i=\frac{1}{\sqrt 3}x^{5(i-1)},\ i=1,\dots, 6.
\end{equation}
\end{corollary}

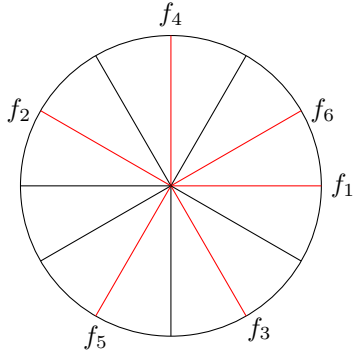
\begin{figure}
\begin{center}
\begin{tikzpicture}
\draw (0,0) circle(2);
\draw[red] (0,0)--(0:2);
\draw[red] (0,0)--(30:2);
\draw (0,0)--(60:2);
\draw[red] (0,0)--(90:2);
\draw (0,0)--(120:2);
\draw[red] (0,0)--(150:2);
\draw (0,0)--(0:-2);
\draw (0,0)--(30:-2);
\draw[red] (0,0)--(60:-2);
\draw (0,0)--(90:-2);
\draw[red] (0,0)--(120:-2);
\draw (0,0)--(150:-2);
\draw[right] (2,0) node{$f_1$};
\draw[right] (30:2) node{$f_6$};
\draw[above] (90:2) node{$f_4$};
\draw[left] (150:2) node{$f_2$};
\draw[below] (-120:2) node{$f_5$};
\draw[right] (-65:2.1) node{$f_3$};
\end{tikzpicture}
\caption{The vectors $f_1,\dots,f_6$ in $F^{\perp}$ and the directions of the 12\textsuperscript{th} roots of unity in $\mathbb C$.}\label{fig-complexe}
\end{center}
\end{figure}

\begin{lemma}\label{degree1}
The complex number $x=e^{i\frac{\pi}{6}}$ is a simple root of the polynomial $X^2-\sqrt 3X+1$. Moreover, every element of $\mathbb Z[x]$ can be expressed as a polynomial in $x$ of degree at most one with coefficients in $\mathbb Z[\sqrt 3]$. The precise expressions are collected in the following chart for $k\in\mathbb N$:
$$
{\renewcommand\arraystretch{1.3}
\begin{array}{|c|c|c|c|c|c|c|c|}
\hline
1&x&x^2&x^3&x^4&x^5&x^6&x^{6+k}\\
\hline
1&x&x\sqrt 3-1&2x-\sqrt 3&\sqrt 3x-2&x-\sqrt 3&-1&-x^k\\
\hline
\end{array}
}
$$
\end{lemma}

\noindent\textbf{Notation.} Since the group $\mathbb Z[\sqrt 3]$ will play a major role in this article, for simplicity we will denote it $G$.

\vspace{0.3cm}
When trying to determine the cohomology groups of the 12-fold tilings, we will need to understand when certain $\mathbb R$-linear combinations of powers of $x$ belong to $\mathbb Z[x]$. In some sense, the next lemma generalizes the equation $\sqrt{3} f_6=f_1-f_5$.

\begin{lemma}\label{lem-algebre-simple}
Let $\alpha,\beta\in\mathbb R$ and $i<j$ two integers in $[0,5]$.
\begin{enumerate}
\item $\alpha x^i\in \mathbb Z[x]$ holds if and only if $\alpha$ belongs to $G$.
\item  $\alpha x^i+\beta x^j\in \mathbb Z[x]$ holds if and only if 
$$
\begin{cases}
j-i=1,5\ \mathrm{and}\ \alpha,\beta\in G.\\
j-i=3\ \mathrm{and}\ \begin{cases}
	\beta=\frac{1}{2}(\beta_{1}+\beta_{2}\sqrt3),\ \alpha=\frac{1}{2}(\alpha_{1}+\alpha_{2}\sqrt3),\ \alpha_{i},\beta_{i}\in\mathbb Z,\\ 
	\alpha_{1}-3\beta_{2},\ \alpha_{2}-\beta_{1}\in2\mathbb Z.
					\end{cases}\\
j-i=2\ \mathrm{and}\ \begin{cases}
	\beta=\frac{1}{\sqrt3}(\beta_{1}+\beta_{2}\sqrt3),\ \alpha=\frac{1}{\sqrt3}(\alpha_{1}+\alpha_{2}\sqrt3),\ \alpha_{i},\beta_{i}\in\mathbb Z,\\ 
	\alpha_{1}-\beta_{1}\in3\mathbb Z.
					\end{cases}\\
j-i=4\ \mathrm{and}\ \begin{cases}
	\beta=\frac{1}{\sqrt3}(\beta_{1}+\beta_{2}\sqrt3),\ \alpha=\frac{1}{\sqrt3}(\alpha_{1}+\alpha_{2}\sqrt3),\ \alpha_{i},\beta_{i}\in\mathbb Z,\\		
	\alpha_{1}-2\beta_{1}\in3\mathbb Z.
					\end{cases}
\end{cases}
$$
\end{enumerate}
\end{lemma}
\begin{proof}
We start with the first claim. By Lemma~\ref{degree1} it is clear that if $\alpha$ belongs to $\mathbb Z[\sqrt 3]$ then $\alpha x^i$ belongs to $\mathbb Z[x]$. 
We prove the other implication. By assumption, $\alpha x^i=\sum_{j} n_j x^j, n_j\in\mathbb Z$, which implies $\alpha= x^{-i}\sum_{j} n_j x^j$. Again by Lemma~\ref{degree1}, we know we can write $\alpha =ax+b$, with $a,b \in G$. Since $\{x,1\}$ is a basis of $\mathbb C$ as an $\mathbb R$ vector space, we conclude that $\alpha=b\in G$.

We now proceed to argue the second claim in the statement. We need to understand which conditions on $\alpha, \beta$ imply
$ 
\alpha x^i+\beta x^j\in \mathbb Z[x].
$
Since $x^i$ is invertible in $\mathbb Z[x]$, this last expression is equivalent to 
$
\alpha+\beta x^{j-i}\in \mathbb Z[x],
$
which, by Lemma~\ref{degree1} can be rewritten as 
$
\alpha+\beta x^{j-i}=ax+b,\ a,b\in G.
$
Thus, depending on the value of $j-i$ and again by Lemma~\ref{degree1}, we are looking to one of the following equations 
\[
\begin{cases} 
\beta x+\alpha=ax+b \quad j-i=1.\\ 
\beta\sqrt 3x-\beta+\alpha=ax+b\quad j-i=2.\\ 
2\beta x-\beta\sqrt 3+\alpha=ax+b \quad j-i=3.\\  
\beta\sqrt 3 x-2\beta+\alpha=ax+b \quad j-i=4.\\
\beta x-\beta\sqrt 3+\alpha=ax+b\quad j-i=5.
\end{cases} 
\] 
Now, since $\{x,1\}$ is an $\mathbb R$ basis for $\mathbb C$, we deduce
\begin{itemize}
\item[-] If $j-1=1$: $\beta=a,\ \alpha=b\iff\beta,\alpha\in G$.
\item[-] If $j-1=5$: $\beta =a,\ -\beta\sqrt 3+\alpha=b\iff\beta,\alpha\in G$.
\item[-] If $j-1=3$: 
\begin{align*}
&2\beta=a, -\beta\sqrt 3+\alpha=b\iff\exists \beta_{1},\beta_{2},\alpha_{1},\alpha_{2}\in\mathbb Z,\ \beta=\frac{1}{2}(\beta_{1}+\beta_{2}\sqrt3),\\ 
&\alpha=\frac{1}{2}(\alpha_{1}+\alpha_{2}\sqrt3) \mathrm{\ and\ }b=\frac{\alpha_{1}-3\beta_{2}}{2}-\sqrt3\frac{\alpha_{2}-\beta_{1}}{2}\iff\\
&\beta,\alpha\in\frac{1}{2}G,\ \alpha_{1}-3\beta_{2},\ \alpha_{2}-\beta_{1}\in 2\mathbb Z.
\end{align*}
\end{itemize}

For the two remaining cases, an analogous computation to the preceding one, and using the same notations, yields
\begin{itemize}
\item[-] If $j-1=2$: $ \beta\sqrt 3=a,\ -\beta+\alpha=b\iff\alpha,\beta\in\frac{1}{\sqrt3}G\ \mathrm{and\ }\alpha_{1}-\beta_{1}\in3\mathbb Z.$
\item[-] If $j-1=4$: $ \beta\sqrt 3=a,\ -2\beta+\alpha=b\iff\alpha,\beta\in\frac{1}{\sqrt3}G\ \mathrm{and\ }\alpha_{1}-2\beta_{1}\in3\mathbb Z.$
\end{itemize}
\end{proof}

We end this section with a last technical lemma which will not be needed until the reader is confronted with Figure~\ref{fig-orbit-line-pair-impair}.

\begin{lemma}\label{lem-intersection-ensemble}
Let $S:=\{(\gamma_{1},\gamma_{2})\not\in\frac{1}{2\sqrt3}G\times\frac{1}{2}G\}$ and consider its two subsets:
$$
A=\{(\gamma_1, \gamma_2)\in S\,|\, 2\gamma_1\sqrt 3+2\gamma_2\in G\}\ \mathrm{and}\ B=\{(\gamma_1, \gamma_2)\in S\,|\,\gamma_1\sqrt 3+2\gamma_2\in G\}.
$$
Moreover, interchanging the roles of $\gamma_{1}$ and $\gamma_{2}$, we obtain the analogous sets
$S':=\{(\gamma_{1},\gamma_{2})\not\in\frac{1}{2}G\times\frac{1}{2\sqrt3}G\}$,
$$
D=\{(\gamma_1, \gamma_2)\in S'\,|\, 2\gamma_2\sqrt 3+2\gamma_1\in G\}\ \mathrm{and}\ 
E=\{(\gamma_1, \gamma_2)\in S'\,|\,  \gamma_2\sqrt 3+2\gamma_1\in G\}.
$$
Finally, set $C=S\setminus (A\cup B)$ and $F=S'\setminus (D\cup E).$
Then, 
\begin{enumerate}
\item $A\cap B=D\cap E=\emptyset$, which implies $S=A\sqcup B\sqcup C$ and $S'=D\sqcup E\sqcup F$.
\item $B\cap S' \subset F$ and $E\cap S\subset C$.
\item $A\cap D$ is the set of couples $(\gamma_{1},\gamma_{2})$ subject to the following conditions:
		\begin{enumerate}
		\item $(\gamma_{1},\gamma_{2})\not\in\frac{1}{2\sqrt3}(G\times G)$.
		\item $(\gamma_{1},\gamma_{2})\in\frac{1}{4}(G\times G)$, i.e. $4\gamma_{1}=a+b\sqrt3, 4\gamma_{2}=c+d\sqrt3$,  $a,b,c,d\in\mathbb Z$.	
		\item The integers $a,b,c,d$ satisfy one of the following parity conditions:
			\begin{enumerate}
			\item $a,d$ even and $c,b$ odd.
			\item $a,d$ odd and $c,b$ even.
			\item $a,b,c,d$ odd. 
			\end{enumerate}
		\end{enumerate}
	
\end{enumerate}
\end{lemma}
\begin{proof}
By the symmetry in the statement, it suffices to show the first claim for the sets $A,B$: if $(x,y)\in A\cap B$ then $2x\sqrt3+2y$ and $x\sqrt3+2y$ are in $G$. Subtracting these two expressions we obtain that $x\sqrt3
\in G$ which contradicts $(x,y)\in S$. We conclude $A\cap B=\emptyset$ and by an analogous argument $D\cap E=\emptyset$.

Now, we proceed to show that $B\cap S' \subset F$ and $E\cap S\subset C$. This amounts to show that $B\cap E=B\cap D=A\cap E=\emptyset$.
	\begin{itemize}
		\item $B\cap E=\emptyset$: if there was an element $(x,y)\in B\cap E\subset S\cap S'$, then $x,y\not\in\frac{1}{2\sqrt3}G$. However, since $(x,y)\in B\cap E$ we arrive to the following contradiction.
$$		
\begin{cases}
	x\sqrt3+2y\in G\\
	2x+y\sqrt3\in G
\end{cases}
		\Rightarrow
\begin{cases}
	6x+4\sqrt3y\in G\\
	6x+3\sqrt3y\in G
\end{cases}
		\Rightarrow		
\sqrt3y\in G.
$$
		\item $B\cap D=\emptyset$: if there was an element $(x,y)\in B\cap D\subset S\cap S'$, then $x,y\not\in\frac{1}{2\sqrt3}G$. However, since $(x,y)\in B\cap D$ we arrive to the following contradiction.
$$		
\begin{cases}
	x\sqrt3+2y\in G\\
	2x+2y\sqrt3\in G
\end{cases}
		\Rightarrow
\begin{cases}
	3x+2\sqrt3y\in G\\
	2x+2\sqrt3y\in G
\end{cases}
		\Rightarrow		
x\in G.
$$
		\item $A\cap E=\emptyset$: the analysis in this case is completely analogous to $B\cap D=\emptyset$ swapping the two variables.
	\end{itemize}

Finally, we will completely determine the set $A\cap D\subset S\cap S'$. If $(x,y)\in A\cap D\subset S\cap S'$, then $x,y\not\in\frac{1}{2\sqrt3}G$ and the following conditions need to be satistfied:
$$		
\begin{cases}
	2x\sqrt3+2y\in G\\
	2x+2y\sqrt3\in G
\end{cases}
		\Rightarrow
\begin{cases}
	2x\sqrt3+2y\in G\\
	2x\sqrt3+6y\in G
\end{cases}
		\Rightarrow
4y\in G.
$$
By the symmetry of the variables we must also have $4x\in G$. So, $(x,y)\in\frac{1}{4}(G\times G)$ and there exist $a,b,c,d\in\mathbb Z$ such that $4x=a+b\sqrt3$ and $4y=c+d\sqrt3$. Now, if $(x,y)\in A$ we have:
$$
2\sqrt3\frac{1}{4}(a+b\sqrt3)+2\frac{1}{4}(c+d\sqrt3)=\frac{3b+c}{2}+\frac{a+d}{2}\sqrt3\in G=\mathbb Z[\sqrt3].
$$
For this last condition to hold we need $a+d,3b+c\in2\mathbb Z$. Analyzing the constraints imposed by $(x,y)\in D$, we obtain $c+b,3d+a\in2\mathbb Z$. It follows that $a$ and $d$ and that $c$ and $b$ have the same parity. Moreover, since $x,y\not\in\frac{1}{2\sqrt3}G$, $a$ and $b$ cannot be both even and likewise $c$ and $d$. The parity conditions in the statement follow.

\end{proof}

%%%%%%%%%%%%%%%%%%%
\subsection{Description of the window}\label{delta=0}
%%%%%%%%%%%%%%%%%%%
In this section we describe the polytope $W=\pi^{\perp}([0,1]^6)$ and its intersection with the family of planes $P$. By definition, the window $W$ is the convex hull of the $2^6=64$ points $\sum_{i=1}^6 n_ig_i$, where each $g_i\in E^{\perp}$ is the projection of the $i^{\mathrm{th}}$ element of the canonical basis of $\mathbb R^6$ and $n_i$ is either 0 or 1. We end this section establishing a convention regarding the description of the three dimensional faces of W, which will be thoroughly studied in the next section. The main result of this part is the following proposition.
\begin{proposition}\label{Prop-W}
$\;$
\begin{itemize}
\item The polytope $W\subset E^{\perp}=F^{\perp}\oplus F$ has $52$ vertices, $132$ edges, $120$ faces of dimension two and $40$ faces of dimension three.  
\item All edges have the same length. 
\item The vertices are distributed in $16$ affine planes parallel to $F^{\perp}$. Each of these planes intersects
	\begin{itemize}
 		\item $F$ in a point (see Figure~\ref{polytope-coupe}); and 
		\item $W$ in either a point, a triangle or a hexagon.
 	\end{itemize}
\end{itemize}
\end{proposition}

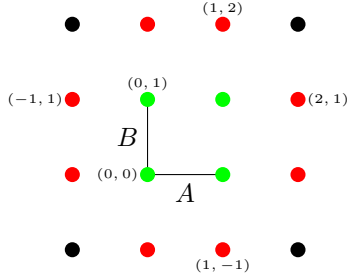
\begin{figure}
\begin{center}

\begin{tikzpicture}
\draw (0,0) node[left]{\tiny $(0,0)$};
\draw (0,1) node[above]{\tiny $(0,1)$};
\draw (2,1) node[right]{\tiny $(2,1)$};
\draw (1,-1) node[below]{\tiny $(1,-1)$};
\draw (-1,1) node[left]{\tiny $(-1,1)$};
\draw (1,2) node[above]{\tiny $(1,2)$};
\draw (0,0)--(1,0);
\draw (0,0)--(0,1);
\draw (0.5,0) node[below]{$A$};
\draw (0,0.5) node[left]{$B$};

\fill[green] (0,1) circle(.1);
\fill[green] (1,0) circle(.1);
\fill[green] (1,1) circle(.1);
\fill[green] (0,0) circle(.1);

\fill[red] (-1,0) circle(.1);
\fill[red] (2,0) circle(.1);
\fill[red] (0,-1) circle(.1);
\fill[red] (0,2) circle(.1);

\fill (-1,-1) circle(.1);
\fill (-1,2) circle(.1);
\fill (2,2) circle(.1);
\fill (2,-1) circle(.1);

\fill[red] (1,2) circle(.1);
\fill[red] (2,1) circle(.1);
\fill[red] (1,-1) circle(.1);
\fill[red] (-1,1) circle(.1);

\end{tikzpicture}
\caption{The polytope $W\subset E^{\perp}=F\cup F^{\perp}$ has its vertices on 16 planes parallel to $F^{\perp}$. The projection of each of these planes onto $F$, with fixed basis $\{A,B\}$, is a point. In this figure the number of vertices of $W$ in each of the different 16 planes is encoded by the following color code: Black=isolated point, red=triangle, green=hexagon.}\label{polytope-coupe}
\end{center}
\end{figure}

 The proof of this proposition will be splitted in several lemmas. 

\begin{lemma}\label{lem-cut}
The intersection of the window $W$ and the family of translates of $F^{\perp}$ given by $P=\bigcup_{\delta\in\Delta}F^{\perp}+\delta$ consists of points, triangles and hexagons.  
\end{lemma}
\begin{proof}
  Let us denote by $P_{\delta}$ the affine plane
  $F^{\perp}+\delta$. Throughout this proof we will follow the
  conventions established in Sections~\ref{subcohom}, \ref{subnfold},
  and \ref{conventions}.

By Corollary \ref{cor-groupe-delta}, $\Delta$ is isomorphic to $\mathbb Z^{2}$ with generators $\{A,B\}$. So, the family $P$ consists of all translates of $F^{\perp}$ by an integer linear combination of $A$ and $B$. The vertices of $W$ are of the form
\begin{equation}\label{e:proj_cube}
n_1g_1+\dots +n_6g_6=n_1f_1+\dots+ n_6f_6+(n_1-n_3+n_5)A+(n_2-n_4+n_6)B,
\end{equation}
where $n_i\in\{0,1\}$. Since the expressions $n_1+n_5-n_3$ and $n_2-n_4+n_6$ take only 4 different values, the numbers $\{-1,0,1,2\}$, it follows that the vertices of $W$ project onto $F$ onto 16 different points of the lattice $\Delta$. The first step will be to understand how many of these vertices are projected onto each of these 16 lattice points. To this end, in the next array we collect in how many different ways the above mentioned numbers are obtained:
$${\renewcommand\arraystretch{1.3}
\begin{array}{|c|c|c|c|c|}
\hline
n_i+n_{i+4}-n_{i+2}&-1&0&1&2\\
\hline
i=1&1&3&3&1\\
\hline
i=2&1&3&3&1\\
\hline
\end{array}}$$
It follows that on the plane $P_{(n_1+n_5-n_3,n_2-n_4+n_6)}$ the number of vertices of $W$ is at most the product of the corresponding entries on the second and third lines in the above array. (Notice that not all the 64 projections of the vertices of the unit cube in $\mathbb R^{6}$ need to be vertices of $W$.) After a straightforward computation, we obtain: 
\begin{itemize}
\item In each of $P_{(-1,-1)},P_{(-1,2)},P_{(2,-1)},P_{(2,2)}$ there is one single vertex of $W$.
\item In each of $P_{(-1,0)},P_{(-1,1)},P_{(0,-1)},P_{(0,2)},
P_{(1,-1)},P_{(1,2)},P_{(2,0)},P_{(2,1)}$ there are at most 3 vertices of $W$. It follows that the intersection of $W$ with any of these planes is the convex hull of the 3 points obtained projecting the canonical basis of $\mathbb R^{6}$. These points turn out to not be aligned in any of the 8 planes, so in this case the intersection of $W$ and these planes is a collection of 8 triangles. 
\item In each of $P_{(0,0)},P_{(0,1)},P_{(1,0)},P_{(1,1)}$ there are at most 9 vertices of $W$. We list the 9 points in each plane and compute their convex hull. It turns out that in each case the convex hull is a hexagon drawn in Figure \ref{fig-hexa-9pts}. In each of the 4 planes, the vertices of the corresponding hexagon in cyclic order are given by:

\begin{itemize}
\item Hexagon in the plane $P_{(0,0)}$:
$${\tiny \renewcommand\arraystretch{1.3}\begin{array}{|c|c|c|c|c|c|}
\hline
f_2+f_4&f_4+f_6&f_1+f_3+f_4+f_6&f_1+f_3&f_5+f_3&f_2+f_3+f_4+f_5\\
\hline
\end{array}}$$
\item Hexagon in the plane $P_{(1,0)}$:
$${\tiny \renewcommand\arraystretch{1.3}\begin{array}{|c|c|c|c|c|c|}
\hline
f_5+f_2+f_4&f_1+f_2+f_4&f_1+f_4+f_6&f_1+f_5+f_3+f_4+f_6&f_1+f_5+f_3&f_5\\
\hline\end{array} }$$
\item  Hexagon in the plane $P_{(0,1)}$
$${\tiny \renewcommand\arraystretch{1.3}\begin{array}{|c|c|c|c|c|c|}
\hline
f_2+f_4+f_6&f_2&f_3+f_2+f_5&f_6+f_5+f_3&f_1+f_6+f_3&f_1+f_6+f_3+f_2+f_4\\
\hline
\end{array}}$$
\item  Hexagon in the plane $P_{(1,1)}$:
$${\tiny \renewcommand\arraystretch{1.3}\begin{array}{|c|c|c|c|c|c|}
\hline
f_1+f_6&f_1+f_2+f_4+f_6&f_2+f_4+f_5+f_6&f_2+f_5&f_1+f_2+f_3+f_5&f_1+f_3+f_5+f_6\\ 
\hline\end{array}} $$
\end{itemize}
\end{itemize}
\end{proof}

\begin{figure}
\begin{tikzpicture}
\fill[right] (-0.933, 0.933) circle(.1); 
\fill[right] ( 0.067, 0.933) circle(.1); 
\fill[right] ( 0.933, 0.433) circle(.1); 
\fill[right] ( 0.433,-0.433) circle(.1); 
\fill[right] (-0.433,-0.933) circle(.1); 
\fill[right] (-0.933,-0.067) circle(.1); 
\fill[right] (-0.067, 0.433) circle(.1); 
\fill[right] (-0.433, 0.067) circle(.1); 
\fill[right] ( 0.067,-0.067) circle(.1); 
\draw (-0.933, 0.933) --
      ( 0.067, 0.933) --
      ( 0.933, 0.433) --
      ( 0.433,-0.433) --
      (-0.433,-0.933) --
      (-0.933,-0.067) -- cycle;
\end{tikzpicture}
\quad
\begin{tikzpicture}
\fill[right] ( 0.933, 0.933) circle(.1); 
\fill[right] (-0.067, 0.933) circle(.1); 
\fill[right] (-0.933, 0.433) circle(.1); 
\fill[right] (-0.433,-0.433) circle(.1); 
\fill[right] ( 0.433,-0.933) circle(.1); 
\fill[right] ( 0.933,-0.067) circle(.1); 
\fill[right] ( 0.067, 0.433) circle(.1); 
\fill[right] ( 0.433, 0.067) circle(.1); 
\fill[right] (-0.067,-0.067) circle(.1); 
\draw ( 0.933, 0.933) --
      (-0.067, 0.933) --
      (-0.933, 0.433) --
      (-0.433,-0.433) --
      ( 0.433,-0.933) --
      ( 0.933,-0.067) -- cycle;
\end{tikzpicture}
\quad
\begin{tikzpicture}
\fill[right] (-0.933,-0.933) circle(.1); 
\fill[right] ( 0.067,-0.933) circle(.1); 
\fill[right] ( 0.933,-0.433) circle(.1); 
\fill[right] ( 0.433, 0.433) circle(.1); 
\fill[right] (-0.433, 0.933) circle(.1); 
\fill[right] (-0.933, 0.067) circle(.1); 
\fill[right] (-0.067,-0.433) circle(.1); 
\fill[right] (-0.433,-0.067) circle(.1); 
\fill[right] ( 0.067, 0.067) circle(.1); 
\draw (-0.933,-0.933) --
      ( 0.067,-0.933) --
      ( 0.933,-0.433) --
      ( 0.433, 0.433) --
      (-0.433, 0.933) --
      (-0.933, 0.067) -- cycle;
\end{tikzpicture}
\quad
\begin{tikzpicture}
\fill[right] ( 0.933,-0.933) circle(.1); 
\fill[right] (-0.067,-0.933) circle(.1); 
\fill[right] (-0.933,-0.433) circle(.1); 
\fill[right] (-0.433, 0.433) circle(.1); 
\fill[right] ( 0.433, 0.933) circle(.1); 
\fill[right] ( 0.933, 0.067) circle(.1); 
\fill[right] ( 0.067,-0.433) circle(.1); 
\fill[right] ( 0.433,-0.067) circle(.1); 
\fill[right] (-0.067, 0.067) circle(.1); 
\draw ( 0.933,-0.933) --
      (-0.067,-0.933) --
      (-0.933,-0.433) --
      (-0.433, 0.433) --
      ( 0.433, 0.933) --
      ( 0.933, 0.067) -- cycle;
\end{tikzpicture}
\caption{The four hexagons which appear in the planes $P_{(0,0)},P_{(1,0)},P_{(0,1)},P_{(1,1)}$.}
\label{fig-hexa-9pts}
\end{figure}
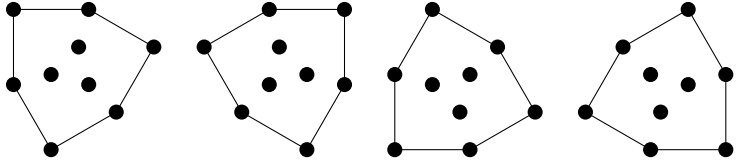
%%%%%%%%%%%%%%%%%%%%%%%%%%%%%%%%%%%%%%%%%%%%%%%%%%%%%%%%

From this last lemma we deduce several results on the structure of $W$. 
We state them here without proof. They are all the result of simple observations or direct computations. In what follows, we will refer to the points of intersection between $W$ and the planes $P_{(-1,-1)},P_{(-1,2)},P_{(2,-1)},P_{(2,2)}$ as \emph{isolated points} or \emph{isolated vertices}.

\begin{corollary}\label{cor-direc-delta-nul}
The edges of the polygons in the statement of Lemma~\ref{lem-cut} are parallel to  the directions $f_i\subset F^{\perp}, i=1,\dots, 6$.
\end{corollary}

\begin{corollary}\label{cor-sts-aretes}
The 1--dimensional complex of $W$, that is, the set of projections of the edges of the unit cube in $\mathbb R^{6}$ onto $F^{\perp}\oplus F$, has the following properties:

\begin{itemize}
\item No two vertices in one of the planes $P_{\delta}$ are connected by an edge.
\item Every isolated vertex has valency $6$.
\item Every vertex in a triangle has valency 5. The $5$ edges join the vertex in the triangle with $2$ vertices on the hexagons, with $2$ vertices on the closest triangle, and with the closest isolated vertex. 
\item In every hexagon, vertices of valency $6$ alternate with vertices of valency 4. The latter vertices connect to vertices in the planes at distance 1, the former ones exhibit the same pattern plus two edges connecting to vertices in other hexagons.
\end{itemize} 
\end{corollary}

In the following corollary we introduce the notion of a \textbf{cube}, which is our shortcut notation for a ``3-dimensional face of the polytope $W$''. These cubes will play an essential role in this article and the next subsection is devoted to their explicit description.

\begin{corollary}\label{cor-valencyandcubes}
The window $W$ has $40$ faces of dimension $3$. Every three dimensional face comprises $8$ vertices and will be referred to as a ``cube''. The 52 vertices of $W$ and the 320 vertices of the cubes satisfy:
\begin{itemize}
\item Each vertex of valency $6$ in $W$ belongs to $8$ cubes.
\item Each vertex of valency $5$ in $W$ belongs to $6$ cubes.
\item Each vertex of valency $4$ in $W$ belongs to $4$ cubes.
\end{itemize}
\end{corollary}

 In what follows we will need a more precise description of these cubes. Notice that, in order to completely determine a cube it suffices to know one of its vertices and the three edges of the cube intersecting at it. A vertex of a cube of $W$ is a point in $F^{\perp}\oplus F$ and we will express it as a linear combination of the vectors $g_{1},\dots,g_{6}$.
To describe the edges of the cube we start by remarking the following: in the plane $F$ there is a ``rightwards'' direction, given by the vector $A$ and an ``upwards'' direction, given by the vector $B$ (see Figure~\ref{polytope-coupe}). From equation~\eqref{e:proj_cube} it follows that an edge of a cube directed by 
\begin{itemize}
\item the vector $g_{1}$ or $g_{5}$: is horizontal and rightwards when projected onto $F$.
\item the vector $g_{2}$ or $g_{6}$: is vertical and upwards when projected onto $F$.
\item the vector $g_{3}$: is horizontal and leftwards when projected onto $F$.
\item the vector $g_{4}$: is vertical and downwards when projected onto $F$.
\end{itemize}
Therefore, it is enough to recall the projections onto $F^{\perp}$, given by vectors $f_{i}$, of the edges of the cube we want to describe, since it is immediate from this information and the above remarks how to recover the cube. Furthermore, instead of the $f_{i}$, since we are only interested at this point in the directions of the edges, we can use the complex vector $x=e^{\frac{i\pi}{6}}$ and its multiples $x^{i}, i=0,\dots,11$ to describe the edges of a cube (cf.\ Equation~\eqref{eq-fx}). With this notation the following holds:
\begin{itemize}
\item the vectors $x^{i}$ with $i$ even represent edges of the cube which are horizontal in $F$. More specifically, $x^{i}$ with $i=0,4,8$ are edges of cubes that project onto rightwards vectors on $F$ while $x^{i}$ with $i=2,6,10$ project onto leftwards vectors.
\item the vectors $x^{i}$ with $i$ odd represent edges of the cube which are vertical in $F$. More specifically, $x^{i}$ with $i=1,5,9$ are edges of cubes that project onto upwards vectors on $F$ while $x^{i}$ with $i=3,7,11$ project onto downwards vectors.
\end{itemize}

\begin{convention}
In the rest of the paper: A cube is given by a vertex $g$ and three vectors $x^i, x^j ,x^k$ which support the edges of the cubes containing the vertex. We denote the cube by
$$
\{g, x^i, x^j, x^k\}.
$$
\end{convention}

\begin{example}
As an example, the set $\{g_{3}+g_{4}, x^{1},x^{5},x^{9}\}$ determines a cube in $W$ with a vertex in $g_{3}+g_{4}\in P_{(-1,-1)}$ and 3 edges which go upwards in the plane $F$ and thus connect the vertex $g_{3}+g_{4}$ with 3 vertices in the plane $P_{(-1,0)}$. Each of these vertices is itself connected to 2 vertices in the plane $P_{(-1,1)}$ with edges that can be again encoded by the vectors $x^{1},x^{5},x^{9}$. Finally, the 3 vertices in $P_{(-1,1)}$ (each of which connects to 2 vertices in $P_{(-1,0)}$) are also connected with the single vertex in $P_{(-1,2)}$. This is a complete description of the 6 vertices and 12 edges of a cube in $W$.
\end{example}

%%%%%%%%%
\subsection{Cubes}\label{cubes}
%%%%%%%%%

With the notation developed above, we now proceed to describe all the cubes in $W$. The details of the computation will not be needed in the rest of the paper; however, the last results of this section Proposition \ref{prop-listecubes} and Corollary \ref{cor-edges}, summarizing the findings, are essential.

According to Corollary~\ref{cor-sts-aretes} each of the 4 isolated vertices has valency 6 and from Corollary~\ref{cor-valencyandcubes} we learn that these 6 edges belong to 8 different cubes. After an easy computation we find out that if we fix any one of the isolated vertices, precisely two of the 8 cubes which contain it contain also another isolated vertex. It follows that the number of different cubes containing isolated vertices is $8+7+7+6=28$. The combinatorics of these cubes is as follows:
\begin{itemize}
\item The cubes containing two isolated vertices contain also the $6$ vertices of the two triangles on a line (or column) in Figure \ref{polytope-coupe}. There are $4$ such cubes. Following our convention, the two vertical cubes are given by the sets $\{g_{3}+g_{4},x^1, x^{5}, x^{9}\}$ and $\{g_{1}+g_{4}+g_{5},x^1, x^{5}, x^{9}\}$ while the two horizontal ones correspond to the sets $\{g_{3}+g_{4}, x^0, x^{4}, x^{8}\}$ and $\{g_{2}+g_{3}+g_{6}, x^0, x^{4}, x^{8}\}$. 
\item The remaining cubes containing one isolated vertex, which are 24 in number, have all two edges going in the same direction on $F$ and a third edge going in the perpendicular direction. The complete list follows:
	\begin{itemize}
		\item Cubes containing $v_{1}=g_{3}+g_{4}$ (bottom left isolated vertex): 
		$$
		\{v_{1}, x^i, x^{i+3}, x^{i+4}\},\ \ i=1,5,9\quad \mathrm{and}\quad \{v_{1}, x^i, x^{i+1}, x^{i+4}\},\ \ i=0,4,8.
		$$ 
		\item Cubes containing $v_{2}=g_{2}+g_{3}+g_{6}$ (top left isolated vertex): 
		$$
		\{v_{2}, x^i, x^{i+1}, x^{i+4}\},\ \ i=3,7,11\quad \mathrm{and}\quad \{v_{2}, x^i, x^{i+3}, x^{i+4}\},\ \ i=0,4,8.
		$$
		\item Cubes containing $v_{3}=g_{1}+g_{2}+g_{5}+g_{6}$ (top right isolated vertex): 
		$$
		\{v_{3}, x^i, x^{i+3}, x^{i+4}\},\ \ i=3,7,11\quad \mathrm{and}\quad \{v_{3}, x^i, x^{i+1}, x^{i+4}\},\ \ i=2,6,10.
		$$
		\item Cubes containing $v_{4}=g_{1}+g_{4}+g_{5}$ (bottom right isolated vertex): 
		$$
		\{v_{4}, x^i, x^{i+1}, x^{i+4}\},\ \ i=1,5,9\quad \mathrm{and}\quad \{v_{4}, x^i, x^{i+3}, x^{i+4}\},\ \ i=2,6,10.
		$$
	\end{itemize}
\end{itemize}

\begin{example} If we focus on the isolated vertex in the plane $P_{(-1,-1)}$ it belongs to 8 cubes which, when projected to $F$, can be described as: a cube developing completely to the right, with a vertex in the isolated point in $P_{(2,-1)}$; a cube developing completely vertically with a vertex in the isolated point $P_{(-1,2)}$; 3 cubes that have two edges towards the right direction, with their furthermost vertices in the planes $P_{(1,-1)}$ and $P_{(0,-1)}$ and one edge upwards; and finally 3 cubes with two upward edges, leaving the furthermost vertices in this direction on the planes $P_{(-1,1)}$ and $P_{(0,1)}$, and one rightward edge. The cubes around the other 3 isolated points have a completely analogous symmetric configuration.
\end{example}

The cubes which contain no isolated vertex have vertices on the triangles and hexagons described in Lemma~\ref{lem-cut}. Moreover, by Corollaries~\ref{cor-sts-aretes} and~\ref{cor-valencyandcubes} we know that each vertex on a triangle has valency 5 and thus it belongs to 6 different cubes. From the description above we have that 5 of these cubes contain an isolated vertex. The cube that we are missing connects the vertex in the triangle to two vertices in the closest hexagon and to one vertex in the closest triangle. There is a total of 12 such cubes, 3 for each pair of nearby triangles in Figure~\ref{polytope-coupe}. We will not give a description as precise as the ones above for these cubes since for our purposes the precise vertex on the triangle that these cubes contain will be irrelevant. The information we need from these cubes is the following:
	\begin{itemize}
		\item Each vertex $v_{1},v_{2},v_{3}$ in the triangle in the plane $P_{(-1,0)}$ belongs to a cube that contains no isolated vertex of $W$. These three cubes are given by $\{v_{j},x^{i},x^{i+5},x^{i+4}\}$, $i=0,4,8$.
		\item Each vertex $v_{1},v_{2},v_{3}$ in the triangle in the plane $P_{(0,-1)}$ belongs to a cube that contains no isolated vertex of $W$. These three cubes are given by $\{v_{j},x^{i},x^{i-1},x^{i+4}\}$, $i=1,5,9$.
		\item Each vertex $v_{1},v_{2},v_{3}$ in the triangle in the plane $P_{(2,1)}$ belongs to a cube that contains no isolated vertex of $W$. These three cubes are given by $\{v_{j},x^{i},x^{i+5},x^{i+4}\}$, $i=2,6,10$.
		\item Each vertex $v_{1},v_{2},v_{3}$ in the triangle in the plane $P_{(1,2)}$ belongs to a cube that contains no isolated vertex of $W$. These three cubes are given by $\{v_{j},x^{i},x^{i-1},x^{i+4}\}$, $i=3,7,11$.
	\end{itemize}

From this description we learn that we there are two different types of cubes in the polytope $W$: 
\begin{definition}\leavevmode
\begin{itemize}
\item
\emph{Standard cubes}, which when projected onto $F$ have two edges in the same direction and a third edge in the perpendicular one; 
\item \emph{Long cubes}, whose 3 defining edges project onto the same direction on $F$. 
\end{itemize}
These two types of cubes can be \emph{horizontal} or \emph{vertical}, depending on which is their longest direction when projected onto $F$, see Figure~\ref{fig-cubes-reseau}.
\end{definition}

We summarize our findings in the next proposition. 

\begin{proposition}\label{prop-listecubes}
The polytope $W$ has $40$ cubes.
\begin{itemize}
\item $2$ vertical and 2 horizontal long cubes each containing two isolated points.
\item $24$ standard cubes each containing precisely one isolated point.
\item $12$ standard cubes containing only points in the triangles and hexagons.
\end{itemize}
\end{proposition}

\begin{corollary}\label{cor-edges}
The sets of directions of the edges of the $40$ cubes in $W$ are: 
\begin{itemize}
\item the long cubes yield the sets $\{x^{0},x^{2},x^{4}\}$ and $\{x^{1},x^{3},x^{5}\}$.
\item The $24$ standard cubes with one isolated point yield 
$$
\{x^{i},x^{i+4},x^{i+1}\}\quad\mathrm{and}\quad\{x^{i},x^{i+4},x^{i+3}\}\quad\mathrm{with}\quad i=0,\dots,11.
$$
\item The $12$ remaining standard cubes yield 
$$
\{x^{i},x^{i+4},x^{i+5}\}, i=0,2,4,6,8,10\quad\mathrm{and}\quad\{x^{i},x^{i+4},x^{i-1}\}, i=1,3,5,7,9,11.
$$
\end{itemize}

\end{corollary}

\begin{figure}
\begin{tikzpicture}[scale=.7]

\fill[green] (0,1) circle(.1);
\fill[green] (1,0) circle(.1);
\fill[green] (1,1) circle(.1);
\fill[green] (0,0) circle(.1);

\fill[red] (-1,0) circle(.1);
\fill[red] (2,0) circle(.1);
\fill[red] (0,-1) circle(.1);
\fill[red] (0,2) circle(.1);

\fill (-1,-1) circle(.1);
\fill (-1,2) circle(.1);
\fill (2,2) circle(.1);
\fill (2,-1) circle(.1);

\fill[red] (1,2) circle(.1);
\fill[red] (2,1) circle(.1);
\fill[red] (1,-1) circle(.1);
\fill[red] (-1,1) circle(.1);

\draw (-1,-1)--(-1,0)--(.2,.2)--(1,0)--(-.2,-.2)--(-1,0);
\draw (-.2,-.2)--(-.2,-1.2)--(-1,-1);
\draw (-.2,-1.2)--(1,-1)--(1,0);
\draw[dashed] (-1,-1)--(.2,-.8)--(.2,.2);
\draw[dashed] (1,-1)--(.2,-.8);
\end{tikzpicture}
\hspace{4mm}
\begin{tikzpicture}[scale=.7]

\fill[green] (0,1) circle(.1);
\fill[green] (1,0) circle(.1);
\fill[green] (1,1) circle(.1);
\fill[green] (0,0) circle(.1);

\fill[red] (-1,0) circle(.1);
\fill[red] (2,0) circle(.1);
\fill[red] (0,-1) circle(.1);
\fill[red] (0,2) circle(.1);

\fill (-1,-1) circle(.1);
\fill (-1,2) circle(.1);
\fill (2,2) circle(.1);
\fill (2,-1) circle(.1);

\fill[red] (1,2) circle(.1);
\fill[red] (2,1) circle(.1);
\fill[red] (1,-1) circle(.1);
\fill[red] (-1,1) circle(.1);

\draw (-1,-1)--(-1.2,-.2)--(-1,1)--(0,1)--(.2,.2)--(0,-1)--(-1,-1);
\draw (-1.2,-.2)--(-.2,-.2)--(0,-1);
\draw (0,1)--(-.2,-.2);
\end{tikzpicture}
\hspace{4mm}
\begin{tikzpicture}[scale=.6]

\fill[green] (0,1) circle(.1);
\fill[green] (1,0) circle(.1);
\fill[green] (1,1) circle(.1);
\fill[green] (0,0) circle(.1);

\fill[red] (-1,0) circle(.1);
\fill[red] (2,0) circle(.1);
\fill[red] (0,-1) circle(.1);
\fill[red] (0,2) circle(.1);

\fill (-1,-1) circle(.1);
\fill (-1,2) circle(.1);
\fill (2,2) circle(.1);
\fill (2,-1) circle(.1);

\fill[red] (1,2) circle(.1);
\fill[red] (2,1) circle(.1);
\fill[red] (1,-1) circle(.1);
\fill[red] (-1,1) circle(.1);

\draw (-1,-1)--(0,-1)--(1,-1.5)--(2,-1)--(1,-.5)--(0,-1);
\draw (1,-.5)--(0,-.5)--(-1,-1)--(0,-1.5)--(1,-1.5);
\end{tikzpicture}
\hspace{4mm}
\begin{tikzpicture}[scale=.7]

\fill[green] (0,1) circle(.1);
\fill[green] (1,0) circle(.1);
\fill[green] (1,1) circle(.1);
\fill[green] (0,0) circle(.1);

\fill[red] (-1,0) circle(.1);
\fill[red] (2,0) circle(.1);
\fill[red] (0,-1) circle(.1);
\fill[red] (0,2) circle(.1);

\fill (-1,-1) circle(.1);
\fill (-1,2) circle(.1);
\fill (2,2) circle(.1);
\fill (2,-1) circle(.1);

\fill[red] (1,2) circle(.1);
\fill[red] (2,1) circle(.1);
\fill[red] (1,-1) circle(.1);
\fill[red] (-1,1) circle(.1);

\draw (-1,-1)--(-1,0)--(-1.5,1)--(-1,2)--(-.5,1)--(-1,0);
\draw (-.5,1)--(-.5,0)--(-1,-1)--(-1.5,0)--(-1.5,1);
\end{tikzpicture}
\caption{standard horizontal, standard vertical, long horizontal and long vertical cubes. Remark that all vertices in the long cubes are projected onto the grid points in a horizontal or vertical line: one vertex on each black dot and three vertices on each red dot.}\label{fig-cubes-reseau}
\end{figure}
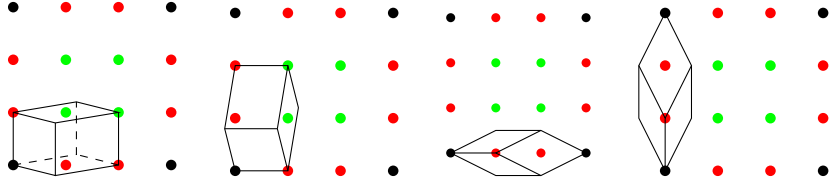

%%%%%%%%%%%%%
%%%%%%%%%%%%%
\section{Cubes, lines of intersection and $H^1(\Omega_{E_{12}^\gamma})$ }\label{sec-droites}
%%%%%%%%%%%%
%%%%%%%%%%%%%
%%%%%%%%%%%%%%%%%
\subsection{Intersection of a cube and a plane}\label{sec-cubes-droites}
%%%%%%%%%%%%%%%%%%
We want to apply the method described in Section~\ref{practical} to compute the cohomology groups of the generalized 12-fold tilings. The first step is to intersect the polytope $W$ by the families of parallel planes $P^{\gamma}=\bigcup_{\delta\in\Delta}F^{\perp}+\gamma+\delta$ with $\gamma=(\gamma_{1},\gamma_{2})\in F$ (expressed with respect to the fixed orthogonal basis $\{A,B\}$). Due to the periodicity of $\Delta$, restricting $\gamma_1, \gamma_2$ to vary in the range $[0,1)\times[0,1)$ is enough to obtain all the different families $P^{\gamma}$. The intersection of $W$ with $P^{\gamma}$ when $\gamma=(0,0)$ has been treated in Proposition \ref{Prop-W} and we now proceed to study the intersection when $\gamma\neq (0,0)$. To this end, we start noticing that a plane intersects $W$ if and only if it intersects one of its 3 dimensional faces. Indeed, since the $4$-dimensional polytope $W$ has a boundary consisting of $3$--dimensional faces and it sits in the $4$--dimensional space $F^{\perp}\oplus F$, a plane cannot intersect it without intersecting its boundary. (Notice that a plane could intersect $W$ without intersecting an edge or a 2 dimensional face.) Thus, our next goal is to intersect the 40 cubes described in Proposition~\ref{prop-listecubes} with the planes in the families $P^{\gamma}=\bigcup_{\delta\in\Delta}F^{\perp}+\gamma+\delta$ and $\gamma\in(0,1)\times(0,1)$.

Inside the family of planes $P^{\gamma}$ very few of them do actually intersect the window $W$. Indeed, it is evident   
from Figure~\ref{fig:decalage-axe} that if $\gamma_{i}\neq 0$ for $i=1,2$ then there are precisely 9 planes in $P^{\gamma}$ with non trivial intersection: the ones that project into the squares delimited by the colored dots in the figure. On the other hand, if $\gamma\neq (0,0)$ but $\gamma_{1}= 0$ or $\gamma_{2}= 0$, then there are 12 such planes. If we want to distinguish one plane in the family $P^{\gamma}$ we will use the same notation as in the preceding section, that is, $P^{\gamma}_{\delta}$ is defined as the plane $F^{\perp}+\gamma+\delta$. Notice that in what follows, when we refer to a cube we will always consider it `parametrized' as in the explicit list of the 40 cubes which constitute the boundary of $W$ described in Section~\ref{cubes}.

\begin{lemma}\label{lem-intersection-schema}
Consider a plane $P^{\gamma}_{\delta}=F^{\perp}+\gamma+\delta$ intersecting a cube of $W$.

\begin{itemize}
\item If the cube is standard and horizontal, then the intersection is a segment directed by the difference of the two horizontal vectors in the cube.
\item If the cube is standard and vertical, then the intersection is a segment directed by the difference of the two vertical vectors.
\item If the cube is long, then $P^{\gamma}_{\delta}$ is orthogonal to a diagonal of the cube and thus $P^{\gamma}_{\delta}\cap W$ is either a triangle or a hexagon. 
\end{itemize}
\end{lemma}
\begin{proof}
An arbitrary point $p$ in a cube determined by the tuple $\{v,x^{i},x^{j},x^{k}\}$ has coordinates in $E^{\perp}=F^{\perp}\oplus F$ given by $p=v+\alpha_{1} g_{r}+\alpha_{2} g_{s}+\alpha_{3} g_{t}$ where $\alpha_\ell\in[0,1]$ and $g_{r},g_{s}$ and $g_{t}$ are the projections of the basis vectors in $\mathbb R^{6}$ onto $F^{\perp}\oplus F$ or their opposites which are determined from $x^{i},x^{j}$ and $x^{k}$ as previously explained. We are interested in the projection onto $F$ of this point, which we will denote $\pi_{F}(p):=\pi^{\perp}(p)\cap F$. 

If the cube is horizontal, then two of the exponents defining it will be even, say they are the two first ones, and therefore $\pi_{F}(p)=\pi_{F}(v)+(\alpha_{1}+\alpha_{2}, \alpha_{3})$. On the other hand, if the cube were vertical, we might assume that again the two first exponents defining it are odd and therefore in this case $\pi_{F}(p)=\pi_{F}(v)+(\alpha_{3},\alpha_{1}+\alpha_{2})$. It follows that if $p\in P^{\gamma}_{\delta}$ then, 
\begin{itemize}
\item in the horizontal case we have $\pi_{F}(v)+(\alpha_{1}+\alpha_{2}, \alpha_{3})=(\delta_{1}+\gamma_{1},\delta_{2}+\gamma_{2})$;
\item in the vertical case $\pi_{F}(v)+(\alpha_{3},\alpha_{1}+\alpha_{2})=(\delta_{1}+\gamma_{1},\delta_{2}+\gamma_{2})$ must hold.
\end{itemize}
So, for a point in a horizontal cube to be also in $P^{\gamma} _{\delta}$ its $\alpha_{3}$ component, the vertical one in this case, is completely determined by the above equality while $\alpha_{1}$ and $\alpha_{2}$ are subject to the linear condition $\alpha_{1}+\alpha_{2}=$ some constant. As $\alpha_{1}$ and $\alpha_{2}$ vary in $[0,1]$, this condition determines a segment in $P^{\gamma}_{\delta}$. Indeed, the projection of $p$ onto this plane can be written as $\pi_{P^{\gamma}_{\delta}}(v)+\alpha_{1} f_{r}+\alpha_{2} f_{s}+\alpha_{3} f_{t}$ and the above conditions found on the $\alpha_{i}$'s allow us to rewrite the expression as:
\begin{align}\label{eq-segment}
&\pi_{P^{\gamma}_{\delta}}(v)+(\delta_{1}+\gamma_{1}-\pi_{F}(v)_{1})f_{r}-\alpha_{2}f_{r}+\alpha_{2} f_{s}+(\delta_{2}+\gamma_{2}-\pi_{F}(v)_{2}) f_{t}\nonumber\\
=&\pi_{P^{\gamma}_{\delta}}(v)+(\delta_{1}+\gamma_{1}-\pi_{F}(v)_{1})f_{r}+\alpha_{2}(f_{s}-f_{r})+(\delta_{2}+\gamma_{2}-\pi_{F}(v)_{2}) f_{t}
\end{align}
Notice that in this last expressions all parameters are fixed except for $\alpha_{2}$ and as it varies it defines the segment of intersection between $P^{\gamma}_{\delta}$ and the cube $\{v,x^{i},x^{j},x^{k}\}$. As claimed in the statement, this segment is directed by the vector $f_{s}-f_{r}$. This vector is a multiple of $x^{i}-x^{j}$, the difference of the two even powers of $x$ defining the cube. The argument is completely analogous in the vertical case. See Figure~\ref{fig-intersection-cubes}.

We now proceed to examine the intersection of $P^{\gamma}_{\delta}$ with a long cube. For the sake of concreteness we will assume that the cube is horizontal, which means that it is directed by three even powers of $x$ (the vertical case is completely analogous with three odd powers instead). Just as before, a point in this cube is given by $p=v+\alpha_{1} g_{r}+\alpha_{2} g_{s}+\alpha_{3} g_{t}$ and this time its projection onto $F$ reads $\pi_{F}(p)=\pi_{F}(v)+(\alpha_{1}+\alpha_{2}+\alpha_{3},0)$. If we assume $p\in P^{\gamma}_{\delta}$ we obtain the equality $\pi_{F}(v)+(\alpha_{1}+\alpha_{2}+\alpha_{3},0)=(\delta_{1}+\gamma_{1},\delta_{2}+\gamma_{2})$. This means that the plane and the cube are in the same affine space of dimension $3$. In this space the equation of the plane is $\alpha_{1}+\alpha_{2}+\alpha_{3}=\delta_1+\gamma_{1}-\pi_{F}(v)_{1}$ and the cube is $[0,1]^3$. Thus the plane is orthogonal to a diagonal of the cube. The intersection is a polygon, more precisely, a triangle or a hexagon depending on the value of $\delta_1+\gamma_{1}-\pi_{F}(v)_{1}$. See Figure~\ref{fig-intersection-cubes}.
\end{proof}

Using the proof of Lemma~\ref{lem-intersection-schema}, we could now describe the explicit equations of the intersection between the family of planes $P^{\gamma}$ and the cubes in $W$. Let us go through what would this entail in the case of a value of $\gamma$ with $\gamma_{i}\neq 0$ for $i=1,2$. In this scenario, we are looking at $9$ different planes in the family $P^{\gamma}$ that effectively intersect the cubes in $W$ and the number of cubes intersecting each plane varies depending on where in F, see Figure~\ref{fig:decalage-axe}, these planes are projected to. For example, as further discussed in the caption of Figure~\ref{fig:decalage-axe}, if a plane $P^{\gamma}_{\delta}\in P^{\gamma}$ projects to the center square, then it intersects $12$ cubes of $W$ whereas if it projects onto a corner square it meets only 6 cubes. Since we are considering the case $\gamma_{i}\neq 0$ for $i=1,2$, none of the planes in $P^{\gamma}$ intersect any long cubes. It follows, from Lemma~\ref{lem-intersection-schema}, that every intersection between a plane in $P^{\gamma}$ and $W$ is a segment. Adding up all the numbers in the grid of the leftmost picture in Figure~\ref{fig:decalage-axe} we conclude that the intersection of $P^{\gamma}$ and the cubes in $W$, when $\gamma_{i}\neq 0$ for $i=1,2$, is a collection of 72 segments.

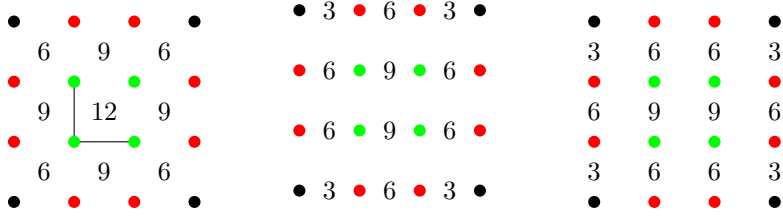
\begin{figure}
\begin{tikzpicture}[scale=.8]
\draw (0,0)--(1,0);
\draw (0,0)--(0,1);
 
\fill[green] (0,1) circle(.1);
\fill[green] (1,0) circle(.1);
\fill[green] (1,1) circle(.1);
\fill[green] (0,0) circle(.1);
\fill[green] (0,1) circle(.1);

\fill[red] (-1,0) circle(.1);
\fill[red] (2,0) circle(.1);
\fill[red] (0,-1) circle(.1);
\fill[red] (0,2) circle(.1);

\fill (-1,-1) circle(.1);
\fill (-1,2) circle(.1);
\fill (2,2) circle(.1);
\fill (2,-1) circle(.1);

\fill[red] (1,2) circle(.1);
\fill[red] (2,1) circle(.1);
\fill[red] (1,-1) circle(.1);
\fill[red] (-1,1) circle(.1);
\draw (.5,.5)node{$12$};
\draw (1.5,.5)node{$9$};
\draw (.5,1.5)node{$9$};
\draw (-.5,.5)node{$9$};
\draw (.5,-.5)node{$9$};

\draw (-.5,1.5)node{$6$};

\draw (-.5,-.5)node{$6$};

\draw (1.5,1.5)node{$6$};

\draw (1.5,-.5)node{$6$};
\end{tikzpicture}
\hspace{1cm}
\begin{tikzpicture}[scale=.8]

\fill[green] (0,1) circle(.1);
\fill[green] (1,0) circle(.1);
\fill[green] (1,1) circle(.1);
\fill[green] (0,0) circle(.1);

\fill[red] (-1,0) circle(.1);
\fill[red] (2,0) circle(.1);
\fill[red] (0,-1) circle(.1);
\fill[red] (0,2) circle(.1);

\fill (-1,-1) circle(.1);
\fill (-1,2) circle(.1);
\fill (2,2) circle(.1);
\fill (2,-1) circle(.1);

\fill[red] (1,2) circle(.1);
\fill[red] (2,1) circle(.1);
\fill[red] (1,-1) circle(.1);
\fill[red] (-1,1) circle(.1);

\draw (-.5,0) node{$6$};
\draw (.5,0) node{$9$};
\draw (1.5,0) node{$6$};
\draw (-.5,1) node{$6$};
\draw (.5,1) node{$9$};
\draw (1.5,1) node{$6$};
\draw (-.5,-1) node{$3$};
\draw (.5,-1) node{$6$};
\draw (1.5,-1) node{$3$};
\draw (-.5,2) node{$3$};
\draw (.5,2) node{$6$};
\draw (1.5,2) node{$3$};
\end{tikzpicture}
\hspace{1cm}
\begin{tikzpicture}[scale=.8]
\fill[green] (0,1) circle(.1);
\fill[green] (1,0) circle(.1);
\fill[green] (1,1) circle(.1);
\fill[green] (0,0) circle(.1);

\fill[red] (-1,0) circle(.1);
\fill[red] (2,0) circle(.1);
\fill[red] (0,-1) circle(.1);
\fill[red] (0,2) circle(.1);

\fill (-1,-1) circle(.1);
\fill (-1,2) circle(.1);
\fill (2,2) circle(.1);
\fill (2,-1) circle(.1);

\fill[red] (1,2) circle(.1);
\fill[red] (2,1) circle(.1);
\fill[red] (1,-1) circle(.1);
\fill[red] (-1,1) circle(.1);
\draw (-1,-.5) node{$3$};
\draw (-1,.5) node{$6$};
\draw (-1,1.5) node{$3$};
\draw (2,-.5) node{$3$};
\draw (2,.5) node{$6$};
\draw (2,1.5) node{$3$};
\draw (0,-.5) node{$6$};
\draw (0,.5) node{$9$};
\draw (0,1.5) node{$6$};
\draw (1,-.5) node{$6$};
\draw (1,.5) node{$9$};
\draw (1,1.5) node{$6$};
\end{tikzpicture}
\caption{Each of these colored grid points represent the projection on $F$ of a plane $F^{\perp}$ containing some vertices of the polytope $W$. Recall that the plane $F$ has basis $\{A,B\}$ highlighted on the left-most grid. The numbers in these grids represent the number of cubes in $W$ that the planes $P^{(\gamma_{1},\gamma_{2})}$ meet. The first grid collects the numbers for $\gamma_{1},\gamma_{2}\neq 0$, the second for $\gamma_{2}=0$ and the last one for $\gamma_{1}=0$.}
\label{fig:decalage-axe}
\end{figure}

Thankfully, for the computation we are interested in, we do not need to go into this level of detail. In fact, the segments of intersection define lines in the space $E^{\perp}=F^{\perp}\oplus F$ and we are interested in them only up to the action of $\Gamma=\pi(\mathbb Z^{6})$. In Lemma~\ref{lem-intersection-schema} we have described these segments/lines as points of the form $v\pm\alpha_{1}g_{i}\pm\alpha_{2}g_{j}\pm\alpha_{3}g_{k}$ where $v$ is a vertex of a cube and the $\alpha_{i}$'s are real numbers in the interval $[0,1]$ subject to some constraints. Since $v$ is a vertex of $W$ it is also an element of the group $\Gamma$ and therefore by a translation of vector $-v$ we might simply consider the expression $\alpha_{1}(\pm g_{i})+\alpha_{2}(\pm g_{j})+\alpha_{3}(\pm g_{k})$. Geometrically, this amounts to look at all the lines we are interested in together in the single plane $P_{(0,0)}^{\gamma}=P_{0}^{\gamma}$. At this point, since all the information about these lines on the $F$ component is summarized in the point $\gamma$, we might look at their expressions in the plane $P_{0}^{\gamma}$, which we recall is parallel to $F^{\perp}$, obtaining the expressions $\alpha_{1}(\pm f_{i})+\alpha_{2}(\pm f_{j})+\alpha_{3}(\pm f_{k})$. Our description of the cubes in $W$ is in terms of the powers of the complex vector $x=e^{\frac{i\pi}{6}}$ instead of the vectors $f_{i}$. In this language, the lines we are interested in are of the form 
$$
\alpha_{1}(\frac{1}{\sqrt 3} x^{r})+\alpha_{2}(\frac{1}{\sqrt 3} x^{s})+\alpha_{3}(\frac{1}{\sqrt 3} x^{t})
$$ 
with the necessary constraints on the parameters $\alpha_{i}$. 

As explained in Section~\ref{practical}, to compute the cohomology groups of the hull of the tiling, we need to know the quantity $L_{1}$, which is the number of orbits of lines in the intersection between $P^{\gamma}$ and $W$ under the action of $\Gamma$. With all the previous conventions in place, we are now ready to give a convenient description of these lines which will be handy to later compute $L_{1}$. We will use the standard convention denoting by $A+\lambda u$ the line of direction $u$ passing through the point $A$ with $\lambda\in\mathbb R$.

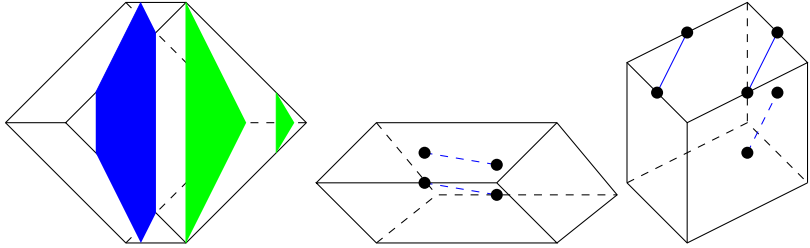
\begin{figure}
\begin{tikzpicture}[scale=.8]
\draw (0,0)--(1,0)--(3,2)--(5,0)--(3,-2)--(1,0);
\draw (0,0)--(2,2)--(3,2);
\draw (0,0)--(2,-2)--(3,-2);
\draw[dashed] (2,-2)--(4,0)--(5,0);
\draw[dashed] (2,2)--(4,0);
\draw[green] (3,2)--(3,-2);
\fill[green, opacity=10] (3,2)--(4,0)--(3,-2)--cycle;

\draw[green] (4.5,-.5)--(4.5,.5);
\fill[green, opacity=5] (4.5,.5)--(4.5,-.5)--(4.8,0)--cycle;

\draw[blue] (1.5,.5)--(2.25,2);
\fill[blue, opacity=10] (1.5,.5)--(1.5,-.5)--(2.25,-2)--(2.5,-1.5)--(2.5,1.5)--(2.25,2)--cycle;
\end{tikzpicture}
\begin{tikzpicture}[scale=.8]
\draw (0,0)--(3,0)--(4,-1)--(1,-1)--cycle;
\draw (0,0)--(1,1)--(4,1)--(5,-.2)--(4,-1);
\draw (3,0)--(4,1);
\draw[dashed] (1,-1)--(2,-.2)--(5,-.2);
\draw[dashed] (1,1)--(2,-.2);
\draw[blue, dashed] (1.8,0)--(3,-.2);
\draw[blue, dashed] (1.8,.5)--(3,.3);
\fill  (1.8, 0) circle(.1);
\fill  (1.8, .5) circle(.1);
\fill  (3,-.2) circle(.1);
\fill  (3,.3) circle(.1);
\end{tikzpicture}
\begin{tikzpicture}[scale=.8]
\draw (0,0)--(2,1)--(3,0)--(1,-1)--cycle;
\draw (0,-2)--(1,-3)--(3,-2);
\draw (0,0)--(0,-2);
\draw (3,0)--(3,-2);
\draw (1,-1)--(1,-3);
\draw[dashed] (0,-2)--(2,-1)--(3,-2);
\draw[dashed] (2,1)--(2,-1);
\draw[blue] (.5,-.5)--(1,.5);
\draw[blue] (2,-.5)--(2.5,.5);
\fill  (.5,-.5) circle(.1);
\fill  (1,.5) circle(.1);
\fill  (2,-.5) circle(.1);
\fill  (2.5,.5) circle(.1);
\draw[blue, dashed] (2,-1.5)--(2.5,-.5);
\fill  (2,-1.5) circle(.1);
\fill  (2.5,-.5) circle(.1);
\end{tikzpicture}
\caption{The left figure represents the intersection between a plane and a long cube. The plane and the cube live in the same 3-dimensional space and the intersection is either a hexagon (pink) or a triangle (green). The two other figures depict the intersection of the planes and the standard cubes. Each plane intersects a given cube in one line whose direction is determined by the cube. In the figure dotted lines are meant to be inside the cubes and solid ones on the faces.}\label{fig-intersection-cubes}
\end{figure}

\begin{proposition}\label{prop-intersections-cubes}
The intersection of the family of planes $P^\gamma$ with the cubes in $W$ yields a family of segments and polygons. Each of these segments and each side of a polygon belongs to a line that, when translated to the plane $P^\gamma_{0}$ via a vector in $\Delta\subset F$, can be described by one of the following $36$ equations:
$$
\pm(\frac{\gamma_1 }{\sqrt3}\begin{cases} x^{i}\\ x^{i+2}\\ x^{i+4}\end{cases}+\frac{\gamma_2}{\sqrt3} x^{i+1})+\lambda x^i\quad\mathrm{for}\ i\in\{0,2,4\}
$$
$$
\pm(\frac{\gamma_2}{\sqrt3}\begin{cases} x^{i+2}\\ x^{i+4}\\ x^{i+6}\end{cases}+\frac{\gamma_1}{\sqrt3} x^{i+1})+\lambda x^i\quad\mathrm{for}\ i\in\{1,3,5\}.
$$

\end{proposition}

\begin{proof}
The list of the 40 cubes which constitute the boundary of the window is explicit in the discussion before Proposition~\ref{prop-listecubes}. Since we are working modulo the translations of $\Delta$ we might ignore the vertices of the cubes and use only the information in Corollary~\ref{cor-edges}. 

We start considering standard cubes, which are all of the form $\{x^{i},x^{i+4},x^{k}\}$ for $i=1,\cdots,11$ and $k$ as in Corollary~\ref{cor-edges}. By  Lemma~\ref{lem-intersection-schema} the intersection of a standard horizontal cube and a single plane in $P^{\gamma}$ is a segment as in Equation~\eqref{eq-segment} directed by the
vector $x^{i+4}-x^{i}$, which is a multiple of the vector $x^{i-1}$. If the cube were vertical, the roles of $\gamma_{1}$ and $\gamma_{2}$ are swapped. For the sake of clarity, let us consider horizontal cubes. Since we are now working modulo the action of $\Delta$ and are interested in the whole line defined by the segment, the relevant part of the expression~\eqref{eq-segment} which we need reads  
$$
\gamma_{1}f_{r}+\alpha_{2}(f_{s}-f_{r})+\gamma_{2}f_{t}\quad\mathrm{where}\quad\alpha_{2}\in\mathbb R.
$$
This expression, rewritten in terms of the powers of $x$ turns into
$$
\gamma_{1}\frac{1}{\sqrt3}x^{i}+\gamma_{2}\frac{1}{\sqrt3}x^{k}+\alpha_{2}\frac{1}{\sqrt3}x^{i-1}.
$$
Finally, rewriting $\alpha_{2}\frac{1}{\sqrt3}$ as $\lambda\in\mathbb R$ we obtain the general expression of the lines we are looking for. To compile the list in the statement of the proposition we simply need to write down all the different lines obtained by this procedure from the standard cubes in Corollary~\ref{cor-edges}. To ease future computations we have adopted the convention that the lines will be written in such a way that the direction is given by the $i$-th power of $x$ yielding the expressions:
\begin{align*}
\mathrm{horizontal\ cubes:}&\ \ \gamma_{1}\frac{1}{\sqrt3}x^{i+1}+\gamma_{2}\frac{1}{\sqrt3}x^{k}+\lambda x^{i}\ \mathrm{for\ } 
\begin{cases}
i\ \mathrm{odd\ mod\ 12}\\
k=i+2,i+4,i+6
\end{cases}
\\
\mathrm{vertical\ cubes:}&\ \ \gamma_{2}\frac{1}{\sqrt3}x^{i+1}+\gamma_{1}\frac{1}{\sqrt3}x^{k}+\lambda x^{i}\ \mathrm{for\ }
\begin{cases}
 i\ \mathrm{even\ mod\ 12}\\
k=i,i+2,i+4
\end{cases}
\end{align*}
Note that we can make the notation a bit more compact taking into account that $x^{i}=-x^{i+6\ \mathrm{mod}\ 12}$. Adjusting the signs of $\gamma_{1}$ and $\gamma_{2}$ we obtain the first point in the statement of the proposition. Indeed, for $\gamma_{1},\gamma_{2}\neq 0$ 
all and only standard cubes are intersected by the planes in $P^{\gamma}$ since the long cubes are projected onto lines where one $\gamma_{i}$ is constantly equal to zero (see Figure~\ref{fig-cubes-reseau}).

We now consider the case of $\gamma=(0,\gamma_{2})$. The intersection of any plane in $P^{\gamma}$ with a standard horizontal cube will be either reduced to a point or a segment as described in Lemma~\ref{lem-intersection-schema}. Furthermore, the intersection with a standard vertical cube will be a segment as described in Lemma~\ref{lem-intersection-schema}. The reader can ``visualize'' this in Figure~\ref{fig-cubes-reseau}: for $\gamma=(0,\gamma_{2})$ the planes intersecting standard horizontal cubes might find their intersection on an edge of the cube or in the ``middle'', while the intersection with a vertical cube for these values of $\gamma$ happens on a ``lateral face''. We obtain thus all the lines in the first point of this proposition with $\gamma_{2}=0$. Furthermore, for these values of $\gamma$ the intersection with the two vertical long cubes is non trivial. By Lemma~\ref{lem-intersection-schema} the intersection consists of triangles or hexagons perpendicular to the diagonal of the cube. Since we are interested on the lines defined by the sides of the polygones only up to the action of $\Delta$, it is enough to consider the intersection with just one of the two cubes. We choose $\{g_{3}+g_{4},x^{1},x^{5},x^{9}\}$. The three planes in the family $P^{(0,\gamma_{2})}$ which intersect this long vertical cube are $P^{(0,\gamma_{2})}_{(-1,j)}$ with $j=-1,0,1$. 

If $j=-1$ we obtain that the intersection is a triangle whose edges, when considered modulo $\Delta$, that is, translated to the plane $P^{\gamma}_{0}$, can be described as
$$
\gamma_{2}\frac{1}{\sqrt3} x^{i}+\lambda x^{i-1}\quad\mathrm{for}\ i=1,5,9\ \mathrm{and}\ \lambda\in\mathbb R.
$$
(Recall that the vector $x^{i+4}-x^{i}$ is a multiple of $x^{i-1}$.)  Following the convention of writing the direction of the line as the $i$-th power we obtain the lines
$$
\gamma_{2}\frac{1}{\sqrt3} x^{i+1}+\lambda x^{i}\quad\mathrm{for}\ i=0,4,8\ \mathrm{and}\ \lambda\in\mathbb R.
$$
Now, if $j=0$ the intersection is a hexagon with parallel opposite sides. The lines they define have the same directions as the ones obtained for $j=-1$ and modulo $\Delta$ they coincide. Indeed, expressing the equations of these lines as $A+\lambda v$ we have that $A$ is a point of coordinates $g_{3}+g_{4}$ plus some integer multiples of some basis vectors $g_{i}$ (yielding the coordinates of another vertex of the cube) plus $\gamma_{2}\frac{1}{\sqrt3}x^{i}$ with $i=1,5$ or 9; and the vector $v$ is given by a difference of the form $x^{i+4}-x^{i}$. A completely analogous argument can be made to justify that modulo $\Delta$ we do not obtain new lines when considering the case $j=1$. Finally, notice that the 3 lines we have found for $j=-1$ coincide with lines obtained from the intersection with standard vertical cubes so they have been already been accounted for.

The case of the long horizontal cubes is completely analogous to the one of the vertical ones, we only need to swap the roles of $\gamma_{1}$ and $\gamma_{2}$ and of the even and odd powers of $x$.

Finally, the intersection $P^{\gamma}\cap W$ when $\gamma=(0,0)$ was studied in Lemma~\ref{lem-cut}. The lines we obtain in this case are supported by the sides of the triangles and hexagons found in this intersection. According to Corollary~\ref{cor-direc-delta-nul} these lines have all direction $x^{i}$ for some $i=0,\dots,5$. Moreover, since they all go through vertices of $W$, it is evident that they can be translated to lines going through $(0,0)\in P^{0}_{0}$ via an element in $\Delta$.
\end{proof}

%%%%%%%%%%%%%%%%%
%%%%%%%%%%%%%%%%%
\subsection{Orbits of lines}\label{sect:lines-gamma}
%%%%%%%%%%%%%%%%%
%%%%%%%%%%%%%%%%%

Now that we have a complete list of lines explicit in Proposition~\ref{prop-intersections-cubes}, we need to count the number of orbits of these lines under the action of $\Gamma=\Delta\oplus\langle f_1,\dots,f_6\rangle_\mathbb Z$ (see Corollary~\ref{cor-groupe-delta}). 
From Proposition~\ref{prop-intersections-cubes} we know that the lines appearing in $W\cap P^{\gamma}$ have direction $x^{i}$ with $i=0,\dots,5$ and, depending on the value of $\gamma$, for each fixed direction $x^{i_{0}}$ there are a certain number of lines when we consider them translated to the plane $P_{0}^{\gamma}$ via an element in $\Delta$. Since we are working in the plane $P_{0}^{\gamma}$ the action of $\Gamma$ is simplified to the action of $\Delta_{0}:=\langle f_1,\dots,f_6\rangle_\mathbb Z$. 
For computational reasons, it is better to use the powers of $x$ than the $f_{i}$, so we remind the reader that Equation~\eqref{eq-fx} tells us that $\Delta_{0}\sim\frac{1}{\sqrt 3}\mathbb Z[x]$.
The number of orbits will depend on the value of $\gamma$, however we can do a first overall simplification. 

\begin{lemma}\label{lem-4directions}
Consider the lines described by the following 24 equations.
$$
\frac{\gamma_1 }{\sqrt3}\begin{cases} x^{i+2}\\ x^{i+4}\end{cases}+\frac{\gamma_2}{\sqrt3} x^{i+1}+\lambda x^i,\ -(\frac{\gamma_1 }{\sqrt3}\begin{cases} x^{i+2}\\ x^{i+4}\end{cases}+\frac{\gamma_2}{\sqrt3} x^{i+1})+\lambda x^i,\ \ i\in\{0,2,4\}
$$
$$
\frac{\gamma_2}{\sqrt3}\begin{cases} x^{i+2}\\ x^{i+4}\end{cases}+\frac{\gamma_1}{\sqrt3} x^{i+1}+\lambda x^i,\ -(\frac{\gamma_2}{\sqrt3}\begin{cases} x^{i+2}\\ x^{i+4}\end{cases}+\frac{\gamma_1}{\sqrt3} x^{i+1})+\lambda x^i,\ \ i\in\{1,3,5\}.
$$
The two lines displayed together in each brace
are in the same $\Delta_{0}$ orbit.
\end{lemma} 
\begin{proof}
Since the argument is completely analogous for all the cases, in order to ease the notation we will write the details of the proof only for the lines:
$$
\frac{\gamma_1 }{\sqrt3}x^{i+2}+\frac{\gamma_2}{\sqrt3} x^{i+1}+\lambda x^i\quad\mathrm{and}\quad
\frac{\gamma_1 }{\sqrt3}x^{i+4}+\frac{\gamma_2}{\sqrt3} x^{i+1}+\lambda x^i\quad\mathrm{for\ }i\in\{0,2,4\}.
$$
The difference of two arbitrary points in these lines can be expressed as
$$
\pm\frac{\gamma_1 }{\sqrt3}(x^{i+2}-x^{i+4})+\mu x^i,\ \mathrm{for\ some\ } \mu\in\mathbb R.
$$
Now, it holds that $\pm(x^{i+2}-x^{i+4})=\pm x^{i}$, and therefore
$$
\pm\frac{\gamma_1 }{\sqrt3}(x^{i+2}-x^{i+4})+\mu x^i\in\Delta_{0}\iff (\pm\frac{\gamma_1}{\sqrt 3} +\mu)x^{i}\in\mathbb Z[x].
$$
By Lemma~\ref{lem-algebre-simple} this last condition holds if and only if $\pm\frac{\gamma_1}{\sqrt 3} +\mu\in\ G$, which is always satisfied for, say, $\mu=\mp\frac{\gamma_1}{\sqrt3}$ and thus the two lines we started with are in the same $\Delta_{0}$ orbit.
\end{proof}

\begin{corollary}\label{cor-simplification}
The lines in Proposition~\ref{prop-intersections-cubes} in different $\Delta_{0}$ orbits are at most the following $24$:
$$
\pm(\frac{\gamma_1 }{\sqrt3}\begin{cases} x^{i}\\ x^{i+2}\end{cases}+\frac{\gamma_2}{\sqrt3} x^{i+1})+\lambda x^i\quad\mathrm{for}\ i\in\{0,2,4\}
$$
$$
\pm(\frac{\gamma_2}{\sqrt3}\begin{cases} x^{i+4}\\ x^{i+6}\end{cases}+\frac{\gamma_1}{\sqrt3} x^{i+1})+\lambda x^i\quad\mathrm{for}\ i\in\{1,3,5\}.
$$

\end{corollary}

%%%%
\begin{proposition}\label{orbits-line}
The number $L_1$ of different orbits of lines under the action of $\Gamma$ depends on $\gamma=(\gamma_1, \gamma_2)$ and belongs to the set $\{6,9,12,15,18,21,24\}$. The precise value depends on various constraints on $\gamma$ and it can be found 
in Figure~\ref{fig-orbit-tout}. 

When $L_{1}=24$ we have that the 1-singularities have 6 different directions and each direction has $4$ representatives under the action of $\Gamma$. The other values of $L_{1}$ are obtained when some of the parallel representatives of the 1-singularities are identified under the action of $\Gamma$. The other extreme case, $L_{1}=6$, arrives when $\gamma_{1},\gamma_{2}\in G$ and there is only one orbit per direction.
\end{proposition}

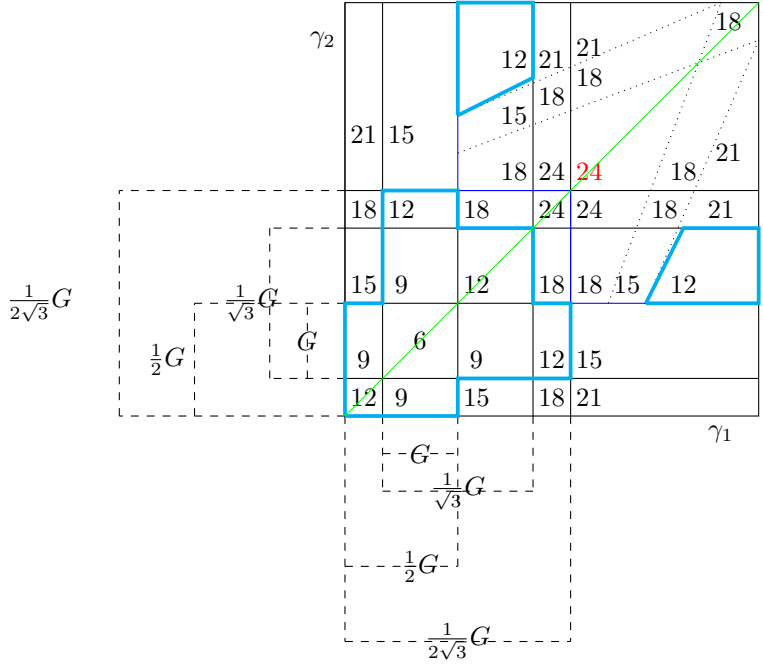
\begin{figure}
\begin{tikzpicture}[scale=.5]
\draw[-] (0,0)--(11,0)--(11,11)--(0,11)--cycle;
\draw[-] (0,0)--(0,11);
\draw[dashed] (1,-1)--(3,-1);
\draw[dashed] (1,-2)--(5,-2);
\draw[dashed] (0,-4)--(3,-4);
\draw[dashed] (0,-6)--(6,-6);
\draw (1,0)--(1,11);
\draw (3,0)--++(0,6);
\draw (5,0)--++(0,11);
\draw (6,0)--++(0,11);

\draw[dashed] (-1,1)--++(0,2);
\draw[dashed] (-2,1)--++(0,4);
\draw[dashed] (-4,0)--++(0,3);
\draw[dashed] (-6,0)--++(0,6);
\draw (0,1)--++(11,0);
\draw (0,5)--++(11,0);
\draw (0,3)--++(6,0);
\draw (0,6)--++(11,0);

\draw[dashed] (0,0)--++(0,-6);
\draw[dashed](1,0)--++(0,-2);
\draw[dashed] (3,0)--++(0,-4);
\draw[dashed] (5,0)--++(0,-2);
\draw[dashed] (6,0)--++(0,-6);

\draw[dashed] (0,0)--++(-6,0);
\draw[dashed] (0,1)--++(-2,0);
\draw[dashed] (0,5)--++(-2,0);
\draw[dashed] (0,6)--++(-6,0);
\draw[dashed] (0,3)--++(-4,0);

\draw[dotted] (3,8)--++(7,3);
\draw[dotted] (3,7)--++(8,3);
\draw[dotted] (8,3)--++(3,7);
\draw[dotted] (7,3)--++(3,8);

\draw[blue] (3,11)--(3,6)--(6,6)--(6,3)--(11,3);
\draw[line width=0.5mm, cyan]  (0,0)--(3,0)--(3,1)--(6,1)--(6,3)--(5,3)--(5,5)--(3,5)--(3,6)--(1,6)--(1,3)--(0,3)--(0,0);
\draw[line width=0.5mm, cyan]  (8,3)--(11,3)--(11,5)--(9,5)--(8,3);
\draw[line width=0.5mm, cyan]  (3,8)--(3,11)--(5,11)--(5,9)--(3,8);

\draw (2,-1) node[]{$G$};
\draw (-1,2) node[]{$G$};
\draw (3,-2) node[]{$\frac{1}{\sqrt 3}G$};
\draw (-2.5,3) node[]{$\frac{1}{\sqrt 3}G$};
\draw (2,-4) node[]{$\frac{1}{2}G$};
\draw (-4,1.5) node[left]{$\frac{1}{2}G$};
\draw (3,-6) node[]{$\frac{1}{2\sqrt 3}G$};
\draw (-7,3) node[left]{$\frac{1}{2\sqrt 3}G$};

\draw (10,0) node[below]{$\gamma_1$};
\draw (0,10) node[left]{$\gamma_2$};

\draw (1/2,1/2) node{$12$};
\draw (3/2,1/2) node{$9$};
\draw (1/2,3/2) node{$9$};
\draw (2,2) node{$6$};
\draw (1/2,7/2) node{$15$};
\draw (1/2,11/2) node{$18$};
\draw (3/2,7/2) node{$9$};
\draw (3/2,11/2) node{$12$};

\draw (7/2,1/2) node{$15$};
\draw (7/2,3/2) node{$9$};
\draw (7/2,7/2) node{$12$};
\draw (7/2,11/2) node{$18$};

\draw (11/2,11/2) node{$24$};
\draw (1/2,15/2) node{$21$};
\draw (3/2,15/2) node{$15$};

\draw (6.5,3.5) node{$18$};
\draw (7.5,3.5) node{$15$};
\draw (9,3.5) node{$12$};

\draw (6.5,5.5) node{$24$};
\draw (8.5,5.5) node{$18$};
\draw (10,5.5) node{$21$};

\draw (6.5,6.5) node{${\color{red}24}$};
\draw (6.5,9) node{$18$};
\draw (6.5,9.8) node{$21$};

\draw (5.5,6.5) node{$24$};
\draw (5.5,8.5) node{$18$};
\draw (5.5,9.5) node{$21$};

\draw (4.5,6.5) node{$18$};
\draw (4.5,8)   node{$15$};
\draw (4.5,9.5) node{$12$};

\draw (5.5,3.5) node{$18$};
\draw (5.5,1.5) node{$12$};
\draw (5.5,.5) node{$18$};

\draw (6.5,1.5) node{$15$};
\draw (6.5,.5) node{$21$};

\draw (9,6.5) node{$18$};
\draw (10.2,7) node{$21$};

\draw (10.2,10.5) node{$18$};
\draw[green] (0,0)--(11,11);
\end{tikzpicture}
\caption{The number of 1-singularities depending on  $\gamma=(\gamma_{1},\gamma_{2})\in\mathbb R^{2}$. As elsewhere in the paper, $G=\mathbb Z[\sqrt 3]$. Each axis in the figure represents the set $\mathbb R$ (unordered) divided into the non-disjoint subsets $G,\frac{1}{\sqrt 3}G,\frac{1}{2}G,\frac{1}{2\sqrt 3}G$ and $\mathbb R\setminus \frac{1}{2\sqrt 3}G$. The top right part of the chart is to be understood as the superposition of the regions labelled $A,B,C,D,E$ and $F$ in Figure~\ref{fig-orbit-line-pair-impair}. Moreover, adding 1 to each of the values on the chart we obtain the rank of $H^1(\Omega_{E_{12}^\gamma})$ (cf.\ Proposition~\ref{rank}). Finally, the regions enclosed in blue are the ones for which the second cohomology groups have been completely determined in Section~\ref{2lines}. }
\label{fig-orbit-tout}
\end{figure}
\begin{proof}
To start with, recall that we are working with the simplifying assumption that $\gamma$ belongs to $[0,1)^2$ (see beginning of Section~\ref{sec-cubes-droites}). Moreover, the orbits of lines under the action of $\Gamma$ will be computed via de action of the group $\Delta_{0}$ on translates by $\Delta$ of the lines to the plane $P^{\gamma}_{0}$ (see beginning of Section~\ref{sect:lines-gamma}).

$\bullet$ If the value of $\gamma=(0,0)$, then by Proposition~\ref{prop-intersections-cubes} we know that we have a total of 6 lines, one for each direction $x^{i}$, and these lines are evidently in different $\Delta_{0}$ orbits.

$\bullet$ If precisely one between $\gamma_{1}$ and $\gamma_{2}$ is zero then, by Proposition~\ref{prop-intersections-cubes}, there are 18 lines to be considered. We start considering the case $\gamma_{1}=0$. In this case there are six lines with 3 different directions $x^{i}$, $i\in\{0,2,4\}$ and 12 lines with direction $x^{i}$, $i\in\{1,3,5\}$. We want to understand under which conditions on $\gamma_{2}$ two parallel lines are on the same orbit under the action of $\Delta_{0}$. For the even directions, the equations of the lines under consideration are
$$
\frac{\gamma_{2}}{\sqrt3}x^{i+1}+\lambda x^{i}\quad\mathrm{and}\quad -\frac{\gamma_{2}}{\sqrt3}x^{i+1}+\lambda' x^{i}
$$
and they will be in the same $\Delta_{0}$ orbit if and only if there exist $\lambda,\lambda'\in\mathbb R$ such that
$$
\frac{\gamma_{2}}{\sqrt3}x^{i+1}+\lambda x^{i} -
(-\frac{\gamma_{2}}{\sqrt3}x^{i+1}+\lambda'x^{i})\in\Delta_{0}.
$$
Reorganizing we obtain that this condition is equivalent to
$$
\gamma_{2}\frac{2}{\sqrt3}x^{i+1}+x^{i}(\lambda-\lambda')\in\Delta_{0}\quad\mathrm{for\ some\ }\lambda,\lambda'\in\mathbb R.
$$
Finally, relabelling adequately the real parameters, this last condition reads
$$
2\gamma_{2} x^{i+1}+\mu x^{i}\in\mathbb Z[x]\quad\mathrm{for\ some\ }\mu\in\mathbb R,
$$
and by Lemma~\ref{lem-algebre-simple} it will be fulfilled if and only if $2\gamma_{2}, \mu\in G$. 

For the odd directions, $i\in\{1,3,5\}$, the equations of the lines under consideration are
$$
\frac{\gamma_{2}}{\sqrt3}x^{i+4}+\lambda_{1} x^{i},\ -\frac{\gamma_{2}}{\sqrt3}x^{i+4}+\lambda_{2} x^{i},\  \frac{\gamma_{2}}{\sqrt3}x^{i+6}+\lambda_{3} x^{i}\ \ \mathrm{and}\ \ -\frac{\gamma_{2}}{\sqrt3}x^{i+6}+\lambda_{4} x^{i}.
$$
Following the same arguments as before, we conclude that the two first lines will be in the same $\Delta_{0}$ orbit if 
$$
2\gamma_{2} x^{i+4}+\mu x^{i}\in\mathbb Z[x]\quad\mathrm{for\ some\ }\mu\in\mathbb R,
$$
which holds, via Lemma~\ref{lem-algebre-simple}, if $2\gamma_{2}\in\frac{1}{\sqrt3}G$.
Noticing that $x^{i+6}=-x^{i}$ we conclude that the last two lines are coincident and in fact one only line.

Finally we need to analyze the possibility of the lines passing through the points $\pm\frac{\gamma_{2}}{\sqrt3}x^{i+4}$ and $\frac{\gamma_{2}}{\sqrt3}x^{i}$ being in the same $\Delta_{0}$ orbit. This will be the case if
$$
\pm\frac{\gamma_{2}}{\sqrt3}x^{i+4}+x^{i}(\lambda-\lambda'-\frac{\gamma_{2}}{\sqrt3})\in\Delta_{0}\quad\mathrm{for\ some\ }\lambda,\lambda'\in\mathbb R,
$$
or equivalently, if
$$
\pm\gamma_{2}x^{i+4}+\mu x^{i}\in\mathbb Z[x]\quad\mathrm{for\ some\ }\mu\in\mathbb R.
$$
By Lemma~\ref{lem-algebre-simple} we conclude this is the case if $\gamma_{2}\in\frac{1}{\sqrt3}G$. Summing up, since $\frac{1}{\sqrt3}G\subset\frac{1}{2\sqrt3}G$ we have:
	\begin{itemize}
		\item[$*$] For each even direction there is one $\Delta_{0}$ orbit if $\gamma_{2}\in\frac{1}{2}G$ and two if $\gamma_{2}\not\in\frac{1}{2}G$. 
		\item[$*$] For each odd direction we can have:
			\begin{itemize}
				\item[$-$] three $\Delta_{0}$ orbits if $\gamma_{2}\not\in\frac{1}{2\sqrt3}G$.
				\item[$-$] two $\Delta_{0}$ orbits if $\gamma_{2}\in\frac{1}{2\sqrt3}G\setminus\frac{1}{\sqrt 3}G$ with representatives passing through the points $\frac{\gamma_{2}}{\sqrt3}x^{i+4}$ and $\frac{\gamma_{2}}{\sqrt3}x^{i}$.
				\item[$-$] one $\Delta_{0}$ orbit if $\gamma_{2}\in\frac{1}{\sqrt 3}G$.
			\end{itemize}
	\end{itemize}

The total number of $\Delta_{0}$ orbits is presented in Figure~\ref{fig-orbit-line-pair-impair}, where the right hand side has the even directions with values 3 and 6 and the left hand side the odd directions, with values 3,6 and 9. The quantity $L_{1}$ is computed by combining together the information on the lines with even and odd directions. This information can be read in Figure~\ref{fig-orbit-tout}.

The arguments and calculations needed to arrive to the computation of the number of $L_{0}$ when $\gamma=(\gamma_{1},0)$ are analogous to the ones presented for the case $\gamma=(0,\gamma_{2})$. Indeed, in the former case there will be 3 lines for each even direction since two of the ones in Corollary~\ref{cor-simplification} will be coincident and the usage of Lemma~\ref{lem-algebre-simple} will yield the same conclusions. All the relevant results are again contained in Figures~\ref{fig-orbit-tout} and~\ref{fig-orbit-line-pair-impair}.
\medskip

$\bullet$ We now proceed to analyze the case $\gamma_{i}\neq 0$ for $i=1,2$. To start with notice that a point in any of these lines is of the form $A+\lambda v$ where $A$ is a point in $P_{0}^{\gamma}$, $\lambda\in\mathbb R$ and $v$ is the direction of the line. If two lines $\ell$ and $\ell'$ are in the same $\Delta_{0}$-orbit we have that
$$
A+\lambda v - (A'+\lambda' v)\in\Delta_{0}\quad\mathrm{for\ some\ \lambda,\lambda'\in\mathbb R},
$$
which implies
$$
(A-A')+(\lambda-\lambda')v\in\Delta_{0},
$$
and relabeling the real parameter it is equivalent to
$$
A-A'+\mu v\in\Delta_{0}\quad\mathrm{for\ some\ }\mu\in\mathbb R.
$$
The vector $v=x^{i}$ and the power $i$ might be even or odd. We will write down the full details for the $i$ even case and the results for the odd case will be given at the end, without details. For each even power there are 4 different lines and our aim is to understand the action of $\Delta_{0}$ on them.

From Corollary~\ref{cor-simplification}, we know that for a fixed even direction $x^{i_{0}}$ the points $A,A'$ we need to consider are of the form
$$
\pm(\frac{\gamma_1 }{\sqrt3}\begin{cases} x^{i_{0}}\\ x^{i_{0}+2}\end{cases}+\frac{\gamma_2}{\sqrt3} x^{i_{0}+1})
$$
and, depending on the choices of signs, the difference $A-A'$ can be written as
\begin{itemize}[leftmargin=2cm]
\item[Case 1.] $\pm\frac{1}{\sqrt3}\gamma_{1}(x^{j}-x^{k})\quad j,k\in\{i_{0},i_{0}+2\},\ j\neq k,$ or
\item[Case 2.] $\pm\frac{1}{\sqrt3}\left(\gamma_{1}(x^{j}+x^{k})+2\gamma_{2}x^{i_{0}+1}\right)\quad j,k\in\{i_{0},i_{0}+2\}.$
\end{itemize}
We start considering the first of these two cases. Under these circumstances $A-A'+\mu x^{i_{0}}\in\Delta_{0}$ if and only if
$\pm\gamma_{1}(x^{j}-x^{k})+\mu' x^{i_{0}}\in\mathbb Z[x]$ for some $\mu'\in\mathbb R$. Due to the $\pm$ sign in front of this expression, we can assume that $j>k$, so $k=i_{0}$ and $j=i_{0}+2$, which implies $x^j-x^{k}=x^{i_{0}+4}$. It follows that in this first case the condition we need to determine on $\gamma_{1}$ is when does
$$
\gamma_{1}x^{i_{0}+4}+\mu' x^{i_{0}}\in\mathbb Z[x]\quad\mathrm{for\ some\ }\mu'\in\mathbb R
$$
hold. By Lemma~\ref{lem-algebre-simple} we conclude that this is the case if and only if 
\begin{equation}\label{eq4}
\sqrt 3\gamma_{1}\in G.
\end{equation}

\medskip

We move on to consider Case 2 above. Now $A-A'+\mu x^{i_{0}}\in\Delta_{0}$ if and only if 
$
\gamma_{1}(x^{j}+x^{k})+2\gamma_{2}x^{i_{0}+1}+\mu' x^{i_{0}}\in\mathbb Z[x],
$  
for some $\mu'\in\mathbb R$. With the appropriate substitutions of $x^{j}+x^{k}$, depending on whether $j=k$ or $j\neq k$, what we need to understand is for which values of $\gamma_{i}$ do the conditions
\begin{align}
&2(\gamma_{1}x^{i_{0}}+\gamma_{2}x^{i_{0}+1})+\mu'x^{i_{0}}\in\mathbb Z[x]\iff (2\gamma_{1}+\mu')x^{i_{0}}+2\gamma_{2}x^{i_{0}+1}\in\mathbb Z[x]\label{eq5}\\
&2(\gamma_{1}x^{i_{0}+2}+\gamma_{2}x^{i_{0}+1})+\mu'x^{i_{0}}\in\mathbb Z[x]\iff 2(\gamma_1 x^2+\gamma_2x)+\mu'\in \mathbb Z[x]\label{eq6}\\
&\gamma_{1}\sqrt3x^{i_{0}+1}+2\gamma_{2}x^{i_{0}+1}+\mu'x^{i_{0}}\in\mathbb Z[x]\iff(\gamma_{1}\sqrt3+2\gamma_{2})x^{i_{0}+1}+\mu'x^{i_{0}}\in\mathbb Z[x]\label{eq7}
\end{align}

With the help of Lemma~\ref{lem-algebre-simple} we conclude that condition~\eqref{eq5} is fulfilled if and only if $\gamma_{2}\in\frac{1}{2} G$, while condition~\eqref{eq7} is fulfilled if and only if $\gamma_{1}\sqrt3+2\gamma_{2}\in G$. Now, condition \eqref{eq6} is equivalent to $(2\sqrt 3\gamma_1+2\gamma_2)x+\mu-\gamma_1\in \mathbb Z[x]$. Thus we obtain the condition $2(\gamma_{1}\sqrt3+\gamma_{2})\in G$.

To help the reader visualize the results obtained so far, we propose the following figure:

\begin{center}
\begin{tikzpicture}
\draw[red] (0,0)--(3,0);
\draw[red] (0,3)--(3,3);
\draw[green] (0,0)--(3,3);
\draw[green] (3,0)--(0,3);
\draw[blue] (3,0)--(3,3);
\draw (0,0)--(0,3);
\fill (0,0) circle(.1);
\fill (3,0) circle(.1);
\fill (0,3) circle(.1);
\fill (3,3) circle(.1);
\draw (0,0) node[left]{$i$};
\draw (0,3) node[left]{$-i$};
\draw (3,0) node[right]{$i+2$};
\draw (3,3) node[right]{$-(i+2)$};
\draw (1.5,0) node[below]{$\sqrt 3\gamma_1$};
\draw (0,1.5) node[left]{$2\gamma_2$};
\draw (3,1.5) node[right]{$2(\gamma_{1}\sqrt3+\gamma_{2})$};
\draw (1.5,1.6) node[above]{$\gamma_{1}\sqrt3+2\gamma_{2}$};
\end{tikzpicture}
\end{center}
After fixing one even direction, every vertex labelled  $i, i+2, -i, -(i+2)$ in the figure represents one of the four parallel lines in this direction. The labelling corresponds to the power of $x$ with coefficient $\gamma_1$ in the description of the line in Corollary~\ref{cor-simplification}. The conditions we have studied above identify parallel lines as follows:
\begin{itemize}
\item Condition \ref{eq4} identifies lines labelled $\pm i$ and $\pm(i+2)$  if $\sqrt3\gamma_1\in G$.
\item Condition \ref{eq5} identifies lines labelled $i$ and $-i$  if $2\gamma_2\in G$.
\item Condition \ref{eq6} identifies lines labelled $i+2$ and $-(i+2)$ if $2(\gamma_{1}\sqrt3+\gamma_{2})\in G$.
\item Condition \ref{eq7} identifies lines labelled $\pm i$ and $\mp(i+2)$  if $\gamma_{1}\sqrt3+2\gamma_{2}\in G$.
\end{itemize}

To finish the analysis and obtain the number of orbits of lines under the action of $\Gamma$, we need to understand when do the above conditions arrive simultaneously. The results follow.
\begin{itemize}
\item If $\sqrt3\gamma_{1}\in G$ and
	\begin{itemize}
		\item $2\gamma_{2}\in G$, then there is 1 orbit.
		\item $2\gamma_{2}\not\in G$, then there are 2 orbits with representatives $i$ and $-i$.
	\end{itemize}
\item If $\sqrt3\gamma_{1}\in\frac{1}{2}G\setminus G$ and
	\begin{itemize}
		\item $2\gamma_{2}\in G$, then there are 2 orbits with representatives $i$ and $i+2$.
		\item $2\gamma_{2}\not\in G$, then there are 4 different orbits.
	\end{itemize}
\item If $\sqrt3\gamma_{1}\not\in\frac{1}{2}G$ and
	\begin{itemize}
		\item $2\gamma_{2}\in G$, then there are 3 orbits with representatives $i,\pm(i+2)$.
		\item $2\gamma_{2}\not\in G$ and
			\begin{itemize}
				\item[(A)] $2\gamma_{1}\sqrt3+2\gamma_{2}\in G$, then 3 orbits with repr. $\pm i$ and $i+2$.
				\item[(B)] $\gamma_{1}\sqrt3+2\gamma_{2}\in G$, then 2 orbits with repr. $i$ and $i+2$.
				\item[(C)] $2\gamma_{1}\sqrt3+2\gamma_{2}, \gamma_{1}\sqrt3+2\gamma_{2}\not\in G$, then 4 orbits.
			\end{itemize}	
  		\end{itemize}
\end{itemize}

Since the above results do not depend on the even direction fixed, to obtain the total number of orbits for the even directions we need to simply multiply the above numbers by $3$. The labelling $A,B,C$ in the last case corresponds to the sets and notation in Lemma~\ref{lem-intersection-ensemble}. The results up to here have been summarized in Figure~\ref{fig-orbit-line-pair-impair}.

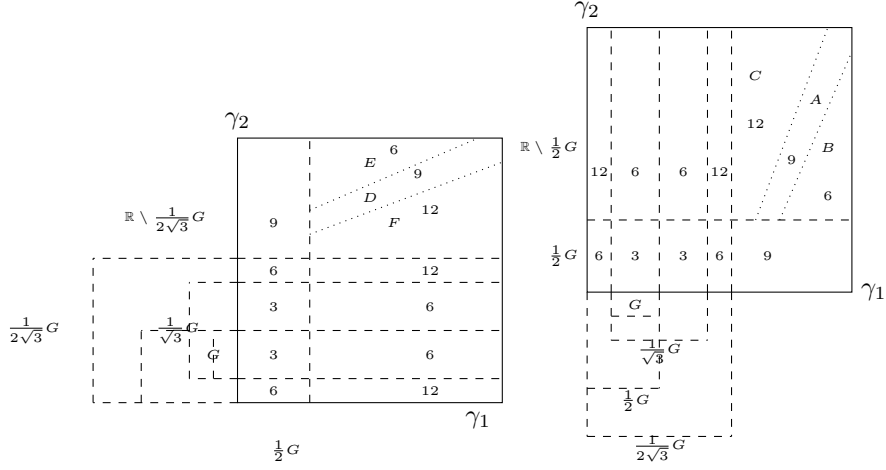
\begin{figure}
\begin{tikzpicture}[scale=.32]
\draw[-] (0,0)--(11,0)--(11,11)--(0,11)--cycle;

\draw[dashed] (3,0)--++(0,11);
\draw[dashed] (0,3)--++(11,0);
\draw[dashed] (0,1)--++(11,0);
\draw[dashed] (0,5)--++(11,0);
\draw[dashed] (0,6)--++(11,0);
\draw [dotted] (3,8)--(10,11);
\draw [dotted] (3,7)--(11,10);

\draw (10,0) node[right, below]{$\gamma_1$};
\draw (0,11) node[above]{$\gamma_2$};

\draw (-3,8.5) node[below]{{\tiny $\mathbb R\setminus \frac{1}{2\sqrt 3}G$}};

\draw (1.5,.5) node{{\tiny $6$}};
\draw (1.5,2) node{{\tiny$3$}};
\draw (1.5,4) node{{\tiny$3$}};
\draw (1.5,5.5) node{{\tiny $6$}};
\draw (8,2) node{{\tiny$6$}};
\draw (8,8) node{{\tiny$12$}};
\draw (1.5,7.5) node{{\tiny$9$}};
\draw (8,4) node{{\tiny$6$}};
\draw (8,.5) node{{\tiny $12$}};
\draw (8,5.5) node{{\tiny $12$}};
\draw (6.5,10.5) node{{\tiny$6$}};

\draw (5.5,10) node{{\tiny$E$}};
\draw (5.5,8.55) node{{\tiny$D$}};
\draw (6.5,7.5) node{{\tiny$F$}};

\draw (7.5,9.5) node{{\tiny $9$}};

\draw (-1,2) node[]{{\tiny $G$}};
\draw (2,-2) node[]{{\tiny $\frac{1}{2}G$}};
\draw (-2.5,3) node[]{{\tiny $\frac{1}{\sqrt 3}G$}};

\draw (-7,3) node[left]{{\tiny $\frac{1}{2\sqrt 3}G$}};
\draw[dashed] (-1,1)--++(0,2);
\draw[dashed] (-2,1)--++(0,4);
\draw[dashed] (-4,0)--++(0,3);
\draw[dashed] (-6,0)--++(0,6);
\draw[dashed] (0,0)--++(-6,0);
\draw[dashed] (0,1)--++(-2,0);
\draw[dashed] (0,5)--++(-2,0);
\draw[dashed] (0,6)--++(-6,0);
\draw[dashed] (0,3)--++(-4,0);
\end{tikzpicture}
\begin{tikzpicture}[scale=.32]

\draw[] (0,0)--(11,0)--(11,11)--(0,11)--cycle;
\draw[dashed] (3,0)--++(0,11);
\draw[dashed] (5,0)--++(0,11);
\draw[dashed] (1,0)--++(0,11);
\draw[dashed] (6,0)--++(0,11);
\draw[dashed] (0,3)--++(11,0);
\draw [dotted] (7,3)--(10,11);
\draw [dotted] (8,3)--(11,10);

\draw (11,0) node[right]{$\gamma_1$};
\draw (0,11) node[above]{$\gamma_2$};

\draw (2,-.5) node[]{{\tiny $G$}};

\draw (3,-2.5) node[]{{\tiny $\frac{1}{\sqrt 3}G$}};
\draw (2,-4.5) node[]{{\tiny $\frac{1}{2}G$}};

\draw (3,-6.5) node[]{{\tiny $\frac{1}{2\sqrt 3}G$}};

\draw[dashed] (1,-1)--(3,-1);
\draw[dashed] (1,-2)--(5,-2);
\draw[dashed] (0,-4)--(3,-4);
\draw[dashed] (0,-6)--(6,-6);

\draw[dashed] (0,0)--++(0,-6);
\draw[dashed](1,0)--++(0,-2);
\draw[dashed] (3,0)--++(0,-4);
\draw[dashed] (5,0)--++(0,-2);
\draw[dashed] (6,0)--++(0,-6);

\draw (0,1.5) node[left]{{\tiny $\frac{1}{2}G$}};
\draw (0,6) node[left]{{\tiny $\mathbb R\setminus \frac{1}{2}G$}};
\draw (.5,1.5) node{{\tiny$6$}};
\draw (5.5,1.5) node{{\tiny$6$}};

\draw (2,1.5) node{{\tiny$3$}};
\draw (2,5) node{{\tiny$6$}};
\draw (4,1.5) node{{\tiny$3$}};
\draw (4,5) node{{\tiny$6$}};
\draw (5.5,5) node{{\tiny $12$}};
\draw (.5,5) node{{\tiny $12$}};
\draw (7.5,1.5) node{{\tiny$9$}};
\draw (7,7) node{{\tiny$12$}};

\draw (7,9) node{{\tiny$C$}};
\draw (9.5,8) node{{\tiny$A$}};
\draw (10,6) node{{\tiny$B$}};

\draw (8.5,5.5) node{{\tiny$9$}};
\draw (10,4) node{{\tiny$6$}};
\end{tikzpicture}
\caption{The number of different orbits of lines depending on $(\gamma_1,\gamma_2)$. On the left picture we consider lines directed by $x^{i}$ with $i$ odd. The right picture shows the result for $i$ even.}
\label{fig-orbit-line-pair-impair}
\end{figure}
Next, we must consider the case of lines directed by $v=x^i$ with $i$ an odd number. Following the same arguments as in the even case, we obtain the analogous 4 conditions:

\begin{itemize}
\item Lines labelled $\pm (i+4)$ and $\pm(i+6)$ are in the same $\Delta_{0}$-orbit if and only if for some $\mu'\in\mathbb R$, it holds
$$
-\gamma_{2}x^{i_{0}+2}+\mu' x^{i_{0}}\in\mathbb Z[x]\iff\sqrt3\gamma_{2}\in G
$$

\item Lines labelled $i+4$ and $-(i+4)$ are in the same $\Delta_{0}$-orbit if and only if for some $\mu'\in\mathbb R$, it holds
$$
2(\gamma_{2}x^{i_{0}+4}+\gamma_{1}x^{i_{0}+1})+\mu'x^{i_{0}}\in\mathbb Z[x]\iff (2\sqrt3\gamma_{2}+2\gamma_{1})x-(4\gamma_{2}+\mu')\in\mathbb Z[x]
$$
and this last condition is equivalent to $2(\sqrt3\gamma_{2}+\gamma_{1})\in G$. This should be compared with set $D$ in Lemma~\ref{lem-intersection-ensemble}.
\item Lines labelled $i+6$ and $-(i+6)$ are in the same $\Delta_{0}$-orbit if and only if for some $\mu'\in\mathbb R$, it holds
$$
2(-\gamma_{2}x^{i_{0}}+\gamma_{1}x^{i_{0}+1})+\mu'x^{i_{0}}\in\mathbb Z[x]\iff 2\gamma_{1}x^{i_{0}+1}+(-2\gamma_2+\mu')x^{i_{0}}\in \mathbb Z[x]
$$
and this last condition is equivalent to $2\gamma_{1}\in G$. 
\item Lines labelled $\pm(i+4)$ and $\mp(i+6)$ are in the same $\Delta_{0}$-orbit if and only if for some $\mu'\in\mathbb R$, it holds
$$
\gamma_{2}\sqrt3x^{i_{0}+5}+2\gamma_{1}x^{i_{0}+1}+\mu'x^{i_{0}}\in\mathbb Z[x]\iff (\gamma_{2}\sqrt3+2\gamma_{1})x+(\mu'-3\gamma_{2})\in\mathbb Z[x]
$$
and this last condition is equivalent to $\gamma_{2}\sqrt3+2\gamma_{1}\in G$. This should be compared with set $E$ in Lemma~\ref{lem-intersection-ensemble}.
\end{itemize}
In the case of lines directed by an odd power of $x$ the `square' we obtain is:

\begin{center}
\begin{tikzpicture}
\draw[red] (0,0)--(3,0);
\draw[red] (0,3)--(3,3);
\draw[green] (0,0)--(3,3);
\draw[green] (3,0)--(0,3);
\draw[blue] (3,0)--(3,3);
\draw (0,0)--(0,3);
\fill (0,0) circle(.1);
\fill (3,0) circle(.1);
\fill (0,3) circle(.1);
\fill (3,3) circle(.1);
\draw (0,0) node[left]{$i+4$};
\draw (0,3) node[left]{$-(i+4)$};
\draw (3,0) node[right]{$i+6$};
\draw (3,3) node[right]{$-(i+6)$};
\draw (1.5,0) node[below]{$\sqrt 3\gamma_2$};
\draw (0,1.5) node[left]{$2(\sqrt3\gamma_2+\gamma_{1})$};
\draw (3,1.5) node[right]{$2\gamma_{1}$};
\draw (1.5,1.6) node[above]{$\gamma_{2}\sqrt3+2\gamma_{1}$};
\end{tikzpicture}
\end{center}

A summary of our findings on the number of orbits of lines directed by odd powers of $x$ is summarized in Figure~\ref{fig-orbit-line-pair-impair}.

Finally, to obtain $L_{1}$, the total number of orbits of lines under the action of $\Gamma$, we need to combine the information on the lines directed by odd and even powers of $x$. The considerations in Lemma~\ref{lem-intersection-ensemble} together with the analysis carried out in this proof yield the final result, which we have encapsulated in Figure~\ref{fig-orbit-tout}.  
\end{proof}

%%%%%%%%%%%%%%%%
\subsection{First cohomology group $H^1(\Omega_{E_{12}^\gamma})$}\label{sec-first-group}
%%%%%%%%%%%%%%%
In order to compute the rank of the group $H^1(\Omega_{E_{12}^\gamma})$ via Theorem~\ref{ref-rank-cohomologie} we need one more ingredient: the value of $R$, which is the rank of the $\mathbb Z$-module generated by $\Lambda^2\Gamma^i, i=1,\dots, n$. Recall that $\Gamma^i\leq\Gamma$ is the stabilizer under the action of $\Gamma$ of the vector space spanned by $f_i$. 

\begin{lemma}\label{Stab-delta}
For all $i=1,\dots, 6$, the rank 2 groups $\Gamma^{i}$ are explicit in the following chart:
$$
{\renewcommand\arraystretch{1.3}
\begin{array}{|c|c|c|}
\hline
\Gamma^1&\Gamma^2&\Gamma^3\\
\hline
\langle f_1, f_6-f_2\rangle&\langle f_2,f_1+f_3\rangle&\langle f_3, f_2+f_4\rangle\\
\hline
\end{array}}
$$
$$
{\renewcommand\arraystretch{1.3}
\begin{array}{|c|c|c|}
\hline
\Gamma^4&\Gamma^5&\Gamma^6\\
\hline
\langle f_4, f_3+f_5\rangle&\langle f_5, f_4+f_6\rangle &\langle f_6,f_1-f_5\rangle \\
\hline
\end{array}}
$$ 

\end{lemma}
\begin{proof}
Consider a line directed by the vector $f_i$. The group $\Gamma^{i}$ is the subgroup $\{h\in\Gamma\,|\, f_{i}+h=\mu f_{i}, \mu\in\mathbb R\}$. Thus we need to solve an equation of the type 
$$
\lambda f_i= h = n_1f_1+n_2f_2+n_3f_3+n_4f_4+n_5f_5+n_6f_6,
$$
where $\lambda\in\mathbb R$ and $ (n_1,\dots,n_6)\in\mathbb Z^6$. Using Equation~\eqref{eq-fx} the above expression can be rewritten in terms of roots of unity and our task is to understand which $\lambda\in\mathbb R$ can be written as $\lambda =\sum_k n_k x^k$. By Lemma~\ref{lem-algebre-simple} we deduce $\lambda\in \mathbb Z[\sqrt 3]$. 
The integer combinations of the $f_{i}$ which yield elements of the form $\mathbb Z[\sqrt3] f_{i} $ are collected in Lemma~\ref{rem-simplif-gamma}. 
The results in the table in the statement follow.
\end{proof}

\begin{lemma}\label{lem-R1-sansDelta}
Let $\beta$ be the map from Eq.~(\ref{beta-map}). For all parameters $\gamma$ we have $R = \rk \beta = 3$, and $\coker \beta \cong \mathbb Z^3$.
\end{lemma}
\begin{proof}
By definition, the image of the map $\beta$ depends only on the stabilizers $\Gamma^\alpha$ of the $1$-singularities $\alpha$, that is, on their
directions $f_i$, and not on their positions. Let us denote $\Lambda_i=\Lambda^2 \Gamma^i$ for $i=1,\dots, 6$, so that $\im \beta$ is generated
by the $\Lambda_i$. Using the explicit description of the groups $\Gamma^{i}$ in Lemma~\ref{Stab-delta} and the relations in Lemma~\ref{rem-simplif-gamma},
we compute the relevant exterior products:
$$
\begin{cases}
\Lambda_1=f_1\wedge f_6-f_1\wedge f_2=-2f_1\wedge f_2+f_1\wedge f_4,\\
\Lambda_2=-f_1\wedge f_2+f_2\wedge f_3,\\
\Lambda_3=-f_2\wedge f_3+f_3\wedge f_4,\\
\Lambda_4=f_1\wedge f_4-2f_3\wedge f_4,\\
\Lambda_5=2f_5\wedge f_4-f_5\wedge f_2=f_1\wedge f_2-2 f_1\wedge f_4+f_2\wedge f_3+2 f_3\wedge f_4, \\ 
\Lambda_6=2f_6\wedge f_1-f_6\wedge f_3=2 f_1\wedge f_2 -2 f_1\wedge f_4+ f_2\wedge f_3+f_3\wedge f_4.\\ 
\end{cases}
$$
In order to compute the rank of the $\mathbb Z$-module $\im \beta$ generated by the $\Lambda^2\Gamma^i$, we represent each of these exterior products,
with respect to the basis $f_i\wedge f_j$ given by the lexicographic order, as a column in the following matrix:
$$
\begin{pmatrix}
-2&-1&0&0&1&2\\
0&0&0&0&0&0\\
1&0&0&1&-2&-2\\
0&1&-1&0&1&1\\
0&0&0&0&0&0\\
0&0&1&-2&2&1
\end{pmatrix}.
$$
The quantity $R$ we seek is the rank of this matrix. It is evident that the first three columns are linearly independent. We have
$
\begin{cases}
\Lambda_4=\Lambda_{1}-2\Lambda_{2}-2\Lambda_3,\\ 
\Lambda_5=-2\Lambda_1+3\Lambda_2+2\Lambda_3,\\ 
\Lambda_6=-2\Lambda_1+2\Lambda_2
\end{cases}
$ 
and therefore $R=3$.

It is also easy to see that $-(\Lambda_2+\Lambda_3)$, $\Lambda_1-2(\Lambda_2+\Lambda_3)$ and $-\Lambda_3$
span a $\mathbb Z$-module of rank 3 with free complement in $\mathbb Z^6$. As the latter equals $\Lambda^2\Delta_0$ in our basis,
$\coker \beta$ is free of rank 3.

\end{proof}

\begin{proposition}\label{rank}
The rank of $H^1(\Omega_{E_{12}^\gamma})$ depends on $\gamma$ and equals the values depicted in Figure~\ref{fig-orbit-tout} plus 1.
\end{proposition}
\begin{proof}
By Theorem~\ref{ref-rank-cohomologie} the rank of $H^1$ can be computed as $4+L_1-R$. The value of $L_{1}$ as a function of $\gamma$ was computed in Lemma~\ref{orbits-line}. Finally, by the previous lemma we know that for all $\gamma$ we have $R=3$ and the result follows.
\end{proof}

\section{Intersection of lines}\label{sec-lines}

We are now left with the task of computing the rank of the second cohomology groups. To this end we need to compute the quantity $e$ in Theorem~\ref{ref-rank-cohomologie}. It is defined as a linear combination of the quantities   $L_0^\alpha$ and $L_0$ which depend on $\gamma$. We recall the reader that the latter quantity is the cardinality of the set of all orbits of 0-singularities, while the former is the cardinality of the set of orbits of $0$-singularities on the specific 1-singularity $\alpha$.

The 0-singularities are obtained intersecting the representatives of orbits of lines studied in Section~\ref{sect:lines-gamma}. From Proposition~\ref{orbits-line} we know that the number of orbits depends on the value of the parameter $\gamma$. In the next subsection we compute $L_0^\alpha$ and and $L_0$ in the case $\gamma=(0,0)$, which corresponds to having 6 different directions, or in other words, one orbit representative per direction. Following this, Subsection~\ref{2lines} describes all the cohomology groups of the generalized 12-fold tilings with the property of having at most two representatives of orbits of lines per direction. In the language of Proposition~\ref{orbits-line} this section deals with $L_{1}$ having value 6, 9 or 12. The corresponding values of $\gamma$ are the ones in the regions enclosed by a blue line in Figure~\ref{fig-orbit-tout}. These results come from a computed aided calculation. Finally, in the last subsection we address the general case. We have not obtained a complete description of the second cohomology groups in this case. However, again with a computed aided calculation, we present an example of the extremal case in Proposition~\ref{orbits-line} when we have 24 different line orbits to deal with. 

\subsection{First case: $\gamma=(0,0)$}\label{gammanul}
\begin{lemma}\label{lem-calc-droites-gamma-nul}
If $\gamma=(0,0)$, then $L_0=14$ and for each 1-singularity $\alpha$ we have $L_0^\alpha=6$.  
\end{lemma}
\begin{proof}
First of all we show that $L_0^\alpha=6$ for each 1-singularity $\alpha$. For a fixed line $\alpha=\mathbb R f_i$, with $i=1,\dots,6$, we are interested in the intersections of this line with the orbits of the other lines under the action of $\Gamma$. Using the bijection described in Equation \ref{eq-fx} we can write the relevant equations in terms of powers of the root of unity $x$ instead of the vectors $f_{i}$. The generic formula we have to deal with sets a point in a line of direction $x^{i}$, or rather points in a $\Gamma^i$-orbit, equal a point in a line of direction $x^{j}$. In symbols we have, for $i\neq j\in \{1,\dots,6\}$:
\[
\mu x^i+\sum n_k x^k=\lambda x^j, n_k\in\mathbb Z,\ \mu,\lambda\in\mathbb R 
\iff 
\lambda= \mu x^{i-j}+\sum_k n_k x^k.
\]
It follows that  $\lambda$ is of the form $cx+d+\mu (ax+b)$ with $a, b, c, d \in G$. This gives us the conditions
\begin{equation}\label{lambda}
\begin{cases} 
\lambda=d+\mu b\\
-c=\mu a
\end{cases} 
\implies
\lambda=d-\frac{b}{a}c.
\end{equation}
Inside the infinite collection of points of intersection between a line of direction $x^{j}$ and the translates by $\Gamma$ of a line of direction $x^{i}$, determined in this setting by the possible values of $\lambda$, we need to count how many are in different orbits. That is, the points $\lambda$ and $\lambda'$ will be identified if $\lambda-\lambda'\in G$.

Since $x^{i-j}=a+xb$, by Lemma~\ref{degree1} when $i-j$ varies we obtain five possible values for $-\frac{b}{a}$ which are collected in the following set: 
$$
S=\{0, -\frac{1}{\sqrt 3}, \frac{-\sqrt 3}{2}, \frac{-2}{\sqrt 3}, -\sqrt 3\}.
$$
Keeping the notation $\lambda$ to refer to its class modulo $G$, from \eqref{lambda} we see that we need to understand how many different equivalence classes are there for numbers of the form $\lambda=\frac{b}{a}(c_{1}+c_{2}\sqrt 3)$ where $c_{1},c_{2}\in\mathbb Z$ and $\frac{b}{a}\in S$. The complete list is the following
$$
T=\{0,  \frac{1}{\sqrt 3},  \frac{2}{\sqrt 3}, \frac{1}{2}, \frac{\sqrt{3}}{2},  \frac{1+\sqrt{3}}{2}\},
$$
from which we deduce that $L_0^\alpha=6$ independently of the value of $\alpha$, that is, of the direction of the 1-singularity.

We have identified on each line 6 different equivalence classes of points. However, some of these classes in different lines can be related by an element of $\Gamma$. For example, the class of 0 belongs to every 1-singularity so there are at most 31 different classes. The final quantity we need to compute, $L_{0}$, is precisely this total number of classes of 0-singularities under the action of $\Gamma$. 

Notice that the classes of 0-singularities in the line $x^{i}$ which are in the same equivalence classes of 0-singularities in the line $x^{j}$ only depend on the difference $i-j$. Indeed, we have to solve an equation of the form
$$
\lambda_{1}x^{i}-\lambda_{2}x^{j}\in\Gamma \iff \lambda_{1}x^{i}-\lambda_{2}x^{j}=a+bx
$$
with $a,b\in G$ and $\lambda_{1},\lambda_{2}\in T$,   
which is equivalent to $\lambda_{1}x^{i-j}-\lambda_{2}=a'+b'x$ for some $a',b'\in G$. With the help of Lemma~\ref{degree1} we see that the line of direction $x^{0}$ and the line directed by $x^{1}$ only share one equivalence class of 0-singularities, that of 0. So we have that $L_{0}\geq 11$. Comparing the equivalence classes between the lines directed by $x^{0}$ and $x^{2}$ we obtain:
$$
\lambda_{1}x^{2}-\lambda_{2}=a+bx\iff \lambda_{1}(x\sqrt3-1)-\lambda_{2}=a+bx\implies
\begin{cases}
\lambda_{1}\in\frac{1}{\sqrt3}G\\
\lambda_{2}\in -\lambda_{1}+G
\end{cases}.
$$
It follows that the class of $\frac{1}{\sqrt3}$ and the class of $\frac{2}{\sqrt 3}$ are in the same $\Gamma$ orbit when considered in lines directed by $x^{i}$ and $x^{i-2}$. From here we conclude that $L_{0}\geq 14$ since we have the eleven 0-singularities in different orbits from the lines directed by $x^{0}$ and $x^{1}$ plus three more in the line directed by $x^{2}$.

An equivalent computation to the one shown above, comparing the lines directed by $x^{3}$ and $x^{0}$ and by $x^{3}$ and $x^{1}$ shows that there are no more orbits of 0-singularities and gives a complete description. Indeed, it turns out that the line directed by $x^{3}$ has four orbits in common with the one directed by $x^{0}$; while it has three orbits in common with the line directed by $x^{1}$. Figure \ref{points-gamma-nul} summarizes the the result: on each line of direction $x^{i}$ there are six marked points and points in different lines with the same color are in the same orbit under the action of $\Gamma$. We conclude that $L_0=14$. 
\end{proof}

\begin{figure}
\begin{tikzpicture}
\draw (0,0)--(6,0);
\draw (0,0)--(30:6);
\draw (0,0)--(60:6);
\draw (0,0)--(90:6);
\draw (0,0)--(120:6);
\draw (0,0)--(150:6);
\draw (30:1) circle(1mm);
\draw (150:1) circle(1mm);
\draw (90:4) circle(1mm);

\draw[pink] (00:1) circle(1mm);
\draw[pink] (60:4) circle(1mm);
\draw[pink] (120:1) circle(1mm);

\draw[red] (00:2) circle(1mm);
\draw[red] (90:5) circle(1mm);

\fill (00:3) circle(1mm);
\fill (90:3) circle(1mm);

\draw[green] (00:4) circle(1mm);
\draw[green] (60:1) circle(1mm);
\draw[green] (120:4) circle(1mm);

\fill[blue] (00:5) circle(1mm);
\fill[blue] (90:2) circle(1mm);

\fill[gray] (30:2) circle(1mm);
\fill[gray] (120:5) circle(1mm);

\fill[yellow] (30:3) circle(1mm);
\fill[yellow] (120:3) circle(1mm);

\fill[purple] (30:4) circle(1mm);
\fill[purple] (90:1) circle(1mm);
\fill[purple] (150:4) circle(1mm);

\draw[dotted] (30:5) circle(1mm);
\draw[dotted] (120:2) circle(1mm);

\draw[] (00:7) node{$x^{0}$};
\draw[] (30:7) node{$x^{1}$};
\draw[] (60:7) node{$x^2$};
\draw[] (90:7) node{$x^3$};
\draw[] (120:7) node{$x^4$};
\draw[] (150:7) node{$x^5$};

\draw (0,-.5) node{$0$};
\draw (1,-.5) node{$\frac{1}{\sqrt 3}$};
\draw (2,-.5) node{$\frac{\sqrt 3}{2}$};
\draw (3,-.5) node{$\frac{1+\sqrt 3}{2}$};
\draw (4,-.5) node{$\frac{2}{\sqrt 3}$};
\draw (5,-.5) node{$\frac{1}{2}$};

\draw[] (60:2) node{$A$};
\draw[] (150:5) node{$A$};

\draw[] (60:3) node{$C$};
\draw[] (150:3) node{$C$};

\draw[] (60:5) node{$B$};
\draw[] (150:2) node{$B$};

\draw (1,0) arc (0:150:1);
\draw (2,0) arc (0:150:2);
\draw (3,0) arc (0:150:3);
\draw (4,0) arc (0:150:4);
\draw (5,0) arc (0:150:5);
\end{tikzpicture}
\caption{Orbits of points if $\gamma\in\Delta$}\label{points-gamma-nul}
\end{figure}
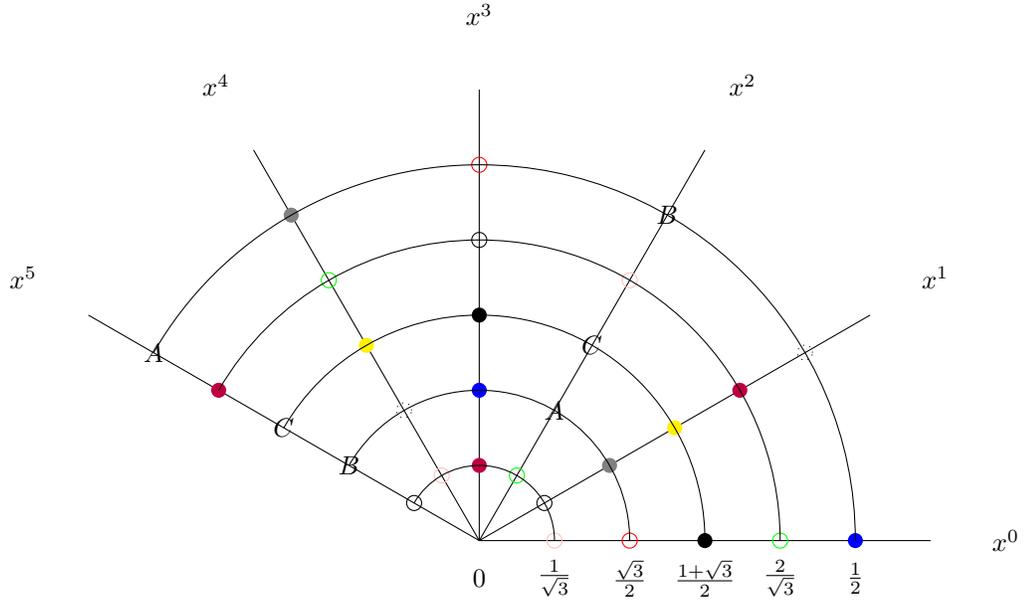

\subsection{All cases with up to two lines per direction}\label{2lines}

In this section, we summarize the results of some of the simpler cases, namely
all those with up to two $1$-singularities per direction. For this, we need to
determine all values of $\gamma$ giving rise to these cases, and then determine
for each of them the values of $L_0$, $L_1$, and $L_0^\alpha$ for all line orbits
$\alpha$. As $R=3$ and $\coker \beta$ is free independently of $\gamma$, we have by Theorem~\ref{ref-rank-cohomologie}:
$$
H^2(\Omega_{E_{12}^{\gamma}}) = \mathbb Z^{L_1+e},\quad H^1(\Omega_{E_{12}^{\gamma}}) = \mathbb Z^{1+L_1}\quad\mathrm{and}\quad H^0(\Omega_{E_{12}^{\gamma}})=\mathbb Z,
$$
where $e=-L_0+\sum_{\alpha \in I_1} L_0^\alpha$. 

For each value of $\gamma$ involved,
we have determined the combinatorics of line intersections with the help of a
computer program, which is implemented in the GAP language \cite{GAP4}.
It builds upon the GAP package Cryst \cite{Cryst} and is made available
for reference \cite{FGweb}.
In order to be able to connect the number of orbits
$L_0^\alpha$ of points on the individual lines $\alpha$ to the total number
of point orbits $L_0$, we have split the numbers $L_0^\alpha$ as
$L_0^\alpha=\sum_p L_{0,p}^\alpha$, where $L_{0,p}^\alpha$ is the number of those
orbits whose points are intersections of exactly $p$ 1-singularities. Correcting for
double counting, we obtain $L_0=\sum_{\alpha \in I_1} L_{0,p}^\alpha/p$.

These values are tabulated in each case. In these tables, we list $L_{0,p}$ for each \emph{type of line}. For example, the symbol $L^{1}_{0,p}$ indicates lines with combinatorics of ``type 1'', the symbol $L^{2}_{0,p}$ lines with combinatorics of ``type 2'' etc. The superscripts do not stand for directions or any other property of the lines, but they collect lines with the same type of combinatorics. The column $n$ in the table indicates the number of lines of a certain type. The column dir specifies whether lines of this type have even or odd directions, or both.
In two further lines in the tables, labelled $\sum L_{0,p}^\alpha$ and $ L_{0,p}$,  we give the sum of these data over all
line types, and the corresponding data for the total orbits. The last column
contains the sum over all values of $p$.

\textbf{Case 1.} We start analyzing $\gamma_1,\gamma_{2} \in G$, which is equivalent
to $\gamma=(0,0)$. This is the only case with just one line in each of the six directions. The explicit computations of $L_{0}$ and $L_{0}^{\alpha}$ in this case were carried out in Section~\ref{gammanul}. The intersection combinatorics in the notation introduced in this section are given in the following table:
\[ 
{\renewcommand\arraystretch{1.3}
\begin{array}{|c|c|c||c|c|c|c|c||c|}
  \hline
                    & n & \text{dir} & p=2 & p=3 & p=4 & p=5 & p=6 & \text{tot} \\
  \hline
  L_{0,p}^1          & 6 & e,o        &   3 &   2 &   0 &   0 &   1 &   6 \\
  \hline
  \sum L_{0,p}^\alpha & 6 &            &  18 &  12 &   0 &   0 &   6 &  36 \\
  \hline
  L_{0,p}            &   &            &   9 &   4 &   0 &   0 &   1 &  14 \\
  \hline
\end{array}} 
\]

We see that there is only one line type, containing one orbit with six-fold
intersections, two orbits with three-fold intersections, and three orbits
with two-fold intersections. These combinatorics should be compared with Figure~\ref{points-gamma-nul}: all lines, with even or odd direction, have the same combinatorics. There is precisely one orbit point, with representative the origin, which is a six-fold intersection. In each of the lines in Figure~\ref{points-gamma-nul} there are 3 symbols that appear exactly two times in the figure, those are the orbits of two-fold intersections; and there are two symbols that appear exactly three times in the figure, those are the orbits of three-fold intersections. Putting everything together, we get the following
end result:
\begin{proposition}[cf.\ Theorem~\ref{known}.3]\label{allgamma0} 
The ranks of the cohomology groups of the 12-fold tiling, together with the quantities $L_0^\alpha, L_0, L_1$ and $e$, are given by:
\[ 
{\renewcommand\arraystretch{1.3}
\begin{array}{|c|c|c|c|c|c|c|}
  \hline
  \sum L_0^\alpha & L_0 & L_1 &   e & \rk H^2(\Omega_{E_{12}}) & \rk H^1(\Omega_{E_{12}}) & \rk H^0(\Omega_{E_{12}}) \\
  \hline
             36 &  14 &   6 &  22 &      28 &       7 &      1 \\
  \hline
\end{array} } 
\]
\end{proposition}

\textbf{Case 2.} We now determine the intersection combinatorics under the assumption that there is only one line in each even direction and two lines per odd direction. From the proof of Proposition~\ref{orbits-line} (see particularly Figure~\ref{fig-orbit-line-pair-impair}) we see that this situation can arise in two ways.
\begin{itemize}
\item The first possibility is 
$\gamma_1 \in G$ and $\gamma_2 \in\frac{1}{2} G\setminus G$. All $\gamma$-values
in this set give the same intersection combinatorics, which are given in the following table:

\[ {\renewcommand\arraystretch{1.3}
\begin{array}{|c|c|c||c|c|c|c|c||c|}
  \hline
                    & n & \text{dir} & p=2 & p=3 & p=4 & p=5 & p=6 & \text{tot} \\
  \hline
  L_{0,p}^1          & 3 & e          &  10 &   0 &   0 &   2 &   0 &  12 \\
  \hline
  L_{0,p}^2          & 3 & o          &   4 &   5 &   0 &   1 &   0 &  10 \\
  \hline
  L_{0,p}^3          & 3 & o          &   2 &   4 &   0 &   2 &   0 &   8 \\
  \hline
  \sum L_{0,p}^\alpha & 9 &            &  48 &  27 &   0 &  15 &   0 &  90 \\
  \hline
  L_{0,p}            &   &            &  24 &   9 &   0 &   3 &   0 &  36 \\
  \hline
\end{array} } \]

We see here that we have two line types in each odd direction, with different
intersection combinatorics, but only one line type in even directions.
Putting these data together, we obtain:

\begin{proposition} 
The ranks of the cohomology groups of the generalized 12-fold tilings with parameter $\gamma=(\gamma_{1},\gamma_{2})$ satisfying $\gamma_{1}\in G$ and $\gamma_{2}\in\frac{1}{2}G\setminus G$, together with the quantities $L_0^\alpha, L_0, L_1$ and $e$, are given by:
\[ 
{\renewcommand\arraystretch{1.3}
\begin{array}{|c|c|c|c|c|c|c|}
  \hline
  \sum L_0^\alpha & L_0 & L_1 &   e & \rk H^2(\Omega_{E_{12}^\gamma}) & \rk H^1(\Omega_{E_{12}^\gamma}) & \rk H^0(\Omega_{E_{12}^\gamma}) \\
  \hline
             90 &  36 &   9 &  54 &     63  &      10 &       1 \\
  \hline
\end{array} } 
\]
\end{proposition}

\item The second possibility occurs for $\gamma_1 \in\frac{1}{\sqrt3} G \setminus G$ and $\gamma_2 \in G$.
For all $\gamma$-values in this set, we obtain:

\[ {\renewcommand\arraystretch{1.3}
\begin{array}{|c|c|c||c|c|c|c|c||c|}
  \hline
                    & n & \text{dir} & p=2 & p=3 & p=4 & p=5 & p=6 & \text{tot} \\
  \hline
  L_{0,p}^1          & 3 & e          &   6 &   1 &   0 &   0 &   2 &   9 \\
  \hline
  L_{0,p}^2          & 6 & o          &   9 &   2 &   0 &   0 &   1 &  12 \\
  \hline
  \sum L_{0,p}^\alpha & 9 &            &  72 &  15 &   0 &   0 &  12 &  99 \\
  \hline
  L_{0,p}            &   &            &  36 &   5 &   0 &   0 &   2 &  43 \\
  \hline
\end{array} } \]

Here, for each odd direction, we have two lines of the same type. These two
lines are related by symmetry - they are mirror images of each other.
The cohomology of this case is slightly different:

\begin{proposition} 
The ranks of the cohomology groups of the generalized 12-fold tilings with parameter $\gamma=(\gamma_{1},\gamma_{2})$ satisfying $\gamma_{1}\in\frac{1}{\sqrt3}G\setminus G$ and $\gamma_{2}\in G$, together with the quantities $L_0^\alpha, L_0, L_1$ and $e$, are given by:
\[ {\renewcommand\arraystretch{1.3}
\begin{array}{|c|c|c|c|c|c|c|}
  \hline
  \sum L_0^\alpha & L_0 & L_1 &   e & \rk H^2(\Omega_{E_{12}^\gamma}) & \rk H^1(\Omega_{E_{12}^\gamma}) & \rk H^0(\Omega_{E_{12}^\gamma}) \\
  \hline
             99 &  43 &   9 &  56 &     65  &      10 &       1 \\
  \hline
\end{array} } 
\]
\end{proposition}
\end{itemize}

Of course, there is also the possibility of one line per odd direction,
and two lines per even direction, but this is completely analogous to
the above two cases, and we do not list them separately. This symmetry is reflected in Figure~\ref{fig-orbit-tout} by the symbolic green diagonal. 
\vspace{0.2cm}

We now turn to the situation where we have two lines in each even and each odd direction.

\textbf{Case 3.} We start with the case where the two lines in even directions are mirror images of
each other. There are then three subcases. 
\begin{itemize}
\item
  In the first subcase, both $\gamma_1$ and $\gamma_2$ are contained in
  $\frac{1}{\sqrt3}G \setminus G$ (cf. proof of Proposition~\ref{orbits-line}
  and Figure~\ref{fig-orbit-tout}), in which case also the lines in odd directions
  form mirror pairs. The corresponding $\gamma$-values then all lead to the same
  intersection combinatorics as follows:

\[ {\renewcommand\arraystretch{1.3}
\begin{array}{|c|c|c||c|c|c|c|c||c|}
  \hline
                    & n  & \text{dir} & p=2 & p=3 & p=4 & p=5 & p=6 & \text{tot} \\
  \hline
  L_{0,p}^1          & 12 & e,o        &  12 &   1 &   0 &   0 &   2 &  15 \\
  \hline
  \sum L_{0,p}^\alpha & 12 &            & 144 &  12 &   0 &   0 &  24 & 180 \\
  \hline
  L_{0,p}            &    &            &  72 &   4 &   0 &   0 &   4 &  80 \\
  \hline
\end{array} } \]

We see here that in this case, all even and odd lines are of the same type, and are
in fact all symmetry equivalent. This intersection combinatorics results in the following cohomology:

\begin{proposition} 
The ranks of the cohomology groups of the generalized 12-fold tilings with parameter $\gamma=(\gamma_{1},\gamma_{2})$ satisfying $\gamma_{1},\gamma_{2}\in\frac{1}{\sqrt3}G\setminus G$, together with the quantities $L_0^\alpha, L_0, L_1$ and $e$, are given by:
\[ 
{\renewcommand\arraystretch{1.3}
\begin{array}{|c|c|c|c|c|c|c|}
  \hline
  \sum L_0^\alpha & L_0 & L_1 &   e & \rk H^2(\Omega_{E_{12}^\gamma}) & \rk H^1(\Omega_{E_{12}^\gamma}) & \rk H^0(\Omega_{E_{12}^\gamma}) \\
  \hline
            180 &  80 &  12 & 100 &    112  &      13 &       1 \\
  \hline
\end{array} } 
\]
\end{proposition}

\item
  In the second subcase, we have $\gamma_1 \in \frac{1}{\sqrt3}G \setminus G$, but
  now $\gamma_2 \not\in \frac{1}{2\sqrt3}G$. More precisely, $\gamma_2$ depends on
  $\gamma_1$, and must satisfy $2 \gamma_1 + \sqrt{3} \gamma_2 \in G$, which corresponds
  to region $E$ in Figure~\ref{fig-orbit-line-pair-impair} (cf. proofs of Proposition~\ref{orbits-line}
  and Lemma~\ref{lem-intersection-ensemble}). All $\gamma$-values satisfying the conditions in this second subcase lead to the following intersection combinatorics:

\[ {\renewcommand\arraystretch{1.3}
\begin{array}{|c|c|c||c|c|c|c|c||c|}
  \hline
                    & n  & \text{dir} & p=2 & p=3 & p=4 & p=5 & p=6 & \text{tot} \\
  \hline
  L_{0,p}^1          &  6 & e          &  12 &   0 &   0 &   3 &   0 &  15 \\
  \hline
  L_{0,p}^2          &  6 & o          &  10 &   3 &   0 &   2 &   0 &  15 \\
  \hline
  \sum L_{0,p}^\alpha & 12 &            & 132 &  18 &   0 &  30 &   0 & 180 \\
  \hline
  L_{0,p}            &    &            &  66 &   6 &   0 &   6 &   0 &  78 \\
  \hline
\end{array} } \]

Here, the lines in even and odd directions have different intersection combinatorics,
but in each direction they form mirror pairs. This intersection combinatorics results
in the following cohomology:

\begin{proposition} 
  The ranks of the cohomology groups of the generalized 12-fold tilings with parameter
  $\gamma=(\gamma_{1},\gamma_{2})$ satisfying $\gamma_{1}\in\frac{1}{\sqrt3}G\setminus G$, $\gamma_2 \not\in \frac{1}{2\sqrt3}G$ and $2 \gamma_1 + \sqrt{3} \gamma_2 \in G$, together with the quantities
  $L_0^\alpha, L_0, L_1$ and $e$, are given by:
\[ 
{\renewcommand\arraystretch{1.3}
\begin{array}{|c|c|c|c|c|c|c|}
  \hline
  \sum L_0^\alpha & L_0 & L_1 &   e & \rk H^2(\Omega_{E_{12}^\gamma}) & \rk H^1(\Omega_{E_{12}^\gamma}) & \rk H^0(\Omega_{E_{12}^\gamma}) \\
  \hline
            180 &  78 &  12 & 102 &    114  &      13 &       1 \\
  \hline
\end{array} } 
\]
\end{proposition}

We note that this case also occurs with the roles of even and odd directions
interchanged, which we do not list separately.

\item In the third subcase, the two lines in an odd direction are not symmetry equivalent,
  which is the case if $\gamma_1 \in G$ and
  $\gamma_2 \in\frac{1}{2\sqrt3}G \setminus (\frac{1}{\sqrt3}G \cup \frac{1}{2}G)$
  (again this characterization of the parameters comes from the proof of
  Proposition~\ref{orbits-line}). For all these $\gamma$-values, we obtain:

\[ {\renewcommand\arraystretch{1.3}
\begin{array}{|c|c|c||c|c|c|c|c||c|}
  \hline
                    &  n & \text{dir} & p=2 & p=3 & p=4 & p=5 & p=6 & \text{tot} \\
  \hline
  L_{0,p}^1          &  6 & e          &  16 &   0 &   0 &   2 &   0 &  18 \\
  \hline
  L_{0,p}^2          &  3 & o          &   8 &   4 &   0 &   2 &   0 &  14 \\
  \hline
  L_{0,p}^3          &  3 & o          &   4 &   2 &   0 &   4 &   0 &  10 \\
  \hline
  \sum L_{0,p}^\alpha & 12 &            & 132 &  18 &   0 &  30 &   0 & 180 \\
  \hline
  L_{0,p}            &    &            &  66 &   6 &   0 &   6 &   0 &  78 \\
  \hline
\end{array} } \]

Here, we have two types of odd lines with different intersection combinatorics, but
overall, this results in the same cohomology ranks as above:

\begin{proposition} 
  The ranks of the cohomology groups of the generalized 12-fold tilings with parameter
  $\gamma=(\gamma_{1},\gamma_{2})$ satisfying $\gamma_{1}\in G$ and
  $\gamma_2 \in\frac{1}{2\sqrt3}G \setminus (\frac{1}{\sqrt3}G \cup \frac{1}{2}G)$,
  together with the quantities $L_0^\alpha, L_0, L_1$ and $e$, are given by:
\[ {\renewcommand\arraystretch{1.3}
\begin{array}{|c|c|c|c|c|c|c|}
  \hline
  \sum L_0^\alpha & L_0 & L_1 &   e & \rk H^2(\Omega_{E_{12}^\gamma}) & \rk H^1(\Omega_{E_{12}^\gamma}) & \rk H^0(\Omega_{E_{12}^\gamma}) \\
  \hline
            180 &  78 &  12 & 102 &    114  &      13 &       1 \\
  \hline
\end{array} } 
\]
\end{proposition}
\end{itemize}

Of course, this latter case can also occur with the roles of even and odd directions
interchanged, but we do not list it separately.

\textbf{Case 4:} Finally, we can have two line types in both even and odd directions,
which occurs if both $\gamma_1$ and $\gamma_2$ are contained in $\frac{1}{2}G\setminus G$
(cf. proof of Proposition~\ref{orbits-line} and Figure~\ref{fig-orbit-tout}). This is the first time in our analysis in which the $G$-action on the set of 0-singularities, that is points of intersection between representatives of lines, has more than one orbit. Indeed, the set of parameters $\frac{1}{2}G\setminus G$ splits in three $G$-cosets, so that altogether we
need to analyze the behavior of $\gamma$-values $(\gamma_1,\gamma_2)$ from nine different $(G \times G)$-cosets. The analysis of these 9 sets gives rise to only two different intersection combinatorics. The simpler combinatorics is
obtained for the $\gamma$-values in the cosets with representatives
$\{\frac12(1,\sqrt3),\frac12(\sqrt3,1),\frac12(1+\sqrt3,1+\sqrt3)\}$. These can be characterised
as those cosets for which $\gamma_1+\gamma_2 \in \frac{\sqrt3 + 1}{2}G$. The
intersection combinatorics then is:

\[ {\renewcommand\arraystretch{1.3}
\begin{array}{|c|c|c||c|c|c|c|c||c|}
  \hline
                    & n  & \text{dir} & p=2 & p=3 & p=4 & p=5 & p=6 & \text{tot} \\
  \hline
  L_{0,p}^1          & 12 & e,o        &    6 &   4 &   0 &   0 &   2 &   12 \\
  \hline
  \sum L_{0,p}^\alpha & 12 &            &   72 &  48 &   0 &   0 &  24 &  144 \\
  \hline
  L_{0,p}            &    &            &   36 &  16 &   0 &   0 &   4 &   56 \\
  \hline
\end{array} } \]

Here, we have again a single type of lines, and the cohomology becomes:

\begin{proposition} 
  The ranks of the cohomology groups of the generalized 12-fold tilings with parameter
  $\gamma=(\gamma_{1},\gamma_{2})$ satisfying $\gamma_{1}, \gamma_{2} \in \frac12 G \setminus G$
  and $\gamma_1 + \gamma_2 \in \frac{\sqrt3 + 1}{2} G$,
  together with the quantities $L_0^\alpha, L_0, L_1$ and $e$, are given by:
\[ {\renewcommand\arraystretch{1.3}
\begin{array}{|c|c|c|c|c|c|c|}
  \hline
  \sum L_0^\alpha & L_0 & L_1 &   e & \rk H^2(\Omega_{E_{12}^\gamma}) & \rk H^1(\Omega_{E_{12}^\gamma}) & \rk H^0(\Omega_{E_{12}^\gamma}) \\
  \hline
            144 &  56 &  12 &  88 &    100  &      13 &       1 \\
  \hline
\end{array} } \]
\end{proposition}

For the remaining six cosets of $\gamma$-values, the combinatorics is:

\[ {\renewcommand\arraystretch{1.3}
\begin{array}{|c|c|c||c|c|c|c|c||c|}
  \hline
                    & n  & \text{dir} & p=2 & p=3 & p=4 & p=5 & p=6 & \text{tot} \\
  \hline
  L_{0,p}^1          &  6 & e,o        &   18 &   0 &   2 &   0 &   0 &   20 \\
  \hline
  L_{0,p}^2          &  6 & e,o        &   12 &   0 &   4 &   0 &   0 &   16 \\
  \hline
  \sum L_{0,p}^\alpha & 12 &            &  180 &   0 &  36 &   0 &   0 &  216 \\
  \hline
  L_{0,p}            &    &            &  90  &   0 &   9 &   0 &   0 &   99 \\
  \hline
\end{array} } \]

Here, we have two line types, each of which occurs both in even and odd directions. 
This results in the following cohomology:

\begin{proposition}\label{last2orbits} 
  The ranks of the cohomology groups of the generalized 12-fold tilings with parameter
  $\gamma=(\gamma_{1},\gamma_{2})$ satisfying $\gamma_{1}, \gamma_{2} \in \frac12 G \setminus G$
  and $\gamma_1 + \gamma_2 \not \in \frac{\sqrt3 + 1}{2} G$,
  together with the quantities $L_0^\alpha, L_0, L_1$ and $e$, are given by:
\[ {\renewcommand\arraystretch{1.3}
\begin{array}{|c|c|c|c|c|c|c|}
  \hline
  \sum L_0^\alpha & L_0 & L_1 &   e & \rk H^2(\Omega_{E_{12}^\gamma}) & \rk H^1(\Omega_{E_{12}^\gamma}) & \rk H^0(\Omega_{E_{12}^\gamma}) \\
  \hline
            216 &  99 &  12 & 117 &    129  &      13 &       1 \\
  \hline
\end{array} } \]
\end{proposition}

\subsection{General case: preliminary computations}\label{general1}

We now turn our efforts into computing the numbers $L_{0}$ and $L_{0}^{\alpha}$ in terms of $\gamma=(\gamma_{1},\gamma_{2})$. The strategy is the same as the one described for the case $\gamma=(0,0)$: we start by intersecting a fixed line from the 24 in Corollary~\ref{cor-simplification} with all the $\Gamma$-translates of the other 20 non-parallel lines. The computation and description of these intersections is the whole content of this subsection. The results obtained, which the reader might want to look at first, are collected and stated in the next section, concretely in Proposition \ref{prop-points-droite}.

Our first observation is that each family $(1, x^i)$ with $i>0$ can be a basis of $F^{\perp}$ (see Figure~\ref{fig-complexe}). For future reference, we collect in the following table the expressions of the powers of $x$ in these basis (cf.\ Lemma~\ref{degree1}):
$$
{\renewcommand\arraystretch{1.3}
\begin{array}{|c|c|c|c|c|}
\hline
&(1,x^2)&(1,x^3)&(1,x^4)&(1,x^5)\\
\hline
x&\frac{1}{\sqrt 3}(1+x^2)&\frac{x^3}{2}+\frac{\sqrt 3}{2}&\frac{2}{\sqrt 3}+\frac{1}{\sqrt 3}x^4&\sqrt 3+ x^5\\
\hline
x^2&&\frac{1}{2}+x^2\frac{\sqrt 3}{2}&1+x^4&2+\sqrt 3x^5\\
\hline
x^3&\frac{1}{\sqrt 3}(2x^2-1)&&\frac{\sqrt 3}{3}+\frac{2}{\sqrt 3} x^4&\sqrt 3+2x^5\\
\hline
x^4&x^2-1&-\frac{1}{2}+x^3\frac{\sqrt 3}{2}&&1+\sqrt 3 x^5\\
\hline
x^5&\frac{1}{\sqrt 3}(x^2-2)&\frac{x^3}{2}-\frac{\sqrt 3}{2}&-\frac{\sqrt 3}{3}+\frac{1}{\sqrt 3}x^4&\\
\hline
\end{array}}
$$

\begin{definition}
In order to simplify the following results we define the sets $A_{3}=\{0,\frac{\sqrt3}{3},\frac{2\sqrt3}{3}\}$ and $A_{4}=\{0,\frac{1}{2},\frac{\sqrt3}{2},\frac{1+\sqrt3}{2}\}$. The notation $x+A_3$ stands for the three values $x+y$ with $y\in A_3$.
\end{definition}

We consider two generic lines given in Corollary~\ref{cor-simplification} of directions $x^i, x^k$ with $k>i$. There are three possible choices and we will deal with them separately: $i, k$ are odd, $i, k$ are even or $i, k$ have different parity.

\subsubsection{Two even lines}

We pick two lines directed by $x^{i}$ and compute their intersection with the translates by $\Gamma$ of the four different lines directed by $x^k$ when $i,k$ are both even:

$$\frac{\gamma_1 }{\sqrt3}
\begin{cases} 
x^{i}\\ 
x^{i+2}
\end{cases}\!\!
+\frac{\gamma_2}{\sqrt3} x^{i+1}+\lambda x^i
= 
\pm(\frac{\gamma_1 }{\sqrt3}
\begin{cases} x^{k}\\ 
x^{k+2}
\end{cases}\!\!
+\frac{\gamma_2}{\sqrt3} x^{k+1})+\mu x^k+\sum n_px^p.
$$
We simplify by $x^i$ and obtain
\begin{equation}\label{lambda-pair}
\frac{\gamma_1 }{\sqrt3}
\begin{cases} 1 \\ 
x^{2}
\end{cases}\!\!
+\frac{\gamma_2}{\sqrt3} x+\lambda 
= 
\pm(\frac{\gamma_1 }{\sqrt3}
\begin{cases} 
x^{k-i}\\ 
x^{k+2-i}
\end{cases}\!\!
+\frac{\gamma_2}{\sqrt3} x^{k+1-i})+\mu x^{k-i}+\sum n_px^{p-i}.
\end{equation}
Recall that, by Lemma~\ref{degree1}, the number $\sum n_px^{p-i}$ is equal to $ax+b$ with $a, b \in G=\mathbb Z[\sqrt 3]$. We consider the basis $(1,x^{k-i})$ for $F^{\perp}$ and we compute the value of $\lambda$ using the array at the beginning of the section. Notice that because of our choice of basis we can compute directly $\lambda$ without worrying about $\mu$. There are two subcases to consider: 
\begin{itemize}
\item if $k-i=2$, then subbing the information in \eqref{lambda-pair} we obtain, considering only the first coordinate in the basis $(1,x^{2})$:
$$
\frac{\gamma_1 }{\sqrt3}
\begin{cases} 
1 \\ 
0
\end{cases}
+\frac{\gamma_2}{\sqrt3} \sqrt 3/3+\lambda 
= 
\pm(\frac{\gamma_1 }{\sqrt3}
\begin{cases} 
0\\ 
-1
\end{cases}
+\frac{\gamma_2}{\sqrt3} (-\frac{\sqrt3}{3}))+a\frac{\sqrt 3}{3}+b,
$$
which implies
$$
\lambda=\pm(\frac{\gamma_1 }{\sqrt3}
\begin{cases} 
0\\ 
-1
\end{cases}\!\!
-\frac{\gamma_{2}}{3})+a\frac{\sqrt 3}{3}+b-\frac{\gamma_1 }{\sqrt3}
\begin{cases} 
1 \\ 
0
\end{cases}\!\!
-\frac{\gamma_2}{3}.
$$
The $8$ families of values for $\lambda$, each corresponding to the intersection of one of the two lines directed by $x^{i}$ with one of the four lines directed by $x^{k}$ with $k-i=2$, are collected in the following accolades. On the left hand side we have the intersection with the line passing through the point $\frac{\gamma_{1}}{\sqrt3}x^{i}+\frac{\gamma_2}{\sqrt3} x^{i+1}$ and on the right hand side passing through the point $\frac{\gamma_{1}}{\sqrt3}x^{i+2}+\frac{\gamma_2}{\sqrt3} x^{i+1}$.
$$
\lambda=
\begin{cases}
\frac{-2\gamma_{2}}{3}+\frac{a\sqrt 3}{3}-\frac{\gamma_{1}}{\sqrt3}\\
\frac{-2\gamma_{2}}{3}+\frac{a\sqrt 3}{3}-\frac{2\gamma_{1}}{\sqrt3}\\

\frac{a\sqrt 3}{3}-\frac{\gamma_{1}}{\sqrt3}\\

\frac{a\sqrt 3}{3}\\
\end{cases}
\ \mathrm{and}\quad 
\lambda=
\begin{cases}
\frac{-2\gamma_{2}}{3}+\frac{a\sqrt 3}{3}\\
\frac{-2\gamma_{2}}{3}+\frac{a\sqrt 3}{3}-\frac{\gamma_{1}}{\sqrt3}\\

\frac{a\sqrt 3}{3}\\

\frac{a\sqrt 3}{3}+\frac{\gamma_1}{\sqrt 3}\\
\end{cases}
$$

The final step is to consider the above values modulo $G$. Since $a,b\in G$, we can get rid of $b$ in all of the above expressions and moreover, expressing $a$ as $a_1+a_2\sqrt 3$ where $a_{1},a_{2}\in\mathbb Z$ we see that modulo $G$ the numbers $\frac{a\sqrt 3}{3}$ split in three different equivalence classes with representatives in $\{0, \frac{\sqrt3}{3},\frac{2\sqrt3}{3}\}$. We conclude that the orbits of points on the two lines 
$$
\frac{\gamma_1 }{\sqrt3}
\begin{cases} 
x^{i}\\ 
x^{i+2}
\end{cases}\!\!
+\frac{\gamma_2}{\sqrt3} x^{i+1}+\lambda x^i
$$
obtained by intersecting them with the 4 lines directed by $x^{i+2}$ are determined by the following two sets of 12 values of $\lambda$. 
\begin{itemize}
\item Intersection with $\frac{\gamma_{1}}{\sqrt3}x^{i}+\frac{\gamma_2}{\sqrt3} x^{i+1}+\lambda x^i$ for the values of $\lambda$:
\begin{equation}\label{tab-1}
{\renewcommand\arraystretch{1.3}
\arraycolsep=0.25cm
\begin{array}{|c|c|c|c|}
\hline
-\frac{\gamma_{1}}{\sqrt3}-\frac{2\gamma_{2}}{3}+A_3&-\frac{2\gamma_2}{3}-\frac{2\gamma_1}{\sqrt 3}+A_3&-\frac{\gamma_{1}}{\sqrt3}+A_3 &A_3\\
\hline
\end{array}
}
\end{equation}
\item Intersection with $\frac{\gamma_{1}}{\sqrt3}x^{i+2}+\frac{\gamma_2}{\sqrt3} x^{i+1}+\lambda x^i$ for the values of $\lambda$:

\begin{equation}\label{tab-2}
{\renewcommand\arraystretch{1.3}
\arraycolsep=0.25cm
\begin{array}{|c|c|c|c|}
\hline
-\frac{2\gamma_{2}}{3}+A_3&-\frac{2\gamma_2}{3}-\frac{\gamma_1}{\sqrt 3}+A_3&A_3&\frac{\gamma_1}{\sqrt 3}+A_3\\
\hline
\end{array}
}
\end{equation}

\end{itemize}

\begin{remark}\label{rem-signe-moins}
Before we analyze the next case, we remark that if instead of the two lines 
$$
\frac{\gamma_1 }{\sqrt3}
\begin{cases} 
x^{i}\\ 
x^{i+2}
\end{cases}\!\!
+\frac{\gamma_2}{\sqrt3} x^{i+1}+\lambda x^i,
$$
we had considered the intersections with the two lines
$$
-\left(\frac{\gamma_1 }{\sqrt3}
\begin{cases} 
x^{i}\\ 
x^{i+2}
\end{cases}\!\!
+\frac{\gamma_2}{\sqrt3}x^{i+1}\right)+\lambda x^i,
$$
then we would have obtained the same values as in the above two tables \emph{with the opposite signs}.
\end{remark}

\item  if $k-i=4$, we perform the same computations as in the preceding point. The details are left to the reader. The equations to consider, ignoring the parameter $b\in G$ and written in the basis $(1,x^{4})$ are
$$
\frac{\gamma_1 }{\sqrt3}\begin{cases} 1 \\ 1\end{cases}\!\!+\frac{2\gamma_2}{3} +\lambda = \pm\left(\frac{\gamma_1 }{\sqrt3}\begin{cases} 0\\ -1\end{cases}\!\!-\frac{\gamma_2}{3}\right)+a\frac{2}{\sqrt3}.
$$
These yield $4$ different families of values for $\lambda$: 
$$
\lambda=\begin{cases}
-\gamma_2-\frac{\gamma_1 }{\sqrt3}+\frac{2a}{\sqrt 3}\\
-\gamma_2-\frac{2\gamma_1 }{\sqrt3}+\frac{2a}{\sqrt 3}\\
-\frac{\gamma_{2}}{3}+\frac{2a}{\sqrt 3}\\
-\frac{\gamma_{2}}{3}-\frac{\gamma_1 }{\sqrt3}+\frac{2a}{\sqrt 3}
\end{cases}
$$
Notice that in this case, the two parallel lines of direction $x^{i}$ under consideration yield the same values of $\lambda$ which, modulo de action of $G$, are collected in the following table:

\begin{equation}\label{tab-3}
{\renewcommand\arraystretch{1.3}
\arraycolsep=0.25cm
\begin{array}{|c|c|c|c|}
\hline
-\gamma_2-\frac{\gamma_1 }{\sqrt3}+A_3&-\gamma_2-\frac{2\gamma_1 }{\sqrt3}+A_3&-\frac{\gamma_{2}}{3}+A_3&-\frac{\gamma_{2}}{3}-\frac{\gamma_1 }{\sqrt3}+A_3\\
\hline
\end{array}
}
\end{equation}
\end{itemize}

\subsubsection{Two odd lines}
If $i, k$ are both odd numbers then we obtain, as in the preceding case, the following equations from Corollary~\ref{cor-simplification} (with $k>i$):
$$
\pm\left(\gamma_2\begin{cases}x^{k+4}\\ x^{k+6}
\end{cases}\!\!+\gamma_1x^{k+1}\right)+\mu\sqrt 3 x^k+\sqrt 3 \sum_p n_p x^p=\lambda \sqrt 3 x^i+\gamma_2\begin{cases}x^{i+4}\\ x^{i+6}\end{cases}\!\!+\gamma_1x^{i+1}.
$$
We simplify, multiplying the equation by $x^{-i}$:
$$
\pm\left(\gamma_2\begin{cases}x^{k+4-i}\\ x^{k+6-i}
\end{cases}\!\!+\gamma_1x^{k+1-i}\right)+\mu\sqrt 3 x^{k-i}+\sqrt 3\sum_p n_p x^{p-i}=\lambda\sqrt 3 +\gamma_2\begin{cases}x^{4}\\ x^{6}\end{cases}\!\!+\gamma_1x,
$$
and we substitute the expression $\sum_p n_p x^{p-i}$ by $ax+b$ using Lemma~\ref{degree1}:
\begin{equation}\label{lastone}
\pm\left(\gamma_2\begin{cases}x^{k+4-i}\\ x^{k+6-i}
\end{cases}\!\!+\gamma_1x^{k+1-i}\right)+\mu\sqrt 3 x^{k-i}+\sqrt 3(ax+b)=\lambda\sqrt 3 +\gamma_2\begin{cases}x^{4}\\ x^{6}\end{cases}\!\!+\gamma_1x.
\end{equation}
We now consider the above equation rewritten in the basis $(1,x^{k-i})$, which allows us to easily deduce the value of $\lambda$. To this end, we use the array at the beginning of the section. We split the analysis in two cases, depending on whether $k-i$ is 2 or 4. 
We deduce
\begin{itemize}
\item If $k-i=2$, we use the basis $(1,x^{2})$ and we further split the computations depending on the sign on the right hand side of \eqref{lastone}:

\begin{itemize}
\item Considering Equation~\eqref{lastone} with sign $+$:
\begin{align*}
&\lambda\sqrt 3+\gamma_2\begin{cases} -1\\-1\end{cases}\!\!+\gamma_1\frac{\sqrt 3}{3}=b\sqrt 3+a-\frac{\gamma_1}{\sqrt 3}+\gamma_2\begin{cases}-1 \\ 0\end{cases}&\iff\\
&\lambda\sqrt 3=a+b\sqrt 3-2\gamma_1\frac{\sqrt 3}{3}+\gamma_2\begin{cases}-1 \\ 0\end{cases}\!\!-\gamma_2\begin{cases} -1\\-1\end{cases}&\iff\\
&\lambda\sqrt 3 =a+b\sqrt 3-2\gamma_1\frac{\sqrt 3}{3}+\gamma_2\begin{cases}0 \\ 1\end{cases}&\iff\\
&\lambda =b+\frac{a\sqrt 3}{3}-\frac{2\gamma_1}{3}+\frac{\gamma_2}{\sqrt 3}\begin{cases}0 \\ 1\end{cases}
\end{align*}

\item Considering Equation~\eqref{lastone} with sign $-$:
\begin{align*}
&\lambda\sqrt 3+\gamma_2\begin{cases} -1\\-1\end{cases}\!\!+\gamma_1\frac{\sqrt 3}{3}=b\sqrt 3+a+\gamma_1\frac{\sqrt 3}{3}-\gamma_2\begin{cases}-1 \\ 0\end{cases}&\iff\\
&\lambda\sqrt 3=b\sqrt 3+a-\gamma_2\begin{cases}-2 \\ -1\end{cases}&\iff\\
&\lambda=b+a/\sqrt 3-\frac{\gamma_2}{\sqrt 3}\begin{cases}-2\\ -1\end{cases}
\end{align*}
\end{itemize}
Finally, we obtain the following 12 different values of $\lambda$ modulo $G$. Remark that in this case the two parallel lines of direction $x^i$ yield the same values of $\lambda$.

\begin{equation}\label{tab-4}
{\renewcommand\arraystretch{1.3}
\arraycolsep=0.25cm
\begin{array}{|c|c|c|c|}
\hline
\frac{-2\gamma_1}{3}+A_3&-\frac{2\gamma_1}{3}+\frac{\gamma_2}{\sqrt 3}+A_3&\frac{2\gamma_2}{\sqrt 3}+A_3&\frac{\gamma_2}{\sqrt 3}+A_3\\
\hline
\end{array}
}
\end{equation}

\item If $k-i=4$, we use the basis $(1,x^{4})$ in \eqref{lastone} and perform a computation analogous to the one in the preceding case:
\begin{align*}
&\lambda\sqrt 3+\gamma_2\begin{cases}0\\-1 \end{cases}\!\!+\gamma_1\frac{2}{\sqrt 3}=b\sqrt 3+2a\pm\left(-\gamma_1\frac{\sqrt 3}{3}+\gamma_2\begin{cases} -1\\ 0\end{cases}\right)&\iff\\
&\lambda\sqrt 3=
\begin{cases}b\sqrt 3+2a-\frac{\gamma_1}{\sqrt3}+\gamma_2
\begin{cases} -1\\ 0
\end{cases}\!\! +\gamma_2
\begin{cases}0\\1 
\end{cases}\!\!-\frac{2\gamma_1}{\sqrt 3}\\
\\
b\sqrt 3+2a+\frac{\gamma_1}{\sqrt3}-\gamma_2
\begin{cases} -1\\ 0
\end{cases}\!\! -\gamma_2
\begin{cases}0\\-1 
\end{cases}\!\!-\frac{2\gamma_1}{\sqrt 3}
\end{cases}&\iff\\
&\lambda=
b+\frac{2a}{\sqrt 3}-\gamma_1+\frac{\gamma_2}{\sqrt 3}\begin{cases} -1\\ 0\\ 1
\end{cases}
\quad\mathrm{or}\quad
\lambda=b+\frac{2a}{\sqrt 3}-\frac{\gamma_1}{3}+\frac{\gamma_2}{\sqrt 3}\begin{cases}0\\1\\2 \end{cases}
\end{align*}
We conclude that, for $i$ odd, the orbits of points on the two lines 
$$
\frac{\gamma_2}{\sqrt3}
\begin{cases} x^{i+4}\\ 
x^{i+6}\end{cases}\!\!
+\frac{\gamma_1}{\sqrt3} x^{i+1}+\lambda x^i
$$
obtained by intersecting them with the 4 lines directed by $x^{i+4}$ are determined by the following 2 sets of 12 values of $\lambda$. 
\begin{itemize}
\item Intersection with $\frac{\gamma_2}{\sqrt3}x^{i+4}+\frac{\gamma_1}{\sqrt3} x^{i+1}+\lambda x^i$ for the values of $\lambda$:

\begin{equation}\label{tab-5}
{\renewcommand\arraystretch{1.3}
\arraycolsep=0.25cm
\begin{array}{|c|c|c|c|}
\hline
-\gamma_{1}+A_3&-\frac{\gamma_1}{3}+A_3&-\gamma_{1}-\frac{\gamma_2}{\sqrt3}+A_3&-\frac{\gamma_1}{3}+\frac{\gamma_2}{\sqrt3}+A_3\\
\hline
\end{array}
}
\end{equation}
\item Intersection with $\frac{\gamma_2}{\sqrt3}x^{i+6}+\frac{\gamma_1}{\sqrt3} x^{i+1}+\lambda x^i$ for the values of $\lambda$:

\begin{equation}\label{tab-6}
{\renewcommand\arraystretch{1.3}
\arraycolsep=0.25cm
\begin{array}{|c|c|c|c|}
\hline
-\gamma_1+A_3&-\gamma_1+\frac{\gamma_2}{\sqrt 3}+A_3&-\frac{\gamma_1}{3}+\frac{2\gamma_2}{\sqrt3}+A_3&-\frac{\gamma_1}{3}+\frac{\gamma_2}{\sqrt3}+A_3\\
\hline
\end{array}
}
\end{equation}

\end{itemize}
\end{itemize}

\subsubsection{Two lines of different parity, $k$ odd}
If $i, k$ are of different parity and $k$ is odd, then we use  Corollary~\ref{cor-simplification} to understand the intersection of two lines of direction $x^{i}$ with the translates of 4 lines of direction $x^{k}$. We obtain
\begin{equation}\label{pair-impair}
\pm\left(\gamma_2\begin{cases}x^{k+4}\\ x^{k+6}\end{cases}\!\!+\gamma_1x^{k+1}\right)+\mu \sqrt 3x^k+\sqrt 3\sum_p n_p x^p=\lambda\sqrt 3 x^i+\gamma_1\begin{cases}x^i\\ x^{i+2}\end{cases}\!\!+\gamma_2x^{i+1}.
\end{equation}
We isolate $\lambda$ and divide the whole expression by $x^{i}$,
$$
\lambda\sqrt 3=\pm\left(\gamma_2\begin{cases}x^{k-i+4}\\ x^{k-i+6}\end{cases}\!\!+\gamma_1x^{k+1-i}\right)+\mu\sqrt 3 x^{k-i}+\sqrt 3\sum_p n_p x^{p-i}-\gamma_1\begin{cases}x^{0}\\ x^{2}\end{cases}\!\!-\gamma_2x.
$$
We use Lemma~\ref{degree1} to substitute the expression $\sum_p n_p x^{p-i}$:
$$
\lambda\sqrt 3=\pm\left(\gamma_2\begin{cases}x^{k-i+4}\\ x^{k-i+6}\end{cases}\!\!+\gamma_1x^{k+1-i}\right)+\mu\sqrt 3 x^{k-i}+\sqrt 3(ax+b)-\gamma_1\begin{cases}1\\ x^2\end{cases}\!\!-\gamma_2x.
$$
We finally divide by $\sqrt3$ and obtain the general equation we will be dealing with in the case of $i,k$ of different parity, $k$ odd:
\begin{equation}\label{diffparity}
\lambda=\pm\left(\frac{\gamma_2}{\sqrt 3}\begin{cases}x^{k-i+4}\\ x^{k-i+6}\end{cases}\!\!+\frac{\gamma_1}{\sqrt 3}x^{k+1-i}\right)+\mu x^{k-i}+ax+b-\frac{\gamma_1}{\sqrt 3}\begin{cases}1\\ x^2\end{cases}\!\!-\frac{\gamma_2}{\sqrt 3}x.
\end{equation}
 
The expression $k-i$ can take any of the following values $\{-3,-1,1,3,5\}$. Notice that since $x^n=-x^{n+6}$, the cases $-3$ and $3$ and $-1$ and $5$ will yield the same sets of values of $\lambda$. Indeed, in Equation~\ref{diffparity}, when written with respect to the bases $(1,x^{3})$ or $(1,x^{-3})$ the only change is an overall sign change in the first parenthesis, which leaves unchanged the set of possible values of $\lambda$. The same occurs when comparing the values $k-i\in\{-1,5\}$ with respect to the bases $(1,x^{-1})$ or $(1,x^{5})$. It follows that we will have a complete answer by studying the cases $k-i=-3,-1,1$. To simplify the notation we will denote by $\Re[\,\cdot\,]$ the linear map that returns the first coordinate of an arbitrary expression on the basis $(1,x^{k-i})$.

\begin{itemize}
\item If $k-i=-3$, then Equation~\eqref{diffparity} yields, ignoring the monomial with coefficient $\mu$,
$$
\lambda=\Re\left[\pm\left(\frac{\gamma_2}{\sqrt 3}\begin{cases}x\\ x^{3}\end{cases}\!\!+\frac{\gamma_1}{\sqrt 3}x^{-2}\right)+ax+b-\frac{\gamma_1}{\sqrt 3}\begin{cases}1\\ x^2\end{cases}\!\!-\frac{\gamma_2}{\sqrt 3}x\right].
$$
Now, ignoring the parameter $b$, since ultimately we are interested in $\lambda$ up to the action of $G$, and using the chart at the beginning of this section we obtain:

$$\lambda=\pm\left(\frac{\gamma_2}{\sqrt 3}\begin{cases}\frac{\sqrt 3}{2}\\ 0\end{cases}\!\!+\frac{\gamma_1}{2\sqrt 3}\right)+\frac{a\sqrt 3}{2}-\frac{\gamma_1}{\sqrt 3}\begin{cases}1\\ \frac{1}{2}\end{cases}\!\!-\frac{\gamma_2}{2}$$

We conclude that, for $i$ even, the orbits of points on the two lines 
$$
\frac{\gamma_1 }{\sqrt3}
\begin{cases} 
x^{i}\\ 
x^{i+2}
\end{cases}\!\!
+\frac{\gamma_2}{\sqrt3} x^{i+1}+\lambda x^i,
$$
obtained by intersecting them with the translates of the 4 lines directed by $x^{k}=x^{i-3}$ are determined by the following 2 sets of 16 values of $\lambda$ modulo $G$. To compile the tables we used the fact that $\frac{a\sqrt 3}{2}$ with $a\in G$ can be expressed as
$$
a\frac{\sqrt 3}{2}=(a_1+a_2\sqrt 3)\frac{\sqrt 3}{2}=\frac{1}{2}(a_1\sqrt 3+3a_{2})\ \mathrm{with\ } a_1, a_2\in \mathbb Z.
$$ 
It follows that modulo $G$, the expression $\frac{a\sqrt 3}{2}$ is equivalent to an element in the set $\{0,\frac{1}{2},\frac{\sqrt 3}{2},\frac{1+\sqrt 3}{2}\}$.

\begin{itemize}
\item Intersection with $\frac{\gamma_{1}}{\sqrt3}x^{i}+\frac{\gamma_2}{\sqrt3} x^{i+1}+\lambda x^i$ for the values of $\lambda$:
\begin{equation}\label{tab-7}
{\renewcommand\arraystretch{1.3}
\arraycolsep=0.25cm
\begin{array}{|c|c|c|c|}
\hline
-\frac{\gamma_{1}}{2\sqrt3}+A_4& -\gamma_{2}-\frac{\gamma_{1}\sqrt3}{2}+A_4& -\frac{\gamma_{2}}{2}-\frac{\gamma_{1}}{2\sqrt3}+A_4& -\frac{\gamma_{2}}{2}-\frac{\gamma_{1}\sqrt3}{2}+A_4\\
\hline
\end{array}
}
\end{equation}
\item Intersection with $\frac{\gamma_{1}}{\sqrt3}x^{i+2}+\frac{\gamma_2}{\sqrt3} x^{i+1}+\lambda x^i$ for the values of $\lambda$:
\begin{equation}\label{tab-8}
{\renewcommand\arraystretch{1.3}
\arraycolsep=0.25cm
\begin{array}{|c|c|c|c|}
\hline
A_4& -\gamma_{2}-\frac{\gamma_{1}}{\sqrt3}+A_4& -\frac{\gamma_{2}}{2}-\frac{\gamma_{1}}{\sqrt3}+A_4& -\frac{\gamma_{2}}{2}+A_4\\
\hline
\end{array}
}
\end{equation}
\end{itemize}

\item If $k-i=1$ we follow the same steps as in the preceding case to obtain:
$$
\lambda=\Re\left[\pm\left(\frac{\gamma_2}{\sqrt 3}\begin{cases}x^{3}\\ x^{5}\end{cases}\!\!+\frac{\gamma_1}{\sqrt 3}\right)+\mu x^{-1}+ax+b-\frac{\gamma_1}{\sqrt 3}\begin{cases}1\\ x^2\end{cases}\!\!-\frac{\gamma_2}{\sqrt 3}x\right],
$$
which turns into
$$
\lambda=\pm\left(\frac{\gamma_2}{\sqrt 3}\begin{cases}\sqrt 3\\ 0\end{cases}\!\!+\frac{\gamma_1}{\sqrt 3}\right)+\sqrt3 a-\frac{\gamma_1}{\sqrt 3}\begin{cases}1\\ 2\end{cases}\!\!-\gamma_2.$$

Notice that in this case, modulo de action of $G$, there is only one orbit of values of $\lambda$ for each of the 4 lines of direction $x^{i-1}$. The values of $\lambda$ obtained are:

\begin{itemize}
\item Intersection with $\frac{\gamma_{1}}{\sqrt3}x^{i}+\frac{\gamma_2}{\sqrt3} x^{i+1}+\lambda x^i$ for the values of:

\begin{equation}\label{tab-9}
\lambda\in\{0,-\gamma_2,-2\gamma_2-\frac{2\gamma_1}{\sqrt 3},-\gamma_2-\frac{2\gamma_1}{\sqrt 3}\}.
\end{equation}

\item Intersection with $\frac{\gamma_{1}}{\sqrt3}x^{i+2}+\frac{\gamma_2}{\sqrt3} x^{i+1}+\lambda x^i$ for the values of:

\begin{equation}\label{tab-10}
\lambda\in\{-\frac{\gamma_1}{\sqrt 3},-\gamma_2-\frac{2\gamma_1}{\sqrt 3},-2\gamma_2-\sqrt 3\gamma_1,-\gamma_2-\sqrt 3\gamma_1\}.
\end{equation}

\end{itemize}

\item If $k-i=1$, we follow the same steps as in the preceding cases to obtain:
$$
\lambda=\Re\left[\pm\left(\frac{\gamma_2}{\sqrt 3}\begin{cases}x^{5}\\ x^{6}\end{cases}\!\!+\frac{\gamma_1}{\sqrt 3}x^2\right)+\mu x^{1}+ax+b-\frac{\gamma_1}{\sqrt 3}\begin{cases}1\\ x^2\end{cases}\!\!-\frac{\gamma_2}{\sqrt 3}x\right],
$$
which turns into
$$
\lambda=\pm\left(\frac{\gamma_2}{\sqrt 3}\begin{cases}-\sqrt 3\\ 0\end{cases}\!\!-\frac{\gamma_1}{\sqrt 3}\right)+b-\frac{\gamma_1}{\sqrt 3}\begin{cases}1\\ -1\end{cases}\!\!,
$$
yielding,
\begin{itemize}
\item Intersection with $\frac{\gamma_{1}}{\sqrt3}x^{i}+\frac{\gamma_2}{\sqrt3} x^{i+1}+\lambda x^i$ for the values of:
\begin{equation}\label{tab-11}
\lambda\in\{-\frac{2\gamma_1}{\sqrt 3},-\frac{2\gamma_1}{\sqrt 3}-\gamma_2,0,\gamma_2\}.
\end{equation}
\item Intersection with $\frac{\gamma_{1}}{\sqrt3}x^{i+2}+\frac{\gamma_2}{\sqrt3} x^{i+1}+\lambda x^i$ for the values of:
\begin{equation}\label{tab-12}
\lambda\in\{0,-\gamma_2,\frac{2\gamma_1}{\sqrt 3},\frac{2\gamma_1}{\sqrt 3}+\gamma_2\}.
\end{equation}
\end{itemize}
\end{itemize}

\subsubsection{Two lines of different parity, $k$ even}
The last case we need to analyze is $k$, $i$ of different parity and $k$ even. So, once again we fix two lines directed by $x^{i}$ and we intersect them with the four different lines directed by $x^{k}$. That is, we are looking for the values of $\lambda$ in the following equation:
$$
\pm\left(\frac{\gamma_1 }{\sqrt3}\begin{cases} x^{k}\\ x^{k+2}\end{cases}+\frac{\gamma_2}{\sqrt3} x^{k+1}\right)+\mu x^k+\sum_p n_p x^p=\frac{\gamma_2}{\sqrt3}\begin{cases} x^{i+4}\\ x^{i+6}\end{cases}+\frac{\gamma_1}{\sqrt3} x^{i+1}+\lambda x^i,
$$
which we divide by $x^{i}$ to obtain:
\begin{equation}\label{impair-pair}
\pm\left(\frac{\gamma_1 }{\sqrt3}\begin{cases} x^{k-i}\\ x^{k-i+2}\end{cases}+\frac{\gamma_2}{\sqrt3} x^{k-i+1}\right)+\mu x^{k-i}+\sum_p n_p x^p=\frac{\gamma_2}{\sqrt3}\begin{cases} x^{4}\\ x^{6}\end{cases}+\frac{\gamma_1}{\sqrt3} x^{1}+\lambda.
\end{equation}

This time the expression $k-i$ can take any of the following values $\{-5,-3,-1,1,3\}$, and just like in the case $k$ odd, $i$ even, the analysis will be complete if we study the cases $k-i=-3,-1,1$. 
\begin{itemize}
\item If $k-i=-3$, then working with the basis $(1,x^{-3})$ we have:
\begin{align*}
\lambda&= \pm\left(\frac{\gamma_1 }{\sqrt3}\begin{cases} -x^3\\ -x^5\end{cases}-\frac{\gamma_2}{\sqrt3}x^4 \right)-\mu x^{-3}+\sum_p n_p x^p-\frac{\gamma_2}{\sqrt3}\begin{cases} x^{4}\\ x^{6}\end{cases}-\frac{\gamma_1}{\sqrt3} x^{1}\\[5pt]
\lambda&= \pm\left(\frac{\gamma_1 }{\sqrt3}\begin{cases} 0\\ -\Re(x^5)\end{cases}-\frac{\gamma_2}{\sqrt3}\Re(x^4) \right)+\Re(ax+b)-\frac{\gamma_2}{\sqrt3}\begin{cases} \Re(x^{4})\\ -1\end{cases}-\frac{\gamma_1}{\sqrt3} \Re(x)\\[5pt]
\lambda&= \pm\left(\frac{\gamma_1 }{\sqrt3}\begin{cases} 0\\ \frac{\sqrt 3}{2}\end{cases}+\frac{\gamma_2}{2\sqrt3} \right)+a\frac{\sqrt 3}{2}+b-\frac{\gamma_2}{\sqrt3}\begin{cases} -\frac{1}{2}\\ -1\end{cases}-\frac{\gamma_1}{2}\\
\end{align*}
yielding
\begin{itemize}
\item Intersection with $\frac{\gamma_{1}}{\sqrt3}x^{i+1}+\frac{\gamma_2}{\sqrt3} x^{i+4}+\lambda x^i$ for the values of:
$$
\lambda= \pm\left(\frac{\gamma_1 }{\sqrt3}\begin{cases} 0\\ \frac{\sqrt 3}{2}\end{cases}+\frac{\gamma_2}{2\sqrt3} \right)+a\frac{\sqrt 3}{2}+b+\frac{\gamma_2}{2\sqrt3}-\frac{\gamma_1}{2}
$$
which, modulo de action of $\Gamma$, can be summarized as:
\begin{equation}\label{tab-15}
{\renewcommand\arraystretch{1.3}
\arraycolsep=0.25cm
\begin{array}{|c|c|c|c|}
\hline
 -\frac{\gamma_1}{2}+\frac{\gamma_2}{\sqrt 3}+A_4& \frac{\gamma_2}{\sqrt 3}+A_4& -\gamma_{1}+A_4& -\frac{\gamma_1}{2}+A_4\\
 \hline
 \end{array}
 }
\end{equation}

\item Intersection with $\frac{\gamma_{1}}{\sqrt3}x^{i+1}+\frac{\gamma_2}{\sqrt3} x^{i+6}+\lambda x^i$ for the values of:
$$
\lambda= \pm\left(\frac{\gamma_1 }{\sqrt3}\begin{cases} 0\\ \frac{\sqrt 3}{2}\end{cases}+\frac{\gamma_2}{2\sqrt3} \right)+a\frac{\sqrt 3}{2}+b+\frac{\gamma_2}{\sqrt3}-\frac{\gamma_1}{2}
$$
which, modulo de action of $\Gamma$, can be summarized as:
\begin{equation}\label{tab-16}
{\renewcommand\arraystretch{1.3}
\arraycolsep=0.25cm
\begin{array}{|c|c|c|c|}
\hline
  -\gamma_1+\frac{\gamma_2}{2\sqrt 3}+A_4&-\frac{\gamma_1}{2}+\frac{\gamma_2}{2\sqrt 3}+A_4&-\frac{\gamma_1}{2}+\frac{\gamma_2\sqrt 3}{2}+A_4&\frac{\gamma_2\sqrt 3}{2}+A_4\\
\hline
\end{array}
}
\end{equation}
\end{itemize}

\item If $k-i=-1$, then remark that $x^{-i}=-x^{6-i}$ and we obtain
\begin{align*}
\lambda&= \pm\left(\frac{\gamma_1 }{\sqrt3}\begin{cases} -x^5\\ x\end{cases}+\frac{\gamma_2}{\sqrt3} \right)-\mu x^{5}+\sum_p n_p x^p-\frac{\gamma_2}{\sqrt3}\begin{cases} x^{4}\\ x^{6}\end{cases}-\frac{\gamma_1}{\sqrt3} x^{1}\\[5pt]
\lambda&= \pm\left(\frac{\gamma_1 }{\sqrt3}\begin{cases} 0\\ \Re(x)\end{cases}+\frac{\gamma_2}{\sqrt3} \right)+\Re(ax+b)-\frac{\gamma_2}{\sqrt3}\begin{cases} \Re(x^{4})\\ -1\end{cases}-\frac{\gamma_1}{\sqrt3} \Re(x)\\[5pt]
\lambda&= \pm\left(\frac{\gamma_1 }{\sqrt3}\begin{cases} 0\\ \Re(x)\end{cases}+\frac{\gamma_2}{\sqrt3} \right)+a\Re(x)+b-\frac{\gamma_2}{\sqrt3}\begin{cases} \Re(x^{4})\\ -1\end{cases}-\frac{\gamma_1}{\sqrt3} \Re(x)\\
\end{align*}
which turns into
$$
\lambda=\pm\left(\frac{\gamma_1 }{\sqrt3}\begin{cases} 0\\ \sqrt 3\end{cases}+\frac{\gamma_2}{\sqrt3} \right)+a\sqrt 3+b-\frac{\gamma_2}{\sqrt3}\begin{cases} 1\\ -1\end{cases}-\frac{\gamma_1}{\sqrt3} \sqrt 3
$$
yielding
\begin{itemize}
\item  Intersection with $\frac{\gamma_{1}}{\sqrt3}x^{i+1}+\frac{\gamma_2}{\sqrt3} x^{i+4}+\lambda x^i$ for the values of:
$$
\lambda=\pm\left(\frac{\gamma_1 }{\sqrt3}\begin{cases} 0\\ \sqrt 3\end{cases}+\frac{\gamma_2}{\sqrt3}\right) +a\sqrt 3+b-\frac{\gamma_2}{\sqrt3}-\gamma_1,
$$
that is
\begin{equation}\label{tab-13}
\lambda\in\{-2\gamma_1-\frac{2\gamma_2}{\sqrt 3},-\gamma_1,-\gamma_1-\frac{2\gamma_2}{\sqrt 3},0\}.
\end{equation}

\item   Intersection with $\frac{\gamma_{1}}{\sqrt3}x^{i+1}+\frac{\gamma_2}{\sqrt3} x^{i+6}+\lambda x^i$ for the values of:
$$
\lambda=\pm\left(\frac{\gamma_1 }{\sqrt3}\begin{cases} 0\\ \sqrt 3\end{cases}+\frac{\gamma_2}{\sqrt3}\right) +a\sqrt 3+b+\frac{\gamma_2}{\sqrt3}-\gamma_1,
$$
that is,
\begin{equation}\label{tab-14}
\lambda\in\{\frac{2\gamma_2}{\sqrt 3},-\gamma_1+\frac{2\gamma_2}{\sqrt 3},-\gamma_1,-2\gamma_1 \}.
\end{equation}
\end{itemize}
\item If $k-i=1$, then working with the basis $(1,x)$ and Lemma~\ref{degree1} we have:
\begin{align*}
\lambda&=\pm\left(\frac{\gamma_1 }{\sqrt3}\begin{cases} x\\ x^{3}\end{cases}+\frac{\gamma_2}{\sqrt3} x^{2}\right)+\mu x+\sum_p n_p x^p-\frac{\gamma_2}{\sqrt3}\begin{cases} x^{4}\\ x^{6}\end{cases}-\frac{\gamma_1}{\sqrt3} x^{1}\\[5pt]
\lambda&=\pm\left(\frac{\gamma_1 }{\sqrt3}\begin{cases} 0\\ -\sqrt 3\end{cases}-\frac{\gamma_2}{\sqrt3} \right)-\frac{\gamma_2}{\sqrt3}\begin{cases} -2\\ -1\end{cases}\\
\end{align*}
yielding
\begin{itemize}
\item Intersection with $\frac{\gamma_{1}}{\sqrt3}x^{i+1}+\frac{\gamma_2}{\sqrt3} x^{i+4}+\lambda x^i$ for the values of:
$$
\lambda=\pm\left(\frac{\gamma_1 }{\sqrt3}\begin{cases} 0\\ -\sqrt 3\end{cases}-\frac{\gamma_2}{\sqrt3} \right)+\frac{2\gamma_2}{\sqrt3}
$$
that is,
\begin{equation}\label{tab-17}
\lambda\in\{ -\gamma_1+ \frac{\gamma_2}{\sqrt 3}, \frac{\gamma_2}{\sqrt 3},\gamma_2\sqrt3, \gamma_2\sqrt 3+\gamma_1 \}.
\end{equation}

\item Intersection with $\frac{\gamma_{1}}{\sqrt3}x^{i+1}+\frac{\gamma_2}{\sqrt3} x^{i+6}+\lambda x^i$ for the values of:
$$
\lambda=\pm\left(\frac{\gamma_1 }{\sqrt3}\begin{cases} 0\\ -\sqrt 3\end{cases}-\frac{\gamma_2}{\sqrt3} \right)+\frac{\gamma_2}{\sqrt3}
$$
that is,
\begin{equation}\label{tab-18}
\lambda\in\{ -\gamma_1,0,\frac{2\gamma_2}{\sqrt 3},\frac{2\gamma_2}{\sqrt 3}+\gamma_1 \}.
\end{equation}
\end{itemize}
\end{itemize}

%
%
%%%%%%%%%%%%%%%%%
\subsection{General case: results}\label{generalfinal}

In the previous section we have computed for each line of direction $x^{i}$ all the intersection points with each of the $\Gamma$-translates of the other 20 lines of different directions. Moreover, we have collected these points in orbits under the $\Gamma$-action. The first step in order to compute the numbers $L_0$ and $L_0^\alpha$, needed to determine $H^2(\Omega_{E_{12}^\gamma})$, is to organize the information gathered so far, paying special attention to repetitions in the above tables. This is done in Proposition~\ref{prop-points-droite}. 

We have not computed explicitly the values $L_0$ and $ L_0^\alpha$ for all different parameters $\gamma$. The number of subcases is too large to be considered of interest by itself. We finish our study with Proposition~\ref{prop-points-droite}, which gives an upper bound on $L_0^\alpha$ and moreover it describes explicitly the representatives of the 0-singularities on each 1-singularity in terms of the parameter $\gamma$.

Before we present the general picture in the next proposition, we show here the intersection combinatorics, obtained with computer assistance, for the case
$\gamma_1=1/7+\sqrt3/11$, $\gamma_2=1/13+\sqrt3/17$. For technical reasons, our program requires $\gamma_i \in \mathbb Q(\sqrt3)$, but we conjecture these values are representatives of the `generic case', meaning that lines and line intersections cannot further split
up into several lines or intersections, if the $\gamma$ parameters are varied.
Note that the parameters $(\gamma_{1},\gamma_{2})$ are taken in the region
labelled with a red 24 in Figure~\ref{fig-orbit-tout}. The intersection
combinatorics for this case are:

\[ {\renewcommand\arraystretch{1.3}
\begin{array}{|c|c|c||c|c|c|c|c||c|}
  \hline
                    & n  & \text{dir} & p=2 & p=3 & p=4 & p=5 & p=6 & \text{tot} \\
  \hline
  L_{0,p}^1          & 24 & e,o        &   42 &   0 &   2 &   0 &   0 &   44 \\
  \hline
  \sum L_{0,p}^\alpha & 24 &            & 1008 &   0 &  48 &   0 &   0 & 1056 \\
  \hline
  L_{0,p}            &    &            & 504  &   0 &  12 &   0 &   0 &  516 \\
  \hline
\end{array} } \]

We see that we have here a single line type occurring in both even and odd directions,
and that the value of $L_0^\alpha$ really attains the maximal value of 44 admitted by
Proposition~\ref{prop-points-droite}. It is worth pointing out that in this generic case, not all
line intersections are generic in the sense that only two lines may intersect in a
single point (notice the entry in the table under $p=4$). The reason is that the positions of the singular lines do not move
independently of each other when the $\gamma$-values vary. In this generic case,
we obtain the following cohomology:

\[ {\renewcommand\arraystretch{1.3}
\begin{array}{|c|c|c|c|c|c|c|}
  \hline
  \sum L_0^\alpha & L_0 & L_1 &   e & \rk H^2(\Omega_{E_{12}^\gamma}) & \rk H^1(\Omega_{E_{12}^\gamma}) & \rk H^0(\Omega_{E_{12}^\gamma}) \\
  \hline
           1056 & 516 &  24 & 540 &    564  &      25 &       1 \\
  \hline
\end{array} } \]

We conjecture this to be the maximal cohomology attained among all the generalized 12-fold tilings studied in this paper.

Turning now our attention back to the general case, we present in the next proposition our findings on the quantity $L_{0}^{\alpha}$. The statement refers to the following explicit list of the 1-singularities $\alpha$:

\begin{enumerate}
\item[(a)] $\alpha=\pm\frac{1}{\sqrt 3}(\gamma_1 x^i+\gamma_2x^{i+1})+\lambda x^i$, $\lambda\in\mathbb R$, $i\in\{0,2,4\}$.  
\item[(b)] $\alpha=\pm\frac{1}{\sqrt 3}(\gamma_1 x^{i+2}+\gamma_2x^{i+1})+\lambda x^i$, $\lambda\in\mathbb R$, $i\in\{0,2,4\}$. 
\item[(c)] $\alpha=\pm\frac{1}{\sqrt3}(\gamma_{1}x^{i+1}+\gamma_2 x^{i+4})+\lambda x^i$, $\lambda\in\mathbb R$ $i\in\{1,3,5\}$.
\item[(d)] $\alpha=\pm\frac{1}{\sqrt3}(\gamma_{1}x^{i+1}+\gamma_2 x^{i+6})+\lambda x^i$, $\lambda\in\mathbb R$ $i\in\{1,3,5\}$. 
\end{enumerate}

\begin{proposition}\label{prop-points-droite}
All the 1-singularities $\alpha$ have an associated value $ L_0^\alpha$ bounded by $44$. Moreover, the values of $\lambda$ identifying representatives of points of intersection between $\alpha$ (in the above \emph{(a), (b), (c)} and \emph{(d)} cases) and the translates of the 20 non-parallel 1--singularities are collected (respectively) in Table~\ref{tab-even1}, Table~\ref{tab-even2}, Table~\ref{tab-odd1} and Table~\ref{tab-odd2}. (The values in the tables are for $\alpha$ with a positive $\frac{1}{\sqrt3}$ factor; if the factor is negative, one should consider minus the entries in the tables.) 
\end{proposition}
\begin{proof}
The proof of this statement is nothing but organizing the information collected in the previous computations. We fix a direction $x^{i}$ and we look at the 4 parallel lines given by
Corollary~\ref{cor-simplification}. 
\begin{enumerate}
\item If $i$ is even, then to understand the 0--singularities on:

\begin{enumerate}
\item $\alpha=\frac{1}{\sqrt 3}(\gamma_1 x^i+\gamma_2x^{i+1})+\lambda x^i$, we collect the values of $\lambda$ from the sets~\ref{tab-1} and~\ref{tab-3} (for the intersections with lines directed by $x^{i+2}$ and $x^{i+4}$ respectively) and from the sets~\ref{tab-7}, \ref{tab-9} and~\ref{tab-11} (for the intersections with lines directed by $x^{i+3}$, $x^{i+5}$ and $x^{i+1}$ respectively).

Notice that the value $-\frac{2\gamma_{1}}{\sqrt3}-\gamma_{2}$ appears in the sets~\ref{tab-3}, \ref{tab-9} and~\ref{tab-11}; the value 0 is in both sets~\ref{tab-9} and~\ref{tab-11}. The collection of values with no repetitions is the content of Table~\ref{tab-even1}. Depending on the parameter $\gamma$, some of the values in the table could be in the same orbit under the action of $\Gamma$. We conclude then that the total number of different values in the table provides an upper bound for $L_{0}^{\alpha}$, which in this case is easily checked to be $44$.

If $\alpha=-\frac{1}{\sqrt 3}(\gamma_1 x^i+\gamma_2x^{i+1})+\lambda x^i$ we use Remark~\ref{rem-signe-moins} to conclude that the values we are interested in are minus the ones displayed in Table~\ref{tab-even1}.

\item $\alpha=\frac{1}{\sqrt 3}(\gamma_1 x^{i+2}+\gamma_2x^{i+1})+\lambda x^i$, we proceed just as before extracting the values from the sets~\ref{tab-2}, \ref{tab-3}, \ref{tab-8}, \ref{tab-10} and \ref{tab-12}. The summary of all the values, with one repetition, is presented in Table~\ref{tab-even2}. 

Again, if we consider $\alpha=-\frac{1}{\sqrt 3}(\gamma_1 x^{i+2}+\gamma_2x^{i+1})+\lambda x^i$,  by Remark~\ref{rem-signe-moins}, we obtain minus the values in Table~\ref{tab-even2}. This time, the upper bound for $L_{0}^{\alpha}$ we obtain is 44.

\end{enumerate}
\item If $i$ is odd, we need to discuss the 0-singularities on: 
	\begin{enumerate}
 		\item[(c)] $\alpha=\frac{1}{\sqrt3}(\gamma_2x^{i+4}+\gamma_1x^{i+1})+\lambda x^i$, for which we extract the values of $\lambda$ from the sets~\ref{tab-4}, \ref{tab-5}, \ref{tab-15},  \ref{tab-13} and \ref{tab-17}. Notice that in these sets the value $ -\gamma_1$ appears three times, $-\gamma_1+\frac{\gamma_2}{\sqrt 3}$ twice and $\frac{\gamma_2}{\sqrt 3}$ also twice.  The collection of all these values with only one repetition is the content of Table~\ref{tab-odd1}. The upper bound for $L_{0}^{\alpha}$ in this case is found to be 44. The case of $\alpha=-\frac{1}{\sqrt3}(\gamma_2x^{i+4}+\gamma_1x^{i+1})+\lambda x^i$ is dealt with via Remark~\ref{rem-signe-moins} as in the previous cases.
 		\item[(d)] $\alpha=\frac{1}{\sqrt3}(\gamma_2x^{i+6}+\gamma_1x^{i+1})+\lambda x^i$ we use the values of $\lambda$ in sets~\ref{tab-4}, \ref{tab-6}, \ref{tab-16}, \ref{tab-14},  and \ref{tab-18}. This time the value $-\gamma_{1}$ appears three times and the values $-\gamma_{1}+\frac{\gamma_{2}}{\sqrt 3}+A_{3}$, $-\gamma_{1}+\frac{2\gamma_{2}}{\sqrt 3}$ and $\frac{2\gamma_{2}}{\sqrt3}$ each appears twice. The complete set of values with no repetitions is collected in Table~\ref{tab-odd2}. This time the bound obtained is $L_{0}^{\alpha}\leq 44$. The case $\alpha=-\frac{1}{\sqrt3}(\gamma_2x^{i+6}+\gamma_1x^{i+1})+\lambda x^i$ is dealt with via Remark~\ref{rem-signe-moins}.
\end{enumerate}
\end{enumerate}
\end{proof}

\begin{remark}\leavevmode  
\begin{itemize}
\item If we replace $\gamma$ by $0$ in the tables, the reader can remark that we obtain $L_0^\alpha=6$ as in Lemma \ref{lem-calc-droites-gamma-nul}. 
\item For all $\gamma$ we have $L_0\leq (44+44)\cdot2\cdot3\cdot2$. 
Of course in the case $\gamma=0$ the same upper bound gives $6\cdot6=36$, while the actual result is $14$. Thus we can imagine that $L_0$ is always strictly less than this value.
\end{itemize}
\end{remark}

\begin{table}[h]
\centering
$$
{\renewcommand\arraystretch{2}
\arraycolsep=0.35cm
 \begin{array}{|c|c|c|c|}
 \hline
-\frac{\gamma_{1}}{\sqrt3}+A_3&-\frac{\gamma_1}{\sqrt 3}-\frac{2\gamma_2}{\sqrt 3}+A_3&-\frac{2\gamma_{1}}{\sqrt3}-\frac{2\gamma_{2}}{3} +A_3 &A_3 \\
\hline
-\frac{\gamma_1}{\sqrt 3}-\gamma_2+A_3&-\gamma_2-\frac{2\gamma_1 }{\sqrt3}+A_3&-\frac{\gamma_{2}}{3}+A_3&-\frac{\gamma_{2}}{3}-\frac{\gamma_1 }{\sqrt3}+A_3 \\
\hline
-\frac{\gamma_{1}}{2\sqrt3}+A_4& -\gamma_{2}-\frac{\gamma_{1}\sqrt3}{2}+A_4& -\frac{\gamma_{2}}{2}-\frac{\gamma_{1}}{2\sqrt3}+A_4& -\frac{\gamma_{2}}{2}-\frac{\gamma_{1}\sqrt3}{2}+A_4\\
\hline
0&&-2\gamma_2-\frac{2\gamma_1}{\sqrt 3}&\\
 \hline
-\frac{2\gamma_1}{\sqrt 3}&&&\gamma_2 \\
 \hline
 \end{array}
 }
$$
\caption{Each value of $\lambda$ is of the form $a\gamma_1+b\gamma_2+c$. The value $a\gamma_1+b\gamma_2$ identifies the intersection point between the two 1--singularities, while $c\in A_{3}\cup A_{4}$ encodes the different representatives of the equivalence classes under the action of $G$. The data on the first 2 lines is from the sets~\ref{tab-1}, \ref{tab-3} respectively while the last three lines have values from the sets~\ref{tab-7},  \ref{tab-9} and~\ref{tab-11}.}
\label{tab-even1}
\end{table}
\vspace{10cm}
\begin{table}[h!]
\centering
$$
{\renewcommand\arraystretch{2}
\arraycolsep=0.35cm
 \begin{array}{|c|c|c|c|}
 \hline
\frac{\gamma_1}{\sqrt 3}+A_{3}&
-\frac{2\gamma_2}{3}+A_{3}&
 -\frac{2\gamma_2}{3}-\frac{\gamma_1}{\sqrt 3}+A_{3}&
  A_{3}\\
\hline
  -\frac{\gamma_1}{\sqrt 3}-\frac{\gamma_2}{3}+A_{3}&
   -\frac{2\gamma_1}{\sqrt 3}-\gamma_2+A_{3}&
    -\frac{\gamma_2}{3}+A_{3}&
    -\frac{\gamma_{1}}{\sqrt3}-\gamma_{2}+A_{3}\\
\hline
    A_{4}&
    -\frac{\gamma_{1}}{\sqrt3}-\gamma_{2}+A_{4}&
    -\frac{\gamma_1}{\sqrt 3}-\frac{\gamma_{2}}{2}+A_{4}&
    -\frac{\gamma_2}{2}+A_{4}\\
\hline
-\frac{\gamma_1}{\sqrt 3}&&-2\gamma_2-\sqrt 3\gamma_1&-\gamma_2-\sqrt 3\gamma_1\\
 \hline
&&\frac{2\gamma_1}{\sqrt 3}&\frac{2\gamma_1}{\sqrt 3}+\gamma_2\\
 \hline
 \end{array}
 }
$$
\caption{Each value of $\lambda$ is of the form $a\gamma_1+b\gamma_2+c$. The value $a\gamma_1+b\gamma_2$ identifies the intersection point between the two 1--singularities, while $c\in A_{3}\cup A_{4}$, encodes the different representatives of the equivalence classes under the action of $G$. There is one value repeated in the table, namely $-\frac{\gamma_{1}}{\sqrt3}-\gamma_{2}$.  The data on the first 3 lines is from Tables~\ref{tab-2}, \ref{tab-3} and \ref{tab-8} respectively while the last line has values from tables \ref{tab-10} and \ref{tab-12}.}
\label{tab-even2}
\end{table}
\begin{table}[h!]
\centering
$$
{\renewcommand\arraystretch{2}
\arraycolsep=0.35cm
 \begin{array}{|c|c|c|c|}
 \hline
- \frac{2\gamma_1}{3}+A_3&-\frac{2\gamma_1}{3}+\frac{\gamma_2}{\sqrt 3}+A_3&\frac{\gamma_2}{\sqrt 3}+A_3&\frac{2\gamma_2}{\sqrt 3}+A_3\\
\hline
-\gamma_{1}+A_3&-\frac{\gamma_1}{3}+A_3&-\gamma_{1}-\frac{\gamma_2}{\sqrt3}+A_3&-\frac{\gamma_1}{3}+\frac{\gamma_2}{\sqrt3}+A_3\\

\hline
  -\frac{\gamma_1}{2}+\frac{\gamma_2}{\sqrt 3}+A_4& \frac{\gamma_2}{\sqrt 3}+A_4& -\gamma_{1}+A_4& -\frac{\gamma_1}{2}+A_4\\
 \hline
 -2\gamma_1-\frac{2\gamma_2}{\sqrt 3}&&-\gamma_1-\frac{2\gamma_2}{\sqrt 3}&0\\
 \hline
& &\gamma_2\sqrt3& \gamma_2\sqrt 3+\gamma_1 \\
 \hline
 \end{array}
 }
$$
\caption{Each value of $\lambda$ is of the form $a\gamma_1+b\gamma_2+c$. The value $a\gamma_1+b\gamma_2$ identifies the intersection point between the two 1--singularities, while $c\in A_{3}\cup A_{4}$, encodes the different representatives of the equivalence classes under the action of $G$. The data on the first 2 lines is from the sets~\ref{tab-4}, \ref{tab-5} respectively while the last three lines have values from the sets~\ref{tab-15},  \ref{tab-13} and~\ref{tab-17}. Notice the repetition of the value $-\gamma_{1}$.}
\label{tab-odd1}
\end{table}
\begin{table}[h!]
\centering
$$
{\renewcommand\arraystretch{2}
\arraycolsep=0.35cm
 \begin{array}{|c|c|c|c|}
 \hline
- \frac{2\gamma_1}{3}+A_3&-\frac{2\gamma_1}{3}+\frac{\gamma_2}{\sqrt 3}+A_3&\frac{\gamma_2}{\sqrt 3}+A_3&\frac{2\gamma_2}{\sqrt 3}+A_3\\
\hline
-\gamma_1+A_3&-\gamma_1+\frac{\gamma_2}{\sqrt 3}+A_3&-\frac{\gamma_1}{3}+\frac{2\gamma_2}{\sqrt3}+A_3&-\frac{\gamma_1}{3}+\frac{\gamma_2}{\sqrt3}+A_3\\
\hline
-\gamma_1+\frac{\gamma_2}{2\sqrt 3}+A_4&-\frac{\gamma_1}{2}+\frac{\gamma_2}{2\sqrt 3}+A_4&-\frac{\gamma_1}{2}+\frac{\gamma_2\sqrt 3}{2}+A_4&\frac{\gamma_2\sqrt 3}{2}+A_4\\
\hline
\frac{2\gamma_2}{\sqrt 3}&&&-2\gamma_1\\
\hline
&0&&\gamma_1+\frac{2\gamma_2}{\sqrt 3}\\
 \hline
 \end{array}
 }
$$
\caption{Each value of $\lambda$ is of the form $a\gamma_1+b\gamma_2+c$. The value $a\gamma_1+b\gamma_2$ identifies the intersection point between the two 1--singularities, while $c\in A_{3}\cup A_{4}$, encodes the different representatives of the equivalence classes under the action of $G$.   
The data on the first $2$ lines is from the sets~\ref{tab-4}, \ref{tab-6} respectively while the last lines have values from the sets~\ref{tab-16}, \ref{tab-14} and \ref{tab-18}.}
\label{tab-odd2}
\end{table}

\newpage
\section{Disscussion}
\label{sec-disc}
In this final section, we put some of our results into a wider context.
Many results on the topology of tiling spaces have been obtained for
substitution (or inflation) tilings, for which powerful methods are
availabale \cite{And.Put.98,Sad.08}. The projection tilings considered
here are generically not substitutive, but many interesting examples are.
We sketch here briefly which projection tilings are substitutive,
and then derive some simple consequences.

Substitution tilings have a self-similarity built in. There exists a
factor $\lambda>0$, such that if the tiling is scaled by this factor,
its inflated tiles can be disected into tiles of the original size in
such a way, that the original tiling and the new, inflated and disected
tiling are locally indistinguishable (have the same local patterns).

For such an inflation procedure to exist for a projection tiling, two
conditions must be satisfied. Firstly, the scaling action by $\lambda$ must
lift to an automorphism of the lattice $\mathbb{Z}^n \cap (E \oplus F^\perp)$,
which then maps isomorphically to an automorphism of the projected lattice
$\Delta_0$. The latter extends continuously to a linear map of $F^\perp$.
This first condition depends only on the geometry of the lattice.
For our dodecagonal tilings, the scaling factor must satisfy
$\lambda = (2+\sqrt3)^k$ for some $k\in\mathbb{N}$, where
$2+\sqrt3$ is the fundamental real Pisot unit of infinite order
in the ring of 12-cyclotomic integers. The corresponding map on
$F^\perp$ then scales by its algebraic conjugate,
$\lambda^{-1} = (2-\sqrt3)^k$. $\Delta_0$ is invariant under any
scaling with an integer power of $2+\sqrt3$.

The second condition on the existence of an inflation requests, that
the pattern of singular lines $\{I_1^\alpha\}_{\alpha\in I_1}$ in $F^\perp$
is invariant under the scaling by $\lambda^{-1}$. This is a condition
on the values of the parameters $\gamma$, which is met if and only if
both parameters $\gamma_i$ are rational points in $G = \mathbb{Z}[\sqrt3]$,
that is, rational linear combinations of $1$ and $\sqrt3$. Then, and
only then, the $\Delta_0$-orbits of singular lines form \emph{finite}
orbits under the scaling with $2-\sqrt3$, so that some $k\in\mathbb{N}$
exists such that the whole pattern is invariant under a scaling by
$(2-\sqrt3)^k$.

The rational $\gamma$-parameters leading to substitutive tilings form only
a null set within all $\gamma$-parameters, but it's a null set containing
the most interesting cases. For instance, all tilings considered
in Section~\ref{2lines} fall into this class.

The above arguments let us conclude on the \emph{existence} of an inflation
procedure, which in practice may be far too complicated to be useful, however.
For instance, it might require a huge inflation factor. Still, the mere
existence of an inflation allows us to harvest some low-hanging fruit.
For substitution tilings, there is an associated action of the inflation
on the cohomology, whose eigenvalue spectrum is valuable information for
certain applications. For projection tilings, the cohomology groups are
determined by the singular lines and points, and the stabilizer groups
$\Gamma^\alpha$ acting on them, and there is a known action of the
inflation on these objects. This allows us to obtain at least partial
information on the spectrum of this action.

The first cohomology group is given by \cite{Gah.Hun.Kell.13}
\[
H^1 = \coker (\oplus_{\alpha\in I_1} \Lambda^3\Gamma^{\alpha} \rightarrow \Lambda^3\Delta_0)
       \oplus \ker (\oplus_{\alpha\in I_1} \Lambda^2\Gamma^\alpha
                          \rightarrow \Lambda^{2}\Delta_0).
\]
As the rank of $\Gamma^\alpha$ is $2$, the first summand is 
$\Lambda^3\Delta_0$, on which the inflation acts with eigenvalues
$\lambda$ and $\lambda^{-1}$ (two copies each). The second summand is
the kernel of our map $\beta$ (\ref{beta-map}), which has rank $3$.
The source of $\beta$
consists of summands $\Lambda^2\Gamma^\alpha$, on each of which the inflation
acts with eigenvalue $1$ ($\Gamma^\alpha$ is invariant under the scaling with
$\lambda$). Hence, the inflation acts on $H^1$ with two eigenvalues
$\lambda$, two eigenvalues $\lambda^{-1},$ and $|I_1|-3$ eigenvalues $1$.

The eigenvalues on $H^1$ have been used by Clark and Sadun \cite{CS.06}
to determine which changes of the tile shapes lead to changes of the
spectrum of the tiling dynamical system, and which shape changes leave
the spectrum invariant. Shape changes associated with cohomology classes
with eigenvalue less than one in modulus leave the spectrum invariant
\cite{CS.06}. For practical applications, however, one probably needs a
concrete representation of the inflation of the tiling.

In a similar way, the second cohomology group is given by \cite{Gah.Hun.Kell.13}
$H^2 = \coker\beta \oplus \ker \beta_0$. From the analysis above we
conclude, that the inflation acts on $\coker\beta$ with the three
eigenvalues $\lambda^2$, $\lambda^{-2}$, and $1$. $\ker\beta_0$
actually consists of two parts, $\ker\beta_0'\oplus\ker\beta_0''$.
The first of these is given by
\[
  \ker (\oplus_{\alpha \in I_1} \Lambda^1\Gamma^\alpha
  \rightarrow \Lambda^1 \Delta_0),
\]
carrying an inflation action with eigenvalues $\lambda$ and $\lambda^{-1}$
($|I_1| - 2$ times each). The second part, $\ker\beta_0''$, involves
the ($\Delta_0$-orbits of) singular points, on which the inflation acts
by permutaion. All eigenvalues of a permutation action are roots of unity.
Summarizing, the inflation acts on $H^2$ with eigenvalues $\lambda^2$,
$\lambda^{-2}$, $\lambda$ ($|I_1| - 2$ times) and $\lambda^{-1}$ ($|I_1| - 2$
times). All remaining eigenvalues have modulus $1$.

The subleading eigenvalues on $H^2$ determine the intrinsic fluctuations
in the distribution of a given local pattern in a tiling \cite{Schmi.Trev.18}.
When averaging the number of such patterns over a finite volume, two kinds
of errors are made. There is an error due to the boundary of the region
over which the average is taken, and there is an error due to the intrinsic
fluctuations in the distribution. Our analysis shows, that for dodecagonal
projection tilings which are also subsitutive, the error term due to intrinsic
fluctuations is of the same order of magnitude as the error due to the
boundary terms.

Homeomorphic tilings spaces have isomorphic cohomology, but the
converse is not generally true: cohomology is far from being a
complete invariant.  Still, there are sets of $\gamma$-parameters,
whose tilings spaces are known to be homeomorphic. We indicate here
some examples, but a complete analysis is beyond the scope of this
paper.  The simplest example is obtained, when the two $\gamma$-parameters
are swapped: the tiling spaces with parameters $(\gamma_1,\gamma_2)$ and
$(\gamma_2,\gamma_1)$ are homeomorphic. In fact, the two spaces are
mirror images of each other, where the mirror swaps the singular lines in
even and odd directions, respectively. More generally, if a pattern of
singular lines with parameters $\gamma'$ is an affine image of a pattern of
singular lines with parameters $\gamma$, then the corresponding tiling
spaces are homeomorphic. Similarly, if a pattern of singular lines is
invariant under a scaling with the $k$-th power of $2-\sqrt3$, then
scaling the pattern with $2-\sqrt3$ creates a finite orbit of patterns,
whose tiling spaces are homeomorphic. Such spaces are not invariant
under an inflation with scaling factor $2+\sqrt3$, but form an orbit of
homeomorphic spaces under such an inflation.

\section*{Acknowledgements}
We want to thank the referee for the careful reading of the manuscript and the very interesting questions posed. We thank Thomas Fernique for generously providing the pictures in
Figure~\ref{fig:patches}.  F.G. was supported by the German Research Foundation (DFG), within the CRC 1283. A.G.L.\ was supported by the EPSRC grant EP/T028408/1.

\clearpage
\bibliographystyle{plain}
\bibliography{biblio-cohom-bl}

\end{document}